\documentclass[11pt]{amsart}
\usepackage{fullpage}
\usepackage{graphicx}
\usepackage{subfigure}
\usepackage{psfrag}
\usepackage{color}

\newtheorem{theorem}{Theorem}[section]
\newtheorem{remark}[theorem]{Remark}

\title{Non-negative mixed finite element formulations for a tensorial diffusion equation}

\author{K.~B.~Nakshatrala} \author{A.~J.~Valocchi}
\address{Correspondence to: Professor Kalyana Babu Naskshatrala, Department of Mechanical Engineering, 
  216 Engineering/Physics Building, Texas A\&M University, College Station, Texas - 77843. Phone: (+1)-979-845-1292.}
\email{knakshatrala@tamu.edu}
\address{Professor Albert Valocchi, Department of Civil and Environmental Engineering, 
  1110 Newmark Laboratory, University of Illinois at Urbana-Champaign, Illinois - 61801.}
\email{valocchi@uiuc.edu}
\begin{document}

\begin{abstract}
  We consider the tensorial diffusion equation, and address the discrete maximum-minimum principle 
  of mixed finite element formulations. In particular, we address non-negative solutions (which 
  is a special case of the maximum-minimum principle) of mixed finite element formulations. It 
  is well-known that the classical finite element formulations (like the single-field Galerkin 
  formulation, and Raviart-Thomas, variational multiscale, and Galerkin/least-squares mixed 
  formulations) do not produce non-negative solutions (that is, they do not satisfy the discrete 
  maximum-minimum principle) on arbitrary meshes and for strongly anisotropic diffusivity 
  coefficients. 

  In this paper we present two non-negative mixed finite element formulations for tensorial 
  diffusion equations based on constrained optimization techniques. These proposed mixed 
  formulations produce non-negative numerical solutions on arbitrary meshes for low-order 
  (i.e., linear, bilinear and trilinear) finite elements. The first formulation is based 
  on the Raviart-Thomas spaces, and the second non-negative formulation is based on the 
  variational multiscale formulation. For the former formulation we comment on the effect 
  of adding the non-negative constraint on the local mass balance property of the 
  Raviart-Thomas formulation. 

  We perform numerical convergence analysis of the proposed optimization-based non-negative 
  mixed formulations. We also study the performance of the active set strategy for solving 
  the resulting constrained optimization problems. The overall performance of the proposed 
  formulation is illustrated on three canonical test problems. 
\end{abstract}


  \maketitle


\section{INTRODUCTION}
\label{Sec:Optim_Intro}
Robustness of numerical methods for flow and transport problems in porous media is 
important for development of simulators to be used in a wide range of applications 
in subsurface hydrology and contaminant transport. In order to obtain robust and 
reliable numerical results it is imperative to preserve basic properties of solutions 
of mathematical models by computed approximations. In the simulation of reactive transport 
of contaminants one such basic properties is non-negative solutions as concentration 
of a chemical or biological species physically can never be negative. Since the domain 
of interest in subsurface flows is highly complex, one needs to employ unstructured 
computational grids. Therefore, obtaining non-negative solutions on unstructured meshes 
is an essential feature in the simulation of reactive transport of contaminants, as well 
as many other physical processes. 

However, obtaining non-negative solutions on unstructured grids using a numerical method 
(finite element, finite volume or finite difference) is not an easy task. In addition, 
there is another complexity arising from an anisotropic diffusion tensor. In subsurface 
flows, the heterogeneity in the velocity field will give rise to a non-homogeneous 
anisotropic diffusion tensor with non-negligible cross terms \cite{Pinder_Celia}. 
Several studies have shown that standard treatment of the cross-diffusion term 
will result in negative solutions on general computational grids, see 
\cite{Herrera_Valocchi_GW_2006_v44_p803} and references therein. Several ad-hoc 
procedures have been proposed in the literature. For example, a post processing 
step is typically employed in which one performs some sort of ``smoothing.'' But 
this procedure, in many cases, is not variationally consistent. Some other methods 
are limited in their range of applicability (e.g., the method proposed in Reference 
\cite{Herrera_Valocchi_GW_2006_v44_p803} can handle only structured grids). 

Herein we consider the dispersion/diffusion process for steady single-phase flow in 
heterogeneous anisotropic porous media. Such a flow can be described by the Poisson's 
equation with a tensorial diffusion coefficient, which when written in the mixed form 
is similar to the governing equations of Darcy flow. In this paper we propose two 
optimization-based mixed methods for solving tensorial diffusion equation that gives 
non-negative solutions on general grids for linear finite elements. The two methods 
are developed by rewriting the Raviart-Thomas and variational multiscale formulations 
as constrained minimization problems subject to a constraint on the primary variable 
to be non-negative. A similar approach based on optimization techniques has been 
used by Liska and Shashkov \cite{Liska_Shashkov_CiCP_2008_v3_p852} for a single 
field formulation, but herein we consider mixed finite element formulations. 

\emph{
The main idea behind the proposed methods is to augment a constraint on nodal values to 
be non-negative to the discrete (that is, after spatial finite element discretization) 
variational statement of the underlying formulation. Since we consider only low-order 
finite elements (and since the shape functions for these elements do not change their 
sign within an element), non-negative nodal values ensure non-negative solution every 
where in the element, and hence non-negative solution on the whole domain. This argument 
will not hold for high-order finite elements as shape functions for these finite elements 
(in general) change sign within an element.} 

Throughout this paper continuum vectors are denoted with lower case boldface normal letters, 
and (continuum) second-order tensors will be denoted using (\LaTeX) blackboard font (for 
example, $\mathbf{v}$ and $\mathbb{D}$, respectively).  We denote finite element vectors 
and matrices with lower and upper case boldface italic letters, respectively. For example, 
vector $\boldsymbol{v}$ and matrix $\boldsymbol{K}$. 
The curled inequality symbols $\succeq$ and $\preceq$ are used to denote generalized 
inequalities between vectors, which represent component-wise inequalities. That is, 
given two vectors $\boldsymbol{a}$ and $\boldsymbol{b}$, $\boldsymbol{a} \succeq 
\boldsymbol{b}$ means $a_i \geq b_i \; \forall i$. A similar definition holds for 
the symbol $\preceq$. Other notational conventions adopted in this paper are introduced 
as needed.

\subsection{Governing equations}
\label{Subsec:Optim_Governing_equations}
Let $\Omega \subseteq \mathbb{R}^{nd}$ (where ``$nd$" is the number of 
spatial dimensions) be a bounded domain with boundary $\partial \Omega 
= \bar{\Omega} \setminus \Omega$, where $\bar{\Omega}$ denotes the
closure of $\Omega$. Consider the diffusion of a chemical species 
in anisotropic heterogeneous medium, which is governed by the 
second-order elliptic tensorial diffusion partial differential 
equation. The governing equations are 
\begin{align}
  \label{Eqn:Optim_Poisson}
  - &\nabla \cdot \left(\mathbb{D}(\mathbf{x}) \nabla c(\mathbf{x}) \right) = 
  f(\mathbf{x}) \quad \mathrm{in} \; \Omega \\
  \label{Eqn:Optim_Neumann_BC}
  -&\mathbf{n}(\mathbf{x}) \cdot \mathbb{D} (\mathbf{x}) \nabla c = t^{\mathrm{p}}(\mathbf{x}) \quad \mathrm{on} 
  \; \Gamma^{\mathrm{N}} \\
  \label{Eqn:Optim_Dirichlet_BC}
  &c(\mathbf{x}) = c^{\mathrm{p}}(\mathbf{x}) \quad \mathrm{on} \; \Gamma^{\mathrm{D}} 
\end{align}
where $c(\mathbf{x})$ denotes the concentration field, $f(\mathbf{x})$ is the volumetric 
source, $t^{\mathrm{p}}(\mathbf{x})$ is the prescribed flux (i.e., Neumann boundary condition), 
$c^{\mathrm{p}}(\mathbf{x})$ is the prescribed concentration (i.e., Dirichlet boundary condition), 
$\Gamma^{\mathrm{D}}$ is that part of the boundary on which Dirichlet boundary condition is applied, 
$\Gamma^{\mathrm{N}}$ is the part of the boundary on which Neumann boundary condition is applied, 
$\mathbf{n}(\mathbf{x})$ is unit outward normal to the boundary, and $\nabla$ denotes gradient 
operator. For well-posedness one requires $\Gamma^{\mathrm{D}} \cup \Gamma^{\mathrm{N}} = \partial 
\Omega$ and $\Gamma^{\mathrm{D}} \cap \Gamma^{\mathrm{N}} = \emptyset$, and for uniqueness $\Gamma^{\mathrm{D}} 
\neq \emptyset$. We assume that the coefficient of diffusivity $\mathbb{D}(\mathbf{x})$ is a symmetric positive 
definite tensor such that, for some $0 < \alpha_1 \leq \alpha_2 < +\infty$, we have 
\begin{align}
  \alpha_1 \mathbf{y}^{T} \mathbf{y} \leq 
  \mathbf{y}^{T} \mathbb{D}(\mathbf{x}) \mathbf{y} \leq 
  \alpha_2 \mathbf{y}^{T} \mathbf{y} \quad \forall \mathbf{x} \in \Omega, 
  \; \forall \mathbf{y} \neq \mathbf{0} \in \mathbb{R}^{nd} 
\end{align}
In addition, we assume that $\mathbb{D}(\mathbf{x})$ is continuously differentiable.

\subsection{First-order (or mixed) form}
In many situations, the primary quantity of interest is the flux. But a single field (or primal) 
formulation does not produce accurate solutions for the flux. One can calculate the flux by 
differentiating the obtained $c(\mathbf{x})$, but there will be a loss of accuracy 
during this process. For example, under a single field formulation, linear finite elements 
produce fluxes that are constant and discontinuous across elements. This means that there 
is no flux balance across element edges. Balance of flux along element edges is a highly 
desirable feature and is of physical importance in many practical engineering problems. 
In order to alleviate aforementioned drawbacks of single formulations, mixed formulations 
are often employed. Equations \eqref{Eqn:Optim_Poisson}-\eqref{Eqn:Optim_Dirichlet_BC} in 
mixed (or first-order) form can be written as 
\begin{align}
  \label{Eqn:Optim_mixed_form_Darcy}
  & \mathbb{D}^{-1}(\mathbf{x}) \mathbf{v}(\mathbf{x}) = - \nabla c \quad \mathrm{in} \; \Omega \\
  & \nabla \cdot \mathbf{v} = f(\mathbf{x}) \quad \mathrm{in} \; \Omega \\
  & \mathbf{v}(\mathbf{x}) \cdot \mathbf{n}(\mathbf{x}) = t^{\mathrm{p}}(\mathbf{x}) \quad \mathrm{on} \; \Gamma^{\mathrm{N}} \\
  \label{Eqn:Optim_mixed_form_Dirichlet}
  & c(\mathbf{x}) = c^{\mathrm{p}}(\mathbf{x}) \quad \mathrm{on} \; \Gamma^{\mathrm{D}} 
\end{align}
where $\mathbf{v}(\mathbf{x})$ is an auxiliary variable, which can be interpreted 
as follows: given a plane defined by a normal $\mathbf{n}$, the quantity $\mathbf{v} 
\cdot \mathbf{n}$ will be the flux through the plane. 

\subsection{Maximum-minimum principle}
It is well-known that some elliptic partial differential equations (under appropriate regularity 
assumptions) satisfy the so-called maximum-minimum principle, and the Poisson's equation is one of 
them \cite{McOwen}. We now state the classical maximum-minimum principle for second-order elliptic 
partial differential equations. (In Section \ref{Sec:Optim_VMS} we state and prove a maximum-minimum 
principle under milder regularity assumptions. Note that weak solutions \emph{may} also possess a 
maximum-minimum principle. For example, see Reference \cite{Gilbarg_Trudinger}.)
Consider the boundary value problem given by equations \eqref{Eqn:Optim_Poisson}-\eqref{Eqn:Optim_Dirichlet_BC}. 
Let $c(\mathbf{x}) \in C^{2}(\Omega) \cap C^{0}(\bar\Omega)$, where $C^{2}(\Omega)$ denotes the 
set of twice continuously differentiable functions defined on $\Omega$, and $C^{0}(\bar \Omega)$ 
the set of uniformly continuous functions defined on $\Omega$. If $f(\mathbf{x}) \geq 0$ (or 
if $f(\mathbf{x}) \leq 0$) in $\Omega$ then $c(\mathbf{x})$ attains its minimum (or its 
maximum) on the boundary of $\Omega$. For a detailed discussion on
maximum-minimum principles see References \cite{McOwen,Evans_PDE,Han_Lin,Gilbarg_Trudinger}. 
 
One of the important consequences of maximum-minimum principles is the non-negative 
solution of a (tensorial) diffusion equation under non-negative forcing function with 
non-negative prescribed Dirichlet boundary condition. Obtaining non-negative solutions 
is of paramount importance in studying transport of chemical and biological species as 
negative concentration of a species is unphysical. 

\subsection{Discrete maximum-minimum principle}
The discrete analogy of the maximum-minimum principle is commonly referred to as the \emph{discrete 
maximum-minimum principle} (DMP). However, many numerical formulations do not \emph{unconditionally} 
satisfy the discrete maximum-minimum principle. Typically, there will be restrictions on the mesh or 
on the magnitude of coefficients of the diffusivity tensor. For example, the single field Galerkin 
formulation in the case of scalar diffusion satisfies the discrete maximum-minimum principle if the 
mesh satisfies weak acute condition \cite{Ciarlet_Raviart_CMAME_1973_v2_p17} (and also see Appendix). 
The question whether we get non-negative \emph{numerical} solutions leads us to the discrete maximum-minimum 
principle.

For recent works on DMP see References \cite{Burman_Ern_CMAME_2002_v191_p3833,
Burman_Ern_MathComput_2005_v74_p1637,Hoteit_Mose_Philippe_Ackerer_Erhel_IJNME_2002_v55_p1373,
Karatson_Korotov_NumerMath_2005_v99_p669,Karatson_Korotov_JCAM_2006_v192_p75,Krizek_Liu_CMAME_1998_v157_p387,
Krizek_Liu_ZAMM_2003_v83_p559} and also see the discussion in Reference \cite[Introduction]{Liska_Shashkov_CiCP_2008_v3_p852}. 
Considerable attention to DMP has also been given in the finite volume literature \cite{Herrera_Valocchi_GW_2006_v44_p803,
LePotier_CRM_2005_v341_p787,Lipnikov_Shashkov_Svyatskiy_JCP_2006_v211_p473,Nordbotten_Aavatsmark_Eigestad_NumerMath_2007_v106_p255}.
Optimization-based techniques have been employed in References \cite{Liska_Shashkov_CiCP_2008_v3_p852} and 
\cite{Mlacnik_Durlofsky_JCP_2006_v216_p337} to address DMP. For completeness, some of the classical results 
on discrete maximum-minimum principle are outlined in Appendix.

In this paper we concentrate on obtaining non-negative numerical solutions using mixed formulations 
under non-negative forcing function with non-negative Dirichlet boundary condition where ever it is 
prescribed. (Note that we do not assume that the Dirichlet boundary condition has to be prescribed 
on the whole boundary.) In all our test problems (see Section \ref{Sec:Optim_numerical}) we have 
$f(\mathbf{x}) \geq 0$ in $\Omega$ and $c^{\mathrm{p}} (\mathbf{x}) \geq 0$ on $\partial \Omega$. 
By using the maximum-minimum principle one can conclude that $c(\mathbf{x}) \geq 0$ in whole of 
$\bar \Omega$ (the closure of $\Omega$). That is, for the chosen test problems in Section 
\ref{Sec:Optim_numerical}, we must have non-negative solutions in the whole domain. 

\subsection{Main contributions of this paper} 
Some of the main contributions of this paper are as follows:
\begin{itemize}
\item We numerically demonstrate that various conditions outlined in Appendix (which are sufficient 
  for isotropic diffusion) are not sufficient for tensorial diffusion equation under the Raviart-Thomas 
  and variational multiscale formulations to produce non-negative solutions under non-negative 
  forcing functions with non-negative prescribed Dirichlet boundary conditions.   
\item We develop a non-negative formulation based on the lowest-order Raviart-Thomas 
  spaces, and discuss the consequences of obtaining non-negative solutions on the 
  local mass balance property. 
\item We extend the variational multiscale formulation to produce non-negative solutions 
  on general grids for low-order finite elements under non-negative forcing function and 
  prescribed non-negative Dirichlet boundary condition. We also show that the (continuous) 
  variational multiscale formulation satisfies a continuous maximum-minimum principle (under 
  appropriate regularity assumptions).
\end{itemize}

\subsection{Organization of the paper}
The remainder of this paper is organized as follows. In Section \ref{Sec:Optim_RT} 
we present a non-negative formulation based on the lowest-order Raviart-Thomas (RT0) 
finite element spaces, which is achieved by adding a non-negative constraint to the 
discrete variational setting of the Raviart-Thomas formulation. For this non-negative 
formulation we present both primal and dual constrained optimization problems, and 
comment on the ease of solving these problems and also the consequences of imposing 
the non-negative constraint on the local mass balance. In Section \ref{Sec:Optim_VMS}, 
a non-negative formulation based on the variational multiscale formulation will be 
presented. We also show that the (continuous) variational multiscale formulation 
satisfies a continuous maximum-minimum principle. Numerical results along with a 
discussion on the numerical performance of both the proposed non-negative formulations 
will be presented in Section \ref{Sec:Optim_numerical}. Finally, conclusions are drawn 
in Section \ref{Sec:Optim_conclusions}.

\section{A NON-NEGATIVE MIXED FORMULATION BASED ON  RAVIART-THOMAS SPACES}
\label{Sec:Optim_RT}
The Raviart-Thomas finite element formulation is widely used (for an example in subsurface 
modeling see \cite{Chen_Huan_Ma}) to solve diffusion equations in mixed form, and is based 
on the classical mixed formulation \cite{Raviart_Thomas_MAFEM_1977_p292}. The simplest and 
lowest order Raviart-Thomas space (commonly denoted as RT0) consists of fluxes evaluated 
on the midpoints of edges and constant pressure over elements. We first present the weak 
form and variational structure behind the Raviart-Thomas formulation. We then modify the 
variational structure by adding non-negative constraint on the concentration to build a 
non-negative low-order finite element formulation based on the RT0 spaces.  

To this end, define function spaces as
\begin{align}
  \label{Eqn:Optim_V_space}
  \mathcal{V} &:= \left\{\mathbf{v}(\mathbf{x}) \; | \; \mathbf{v}(\mathbf{x}) 
    \in \left(L^{2} (\Omega)\right)^{nd}, \; \nabla \cdot \mathbf{v} \in L^{2}(\Omega), \; 
    \mathbf{v}(\mathbf{x}) \cdot \mathbf{n}(\mathbf{x}) = t^{\mathrm{p}}(\mathbf{x}) 
    \; \mathrm{on} \; \Gamma^{\mathrm{N}} \right\}\\
  \label{Eqn:Optim_W_space}
  \mathcal{W} &:= \left\{\mathbf{v}(\mathbf{x}) \; | \; \mathbf{v}(\mathbf{x}) 
    \in \left(L^{2} (\Omega)\right)^{nd}, \; \nabla \cdot \mathbf{v} \in L^{2}(\Omega), \; 
    \mathbf{v}(\mathbf{x}) \cdot \mathbf{n}(\mathbf{x}) = 0 \; \mathrm{on} \; \Gamma^{\mathrm{N}} \right\}\\
  \label{Eqn:Optim_P_space}
  \mathcal{P} &:= L^{2} (\Omega)
\end{align}
Recall that ``$nd$'' denotes the number of spatial dimensions. For further details 
on function spaces see the monograph by Brezzi and Fortin \cite{Brezzi_Fortin}. Let 
$\mathbf{w}(\mathbf{x})$ and $q(\mathbf{x})$ denote the weighting functions 
corresponding to $\mathbf{v}(\mathbf{x})$ and $c(\mathbf{x})$, respectively. 
The classical mixed formulation for equations 
\eqref{Eqn:Optim_mixed_form_Darcy}-\eqref{Eqn:Optim_mixed_form_Dirichlet} can be written 
as: Find $\mathbf{v}(\mathbf{x}) \in \mathcal{V}$ and $c(\mathbf{x}) \in 
\mathcal{P}$ such that 
\begin{align}
  \label{Eqn:Optim_classical_mixed_formulation}
\left(\mathbf{w}; \mathbb{D}^{-1}\mathbf{v}\right) - \left(\nabla \cdot \mathbf{w}; c\right) 
+ \left(\mathbf{w} \cdot \mathbf{n}; c^{\mathrm{p}}\right)_{\Gamma^{\mathrm{D}}} - 
\left(q; \nabla \cdot \mathbf{v} - f\right) = 0 \quad 
\forall \mathbf{w}(\mathbf{x}) \in \mathcal{W}, \; q(\mathbf{x}) \in \mathcal{P}
\end{align}
It is well-known that, under appropriate smoothness conditions on the domain and its 
boundary, the above saddle-point formulation is well-posed \cite{Brezzi_Fortin}. That 
is, a unique (weak) solution exists for this problem that depends continuously on the 
input data. However, to obtain stable results using a finite element approximation, the 
finite dimensional spaces $\mathcal{V}^{h} \subset \mathcal{V}$ and $\mathcal{P}^{h} 
\subset \mathcal{P}$ in which a numerical solution is sought have to satisfy the 
Ladyzhenskaya-Babu\v ska-Brezzi (LBB) stability condition \cite{Brezzi_Fortin}. 
One such space that satisfies the LBB condition is the popular Raviart-Thomas (RT) finite 
element space. In this paper we consider only the lowest order Raviart-Thomas triangular 
finite element space (RT0). Let $\mathcal{T}_{h}$ be a triangulation on $\Omega$. The 
lowest order Raviart-Thomas finite dimensional subspaces on triangles are defined as 
\begin{align}
& \mathcal{P}_h := \left\{p \; | \; p = \mbox{a constant on each triangle} \; K \in \mathcal{T}_h \right\}\\
& \mathcal{V}_h := \left\{\mathbf{v} = (v^{(1)},v^{(2)}) \; | \; v^{(1)}_K = a_K + b_K x, \; 
  v^{(2)}_K = c_K + b_K y; a_K, b_K, c_K \in \mathbb{R}; K \in \mathcal{T}_h \right\}
\end{align}

\subsection{Discrete equations}
The discretized finite element equations of the Raviart-Thomas formulation for the 
mixed form of tensorial diffusion equation can be written as \cite{Chen_Huan_Ma}
\begin{align}
  \label{Eqn:Optim_RT0_discretized_FEM}
  \left[\begin{array}{cc}
      \boldsymbol{K}_{vv} & \boldsymbol{K}_{pv}^{T} \\
      \boldsymbol{K}_{pv} & \boldsymbol{O} 
    \end{array}\right] 
  \left\{\begin{array}{c} 
      \boldsymbol{v} \\
      \boldsymbol{p} 
    \end{array}\right\}
  = \left\{\begin{array}{c}
      \boldsymbol{f}_{v} \\ 
      \boldsymbol{f}_{p}
    \end{array}\right\}
\end{align}
where $\boldsymbol{O}$ is a zero matrix of appropriate size, the matrix $\boldsymbol{K}_{vv}$ 
is symmetric and positive definite, $\boldsymbol{v}$ denotes the (finite element) vector 
of flux degrees-of-freedom, and $\boldsymbol{p}$ denotes the vector of concentration 
degrees-of-freedom. 
Comparing the weak form \eqref{Eqn:Optim_classical_mixed_formulation} and discrete equation 
\eqref{Eqn:Optim_RT0_discretized_FEM}, the matrices $\boldsymbol{K}_{vv}$ and $\boldsymbol{K}_{pv}$ are 
obtained after the finite element discretization of the terms $(\mathbf{w};\mathbb{D}^{-1} \mathbf{v})$ 
and $-(q;\nabla \cdot \mathbf{v})$, respectively. Since equation \eqref{Eqn:Optim_classical_mixed_formulation} 
is written in symmetric form, $\boldsymbol{K}_{vp}$ (which comes from the term $-(\nabla \cdot \mathbf{w};c)$) 
will be equal to $\boldsymbol{K}_{pv}^{T}$. The vectors $\boldsymbol{f}_{v}$ and $\boldsymbol{f}_p$ are, 
respectively, obtained from $-(\mathbf{w} \cdot \mathbf{n};c^{\mathrm{p}})_{\Gamma^{\mathrm{D}}}$ and $-(q;f)$ 
after the finite element discretization. Since there is no term in equation \eqref{Eqn:Optim_classical_mixed_formulation} 
that contains both $q$ in the weighting (i.e., first) slot and $c$ in the second slot in a bilinear form $(\cdot;\cdot)$, 
we have the matrix $\boldsymbol{K}_{pp} = \boldsymbol{O}$ (a zero matrix).

The above system of equations \eqref{Eqn:Optim_RT0_discretized_FEM} is equivalent to the 
following constrained minimization problem
\begin{align}
  \label{Eqn:Optim_RT0_primal_problem}
  \mbox{(P1-RT0)} \quad
  \left\{
  \begin{array}{cc} 
    \mathop{\mathrm{minimize}}_{\boldsymbol{v}} &
    \frac{1}{2} \boldsymbol{v}^{T} \boldsymbol{K}_{vv} \boldsymbol{v} - 
    \boldsymbol{v}^{T} \boldsymbol{f}_{v} \\
    \mbox{subject to} & \boldsymbol{K}_{pv} \boldsymbol{v} - \boldsymbol{f}_{p} = \boldsymbol{0} \\
  \end{array} \right.
\end{align}
where $\boldsymbol{0}$ is a zero vector of appropriate size. Note that the constraint in equation 
\eqref{Eqn:Optim_RT0_primal_problem} is the local mass balance condition for each element. We refer 
the above equation as the primal problem for the Raviart-Thomas formulation, and denote it as (P1-RT0). 
This primal problem belongs to the class of \emph{convex quadratic programming} problems, and from 
optimization theory (for example, see Reference \cite{Boyd_convex_optimization}) it can be shown that 
the problem has a unique global minimizer. 

\begin{remark}
  A quadratic program is an optimization problem in which the objective function is a 
  quadratic function and the (equality and inequality) constraints are all linear. In a 
  convex quadratic program the Hessian of the objective function is positive semidefinite. 
\end{remark}

\begin{remark} 
It is interesting to note that, from the complexity theory, problem \eqref{Eqn:Optim_RT0_primal_problem} 
can be solved in polynomial time (for example, using the ellipsoid and interior point methods) 
\cite{Boyd_convex_optimization,Nocedal_Wright}. Note that the term ``polynomial time'' in the context 
of complexity theory should not be confused with the term ``polynomial convergence,'' which is commonly 
used in the convergence studies using the finite element method.
\end{remark}

Define the Lagrangian as 
\begin{align}
  \label{Eqn:Optim_Lagrangian}
  \mathcal{L}\left(\boldsymbol{v},\boldsymbol{p}\right) := 
  \frac{1}{2} \boldsymbol{v}^{T} \boldsymbol{K}_{vv} \boldsymbol{v} - 
  \boldsymbol{v}^{T} \boldsymbol{f}_{v} + 
  \boldsymbol{p}^{T} \left(\boldsymbol{K}_{pv} \boldsymbol{v} - \boldsymbol{f}_{p}\right)
\end{align}
where $\boldsymbol{p}$ is the vector of Lagrange multipliers. Using the Lagrange multiplier 
method \cite{Boyd_convex_optimization} the primal problem \eqref{Eqn:Optim_RT0_primal_problem} 
is equivalent to 
\begin{align}
  \label{Eqn:Optim_Lagrange_multiplier_method}
  \mathop{\mbox{extremize}}_{\boldsymbol{v}, \; \boldsymbol{p}} \; \mathcal{L}(\boldsymbol{v},\boldsymbol{p})
\end{align}
and the first-order optimality conditions for this problem gives rise to the discretized finite 
element equations \eqref{Eqn:Optim_RT0_discretized_FEM}. 
We now write the dual problem corresponding to the primal problem \eqref{Eqn:Optim_RT0_primal_problem}. 
To this end, define the Lagrange dual function as 
\begin{align}
  \label{Eqn:Optim_Lagrange_dual_function}
  g(\boldsymbol{p}) := \inf_{\boldsymbol{v}} \; \mathcal{L} (\boldsymbol{v},\boldsymbol{p}) = 
  -\frac{1}{2} \boldsymbol{p}^{T} \boldsymbol{K}_{pv} \boldsymbol{K}_{vv}^{-1} \boldsymbol{K}_{pv}^{T} \; \boldsymbol{p} 
  +\boldsymbol{p}^{T} \left(\boldsymbol{K}_{pv} \boldsymbol{K}_{vv}^{-1} \boldsymbol{f}_v - \boldsymbol{f}_p\right)
  -\frac{1}{2} \boldsymbol{f}_v^{T} \boldsymbol{K}_{vv}^{-1} \boldsymbol{f}_v
\end{align}
The above expression on the right-hand side is obtained as follows. Let $\boldsymbol{v}^{*}$ be 
the minimizer that gives the infimum of $\mathcal{L}(\boldsymbol{v},\boldsymbol{p})$ with respect 
to $\boldsymbol{v}$. Then $\boldsymbol{v}^{*}$ has to satisfy 
\begin{align}
\boldsymbol{K}_{vv} \boldsymbol{v}^{*}- \boldsymbol{f}_{v} + \boldsymbol{K}^{T}_{pv} \boldsymbol{p} = \boldsymbol{0}
\end{align}
which is a necessary condition, and is obtained by equating the derivative of $\mathcal{L}(\boldsymbol{v},
\boldsymbol{p})$ (which is defined in equation \eqref{Eqn:Optim_Lagrangian}) with respect to $\boldsymbol{v}$ 
to zero. Since the matrix $\boldsymbol{K}_{vv}$ is positive definite (and hence invertible) we have 
\begin{align}
\boldsymbol{v}^{*} = \boldsymbol{K}_{vv}^{-1} \left(\boldsymbol{f}_v - \boldsymbol{K}_{pv}^T \boldsymbol{p}\right)
\end{align}
By substituting the above expression for the minimizer $\boldsymbol{v}^{*}$ into the definition of 
$\mathcal{L}(\boldsymbol{v},\boldsymbol{p})$ \eqref{Eqn:Optim_Lagrangian} we obtain the expression 
on the right-hand side of equation \eqref{Eqn:Optim_Lagrange_dual_function}. 

The dual problem corresponding to the primal problem \eqref{Eqn:Optim_RT0_primal_problem} 
can then be written as 
\begin{align}
  \mathop{\mathrm{maximize}}_{\boldsymbol{p}} & \quad g(\boldsymbol{p}) 
\end{align}
which is \emph{equivalent} to 
\begin{align}
  \label{Eqn:Optim_dual_problem_simplified}
  \mbox{(D1-RT0)} \quad 
  \mathop{\mathrm{minimize}}_{\boldsymbol{p}} \quad 
  \frac{1}{2} \boldsymbol{p}^{T} \boldsymbol{K}_{pv} \boldsymbol{K}_{vv}^{-1} \boldsymbol{K}_{pv}^{T} \; \boldsymbol{p} 
  - \boldsymbol{p}^{T} \left(\boldsymbol{K}_{pv} \boldsymbol{K}_{vv}^{-1} \boldsymbol{f}_v - \boldsymbol{f}_{p} \right)
\end{align}
The stationarity of the above problem implies 
\begin{align}
  \boldsymbol{K}_{pv} \boldsymbol{K}_{vv}^{-1} \boldsymbol{K}_{pv}^{T} \; \boldsymbol{p} 
  = \boldsymbol{K}_{pv} \boldsymbol{K}_{vv}^{-1} \boldsymbol{f}_v - \boldsymbol{f}_{p}
\end{align}
which is the Schur complement form of equation \eqref{Eqn:Optim_RT0_discretized_FEM} expressed in terms 
of Lagrange multipliers by analytically eliminating the variable $\boldsymbol{v}$. Note that the Schur 
complement operator $\boldsymbol{K}_{pv} \boldsymbol{K}_{vv}^{-1} \boldsymbol{K}_{pv}^{T}$ is symmetric
 and positive definite. 

A simple numerical example to be presented later (e.g., see Figure \ref{Fig:Optim_RT0_Problem_2}) 
shows that RT0 triangular element does not satisfy the discrete maximum-minimum principle, and in 
particular, does not produce non-negative solutions for non-negative forcing functions with non-negative 
prescribed Dirichlet boundary conditions. 

\subsection{A non-negative mixed formulation}
\label{Subsec:Optim_RT0_non_negative}
In order to get non-negative solutions under the RT0 spaces, we pose the dual problem as 
\begin{align}
  \label{Eqn:Optim_dual_problem_DMP}
  \mbox{(D2-RT0)} \quad 
  \left\{ \begin{array}{ll}
      \mathop{\mathrm{minimize}}_{\boldsymbol{p}} & \quad 
      \frac{1}{2} \boldsymbol{p}^{T} \boldsymbol{K}_{pv} \boldsymbol{K}_{vv}^{-1} \boldsymbol{K}_{pv}^{T} \; \boldsymbol{p} 
      - \boldsymbol{p}^{T} \left(\boldsymbol{K}_{pv} \boldsymbol{K}_{vv}^{-1} \boldsymbol{f}_v - \boldsymbol{f}_{p}\right) \\
      \mbox{subject to} & \quad \boldsymbol{p} \succeq \boldsymbol{0}
    \end{array} \right.
\end{align}
Recall that the symbol $\succeq$ denotes the generalized inequality between vectors, which represents 
component-wise inequality (see Introduction, just above subsection \ref{Subsec:Optim_Governing_equations}, 
for a discussion on this notation). The primal problem corresponding to this new dual problem will then be 
\begin{align}
  \label{Eqn:Optim_new_primal}
  \mbox{(P2-RT0)} \quad 
  \left\{ \begin{array}{ll}
      \mathop{\mbox{minimize}}_{\boldsymbol{v}} & \quad \frac{1}{2} \boldsymbol{v}^{T} 
      \boldsymbol{K}_{vv} \boldsymbol{v} - \boldsymbol{v}^{T} \boldsymbol{f}_{v} \\
      %
      \mathop{\mbox{subject to}} & \quad \boldsymbol{K}_{pv} \boldsymbol{v} - \boldsymbol{f}_{p} \preceq \boldsymbol{0}
    \end{array} \right.
\end{align}

\begin{remark}
The primal problem given in equation \eqref{Eqn:Optim_new_primal} is obtained by inspection. That 
is, one can easily check (using a direct calculation) that the dual problem of this new primal problem 
\eqref{Eqn:Optim_new_primal} will be the same as equation \eqref{Eqn:Optim_dual_problem_DMP}. Also, it 
should be noted that one can write the dual problem corresponding to a given dual problem (that is, the 
dual of a dual). For the problem at hand, the dual of the dual problem will be the same as the primal 
problem, which is not the case in general \cite{Boyd_convex_optimization}. Hence, one will obtain the 
primal problem \eqref{Eqn:Optim_new_primal}  by writing the dual of the dual problem 
\eqref{Eqn:Optim_dual_problem_DMP}. 
\end{remark}

By comparing the constraints in equations \eqref{Eqn:Optim_RT0_primal_problem} and 
\eqref{Eqn:Optim_new_primal} one can conclude that under the proposed non-negative 
method based on the Raviart-Thomas formulation one may violate local mass balance by 
creating \emph{artificial sinks}. We can infer more on local mass balance by looking 
at the Karush-Kuhn-Tucker (KKT) conditions (which in this case are necessary and 
sufficient for the optimality) for the new primal problem given by equations 
\eqref{Eqn:Optim_new_primal}. The KKT optimality conditions for the (P2-RT0) 
problem are 
\begin{align}
  & \boldsymbol{K}_{vv} \boldsymbol{v} + \boldsymbol{K}_{pv}^{T} \; \boldsymbol{p} = \boldsymbol{f}_{v} \\
  & \boldsymbol{K}_{pv} \boldsymbol{v} - \boldsymbol{f}_{p} \preceq \boldsymbol{0} \\
  & \boldsymbol{p} \succeq \boldsymbol{0} \\
  & p_{i} (\boldsymbol{K}_{pv} \boldsymbol{v} - \boldsymbol{f}_{p})_{i} = 0 \quad \forall i
\end{align}
The last condition (which is basically the complementary slackness condition in the KKT system of equations) 
implies that one \emph{may} not have local mass balance in those elements for which the Lagrange multiplier 
vanishes (i.e., $p_i = 0$, where $i$ denotes the element number). Note that in the RT0 formulation, the 
Lagrange multiplier $p_i$ denotes the concentration in the $i^{\mathrm{th}}$ element. 

\begin{remark}
From optimization theory \cite{Boyd_convex_optimization} one can show that the primal 
(P2-RT0) and dual (D2-RT0) problems are equivalent. That is, there is no duality gap 
for the optimization-based RT0 formulation. The difference between primal and dual 
solutions is commonly referred to as the duality gap. In general, the solution of a 
dual problem gives an upper bound to its corresponding primal problem. 

However, from a computational point of view the primal and dual problems can have different 
numerical performance (especially with respect to computational cost, and selection of 
numerical solvers). The primal problem (P2-RT0) has more complicated constraints than 
the dual problem (D2-RT0) for which the constraints are (lower) bounds on the design 
variable $\boldsymbol{p}$. Compared to the primal problem (P2-RT0), the dual problem 
(D2-RT0) has a more complicated objective function, which is defined in terms of Schur 
complement operator. 
For all the numerical results presented in this paper, we have used the dual problem 
(D2-RT0). However, we have compared the numerical solutions obtained using the dual 
problem with the primal problem (P2-RT0), and the solutions are identical as predicted 
by the theory.

Special solvers are available in the literature (for example, especially designed interior 
point methods \cite{Fletcher_Optimization,Nocedal_Wright}) that are effective for solving 
problems that belong to the class of quadratic programming with constraints being just 
bounds on design variables. Similarly, special solvers are available for problems involving 
Schur complement operators. For example, the preconditioned conjugate gradient (PCG) solver 
is quite effective for solving large-scale problems involving Schur complement operator. A 
detailed analysis comparing computational costs of the primal (P2-RT0) and dual (Q2-RT0) 
problems is beyond the scope of this paper.
\end{remark}

\section{A NON-NEGATIVE VARIATIONAL MULTISCALE MIXED FORMULATION}
\label{Sec:Optim_VMS}

Masud and Hughes \cite{Masud_Hughes_CMAME_2002_v191_p4341} have proposed a stabilized mixed 
formulation for the first-order form of the Poisson equation that satisfies the LBB condition. 
In this paper we refer  to this formulation as the variational multiscale (VMS) formulation. 
Nakshatrala \textit{et al.} \cite{Nakshatrala_Turner_Hjelmstad_Masud_CMAME_2006_v195_p4036} 
have shown that the variational multiscale formulation can be derived based on the multiscale 
framework proposed by Hughes \cite{Hughes_CMAME_1995_v127_p387}. The variational multiscale 
formulation possesses many favorable numerical properties and performs very well in practice. 
For example, the formulation passes three dimensional patch tests even for distorted elements 
\cite{Nakshatrala_Turner_Hjelmstad_Masud_CMAME_2006_v195_p4036}. 
Another feature of this formulation worth mentioning is that the equal-order interpolation 
for $c$ and $\mathbf{v}$ is stable 
\cite{Hughes_CMAME_1995_v127_p387,Nakshatrala_Turner_Hjelmstad_Masud_CMAME_2006_v195_p4036}. 
However, the variational multiscale formulation in general does not satisfy the discrete 
maximum-minimum principle, which will be illustrated below in Figure \ref{Fig:Optim_HVM_WCTmesh_Problem_2} 
using a simple numerical example. In this section, we present a non-negative method based on the 
variational multiscale mixed formulation. To this end, we first present the weak form and variational 
structure behind the variational multiscale formulation. 

Let the function spaces for the concentration and its corresponding weighting 
function under the VMS formulation be 
\begin{align}
  &\tilde{\mathcal{P}} \equiv H^{1}(\Omega) \\
  &\tilde{\mathcal{Q}} \equiv H^{1}(\Omega) 
\end{align}
where $H^{1}(\Omega)$ is a standard Sobolev space defined on domain $\Omega$ \cite{Brezzi_Fortin}. The 
variational multiscale formulation reads \cite{Nakshatrala_Turner_Hjelmstad_Masud_CMAME_2006_v195_p4036,
  Masud_Hughes_CMAME_2002_v191_p4341}: Find $c(\mathbf{x}) \in \tilde{\mathcal{P}}$ and $\mathbf{v}
(\mathbf{x}) \in \mathcal{V}$ such that 
\begin{align}
\label{Eqn:Optim_VMS_weak_form}
  \left(\mathbf{w};\mathbb{D}^{-1}\mathbf{v}\right) &- \left(\nabla \cdot \mathbf{w}; c\right) 
  + \left(\mathbf{w}\cdot \mathbf{n};c^{\mathrm{p}}\right)_{\Gamma^{\mathrm{D}}} - \left(q;\nabla \cdot 
    \mathbf{v} - f \right) \notag \\
  &- \frac{1}{2}\left(\mathbb{D}^{-1} \mathbf{w} + \nabla q; 
    \mathbb{D}\left(\mathbb{D}^{-1}\mathbf{v} + \nabla c\right)\right) = 0 \quad 
  \forall q(\mathbf{x}) \in \tilde{\mathcal{Q}}, \; \mathbf{w}(\mathbf{x}) \in \mathcal{W}
\end{align}
where $\mathbf{w}(\mathbf{x})$ and $q(\mathbf{x})$ are weighting functions corresponding 
to $\mathbf{v}(\mathbf{x})$ and $c(\mathbf{x})$, and $\mathcal{V}$ and $\mathcal{W}$ are 
defined in equations \eqref{Eqn:Optim_V_space} and \eqref{Eqn:Optim_W_space}, respectively. The 
stationarity (minimizing with respect to $\mathbf{v}$ and maximizing with respect to $c$) of 
the following (continuous) optimization problem 
\begin{align}
  \label{Eqn:Optim_continuous_extremization}
  \mathop{\mbox{extremize}}_{\mathbf{v} \in \mathcal{V}, \; c \in \tilde{\mathcal{P}}} \; 
  \frac{1}{2} \left(\mathbf{v};\mathbb{D}^{-1}\mathbf{v}\right) 
  - (c;\nabla \cdot \mathbf{v} - f) + (\mathbf{v} \cdot \mathbf{n}; 
  c^{\mathrm{p}})_{\Gamma^{\mathrm{D}}} - \frac{1}{4} \left(\mathbb{D}^{-1} \mathbf{v} 
    + \nabla c ; \mathbb{D} (\mathbb{D}^{-1}\mathbf{v} + \nabla c)\right)
\end{align}
is equivalent to the weak form given by equation \eqref{Eqn:Optim_VMS_weak_form}. 

Many practically important problems do not have solutions in $C^2(\Omega) \cap C^{0}(\bar{\Omega})$, and 
hence for these problem one cannot employ the classical maximum-minimum principle, which we have outlined 
in Section \ref{Sec:Optim_Intro}. For example, there exist no (classical) solutions to test problems \#1 and 
\#2 (which are defined in Section \ref{Sec:Optim_numerical}) that belong to $C^{2}(\Omega)$ as the forcing functions 
in both these cases are \emph{not} continuous on $\Omega$. (On the other hand, the solution to test problem \#3 
does belong to $C^{2}(\Omega) \cap C^{0}({\bar{\Omega}})$.) However, weak solutions do exist for test problems \#1 
and \#2. Using $L^{p}$ regularity theory (for example, see Reference \cite{Borsuk_Kondratiev}), one can show that 
these solutions, in fact, belong to $H^{2}(\Omega) \cap C^{1}(\Omega) \cap C^{0}(\bar{\Omega})$. 

\begin{remark}
  Note that $H^{2}(\Omega)$ is not a subset of $C^{1}(\Omega)$ 
  or vice-versa. On the other hand, $H^{2}(\Omega) \subset C^{0}
  (\Omega)$. 
\end{remark}

\subsection{Continuous maximum-minimum principle}
In this subsection we demonstrate one of the main contributions of this paper, namely that 
the weak solution under the variational multiscale formulation (under appropriate regularity 
assumptions) satisfies a continuous maximum-minimum principle. To the authors' knowledge, 
this property of the VMS formulation has \emph{not} been discussed/proved in the literature. 
We employ the standard notation used in mathematical analysis, for example see Reference 
\cite{Evans_PDE}. The standard abbreviation `a.e.' for \emph{almost everywhere} is often 
used in this subsection. We now state and prove a continuous maximum-minimum principle for 
the VMS formulation. 

\begin{theorem}
  Assume that the Dirichlet boundary condition is prescribed on the whole of the 
  boundary (that is, $\Gamma^{\mathrm{D}} = \partial \Omega$), and the diffusivity 
  tensor is assumed to be continuously differentiable. 
  Let $f(\mathbf{x}) \in L^{2}(\Omega)$, and $f(\mathbf{x}) \geq 0$ almost everywhere. 
  Let the weak solution $c(\mathbf{x})$ of the variational multiscale formulation 
  \eqref{Eqn:Optim_VMS_weak_form} belong to $H^{2}(\Omega)\cap C^{1}(\Omega) \cap C^{0}
  (\bar{\Omega})$. Then 
  \begin{align*}
    \min_{\bar{\Omega}} c(\mathbf{x}) = \min_{\partial \Omega} c(\mathbf{x}) 
  \end{align*}
\end{theorem}
\begin{proof}
  Since $c(\mathbf{x}) \in H^{2}(\Omega) \cap C^{1}(\Omega)
  \cap C^{0}(\bar{\Omega})$ we have 
  \begin{align}
    \label{Eqn:DMP_v_D_grad_c}
    \mathbf{v} := -\mathbb{D} \nabla c \in H^{1}(\Omega) 
    \cap C^{0}(\Omega) \subset \mathcal{V}
  \end{align}
  Define $m \in \mathbb{R}$ and a scalar field $s(\mathbf{x})$ 
  such that 
  \begin{align}
    \label{Eqn:DMP_definition_m}
    & m = \min_{\mathbf{x} \in \Gamma^{\mathrm{D}}} 
    c^{\mathrm{p}}(\mathbf{x}) \\ 
    \label{Eqn:DMP_definition_s}
    & s(\mathbf{x}) := \max[m - c(\mathbf{x}), 0] 
    \quad \forall \mathbf{x} \in \bar{\Omega}
  \end{align}
  One can show that the function $s(\mathbf{x})$ is piecewise $C^{1}(\Omega)$, 
  and belongs to $H^{1}(\Omega) \cap C^{0}(\bar{\Omega})$. By construction, we also 
  have 
  \begin{align}
    \label{Eqn:DMP_properties_of_s}
    s(\mathbf{x}) \geq 0 \quad 
    \forall \mathbf{x} \in \bar{\Omega}, \; 
    \mbox{and} \; s(\mathbf{x}) = 0 \quad \forall 
    \mathbf{x} \in \Gamma^{\mathrm{D}}
  \end{align}
  
  Using equation \eqref{Eqn:DMP_v_D_grad_c} (and also employing the divergence 
  theorem) equation \eqref{Eqn:Optim_VMS_weak_form} gets simplified to 
  \begin{align}
    \label{Eqn:DMP_weak_continuity_equation}
    (q; \nabla \cdot \mathbf{v} - f) = 0 
  \end{align}
  Since $s(\mathbf{x}) \in H^{1}(\Omega)$ and $s(\mathbf{x})
  = 0$ on $\Gamma^{\mathrm{D}}$, the scalar field $s(\mathbf{x})$ 
  is a legitimate choice for $q(\mathbf{x})$. Substituting 
  $s(\mathbf{x})$ in the place of $q(\mathbf{x})$, and 
  noting that $s(\mathbf{x}) \in H^{1}(\Omega)$ to allow 
  the application of the divergence theorem; we get
  \begin{align}
    \label{Eqn:DMP_substitute_s_instead_of_q}
    (s;\mathbf{v} \cdot \mathbf{n})_{\partial \Omega} - 
    (\nabla s; \mathbf{v}) - (s;f) = 0 
  \end{align}
  Since $\Gamma^{\mathrm{D}} = \partial \Omega$, $s(\mathbf{x}) 
  = 0 \; \mathrm{on} \; \Gamma^{\mathrm{D}}$, $s(\mathbf{x})
  \geq 0 \; \forall \mathbf{x} \in \bar{\Omega}$, and 
  $f(\mathbf{x}) \geq 0$ a.e. in $\Omega$; we 
  conclude that 
  \begin{align}
    \label{Eqn:DMP_s_v_inequality}
    (\nabla s; \mathbf{v}) = -(\nabla s; \mathbb{D} \nabla c) \leq 0 
  \end{align}
  
  To prove the theorem it is sufficient to show that $s(\mathbf{x})
  = 0 \; \forall \mathbf{x} \in \Omega$ (which implies that 
  $c(\mathbf{x}) \geq m \; \forall \mathbf{x} \in
  \bar{\Omega}$). Note that $s(\mathbf{x}) = m -
  c(\mathbf{x})$ unless $s(\mathbf{x}) = 0$. Let 
  \begin{align}
    \label{Eqn:DMP_definition_of_set_Y}
    \mathbb{Y} := \{\mathbf{x} \in \Omega \; | \;
    s(\mathbf{x}) \neq 0 \} 
    \equiv \{\mathbf{x} \in \Omega \; | \; s(\mathbf{x}) > 0 \}
  \end{align}
  The case $\mathbb{Y} = \emptyset$ is trivial. We now deal with the 
  case when $\mathbb{Y}$ is not empty. We first note that the 
  weak derivative of $s(\mathbf{x})$ is zero on $\Omega \backslash
  \mathbb{Y}$, and is $-\nabla c$ on $\mathbb{Y}$. This result 
  along with equation \eqref{Eqn:DMP_s_v_inequality} implies that 
  \begin{align}
    \label{Eqn:DMP_s_s_inequality}
    (\nabla s; \mathbb{D} \nabla s)_{\mathbb{Y}} \leq 0 
  \end{align}
  Since $\mathbb{D}(\mathbf{x})$ is a positive definite tensor,
  the above equation implies that $s(\mathbf{x}) = s_0 = 
  \mbox{constant}$ \emph{almost everywhere} in $\mathbb{Y}$. 
  Since $s(\mathbf{x})$ is continuous in $\bar{\Omega}$, we 
  conclude that $s(\mathbf{x}) = s_0$ \emph{everywhere} in
  $\mathbb{Y}$. 
  Since $\Gamma^{\mathrm{D}} = \partial \Omega 
  \subseteq \bar{\mathbb{Y}}$, and $s(\mathbf{x}) = 0 \; \mathrm{on}
  \; \Gamma^{\mathrm{D}}$ we conclude that $s(\mathbf{x}) = 0$ 
  on whole of $\mathbb{Y}$ and also on $\bar{\mathbb{Y}}$. Noting the
  fact that the function $s(\mathbf{x})$ vanishes on the set 
  complement of $\mathbb{Y}$, we conclude that $s(\mathbf{x}) 
  = 0$ on whole of $\bar{\Omega}$. Hence, we have proved the desired 
  result. 
\end{proof}

\subsection{Discrete equations}
The discretized finite element equations for the variational multiscale formulation 
(given by equation \eqref{Eqn:Optim_VMS_weak_form}) can be written as 
\begin{align}
  \label{Eqn:Optim_VMS_discrete}
  \left[\begin{array}{cc} 
      \boldsymbol{K}_{vv} & \boldsymbol{K}_{pv}^{T} \\ 
      \boldsymbol{K}_{pv} & -\boldsymbol{K}_{pp}
    \end{array} \right] 
  \left\{\begin{array}{c} 
    \boldsymbol{v} \\ 
    \boldsymbol{p} 
  \end{array} \right\}
  = \left\{\begin{array}{c}
      \boldsymbol{f}_{v} \\
      \boldsymbol{f}_{p} 
    \end{array} \right\}
\end{align}
where the matrices $\boldsymbol{K}_{vv}$ and $\boldsymbol{K}_{pp}$ are symmetric and positive 
definite, $\boldsymbol{v}$ denotes the (finite element) vector of nodal velocity (or auxiliary 
variable) degrees-of-freedom, and $\boldsymbol{p}$ denotes nodal vector of concentration 
degrees-of-freedom.
Comparing the weak form \eqref{Eqn:Optim_VMS_weak_form} and discrete equations \eqref{Eqn:Optim_VMS_discrete}, 
the matrices $\boldsymbol{K}_{vv}$, $\boldsymbol{K}_{pv}$ and $\boldsymbol{K}_{pp}$ are obtained after the finite 
element discretization of the terms $\frac{1}{2} \left(\mathbf{w};\mathbb{D}^{-1} \mathbf{v}\right)$, $-(q,\nabla 
\cdot \mathbf{v}) - \frac{1}{2} \left(\nabla q; \mathbf{v}\right)$ and $-\frac{1}{2}\left(\nabla q; \mathbb{D} 
\nabla c\right)$; respectively. Since equation \eqref{Eqn:Optim_VMS_weak_form} is written in symmetric form, 
$\boldsymbol{K}_{vp}$ (which comes from the term $-(\nabla \cdot \mathbf{w};c) - \frac{1}{2} \left(\mathbf{w};
\nabla c\right)$) will be equal to $\boldsymbol{K}_{pv}^{T}$. Note that, some of the terms in equation 
\eqref{Eqn:Optim_VMS_weak_form} are combined and simplified to obtain the terms presented in the previous line. 
For example, we have used the symmetry of $\mathbb{D}$ in obtaining the term $-\frac{1}{2}\left(\mathbf{w};\nabla 
c\right)$. The vectors $\boldsymbol{f}_{v}$ and $\boldsymbol{f}_p$ are, respectively, obtained from $-(\mathbf{w} 
\cdot \mathbf{n};c^{\mathrm{p}})_{\Gamma^{\mathrm{D}}}$ and $-(q;f)$ after the finite element discretization. 

The discrete form of equation \eqref{Eqn:Optim_continuous_extremization} can be written as 
\begin{align}
  \label{Eqn:Optim_discrete_extremization}
  \mathop{\mathrm{extremize}}_{\boldsymbol{v}, \; \boldsymbol{p}} \; \frac{1}{2} 
  \boldsymbol{v}^{T} \boldsymbol{K}_{vv} \boldsymbol{v} + \boldsymbol{p}^{T} 
  \boldsymbol{K}_{pv} \boldsymbol{v} - \frac{1}{2} \boldsymbol{p}^T \boldsymbol{K}_{pp} 
  \boldsymbol{p} - \boldsymbol{v}^{T} \boldsymbol{f}_{v} - \boldsymbol{p}^{T} \boldsymbol{f}_{p}
\end{align}
Similar to the continuous problem (that is, equations \eqref{Eqn:Optim_VMS_weak_form} 
and \eqref{Eqn:Optim_continuous_extremization} are equivalent), the stationarity of the 
above equation (minimizing with respect to $\boldsymbol{v}$ and maximizing with respect 
to $\boldsymbol{p}$) is equivalent to equation \eqref{Eqn:Optim_VMS_discrete}.  By eliminating 
$\boldsymbol{v}$, the Schur complement form of equation \eqref{Eqn:Optim_VMS_discrete} can be 
written as 
\begin{align}
  \left(\boldsymbol{K}_{pv} \boldsymbol{K}_{vv}^{-1} \boldsymbol{K}_{pv}^{T} + \boldsymbol{K}_{pp}\right) 
  \boldsymbol{p} = \boldsymbol{K}_{pv} \boldsymbol{K}_{vv}^{-1} \boldsymbol{f}_{v} - \boldsymbol{f}_{p}
\end{align}
Clearly, the Schur complement operator $\boldsymbol{K}_{pv} \boldsymbol{K}_{vv}^{-1} 
\boldsymbol{K}_{pv}^{T} + \boldsymbol{K}_{pp}$ is symmetric and positive definite. 
The discrete variational statement of the variational multiscale mixed formulation 
can be posed solely in terms of the variable $\boldsymbol{p}$, and takes the following 
form:
\begin{align}
  \label{Eqn:Optim_HVM_discrete_minimize_wrt_p}
  \mathop{\mathrm{minimize}}_{\boldsymbol{p}} \quad \frac{1}{2} \boldsymbol{p}^{T} 
  \left(\boldsymbol{K}_{pv} \boldsymbol{K}_{vv}^{-1} \boldsymbol{K}_{pv}^{T} + 
    \boldsymbol{K}_{pp} \right) \boldsymbol{p} - \boldsymbol{p}^{T} 
  \left(\boldsymbol{K}_{pv} \boldsymbol{K}_{vv}^{-1} \boldsymbol{f}_v - \boldsymbol{f}_p\right)
\end{align}
As mentioned earlier the variational multiscale mixed formulation does not (always) 
produce non-negative solutions for the non-negative forcing function and non-negative 
prescribed Dirichlet boundary condition. 

\begin{remark}
Unlike in the Raviart-Thomas formulation, the optimization problem \eqref{Eqn:Optim_HVM_discrete_minimize_wrt_p} 
is not the dual problem of equation \eqref{Eqn:Optim_discrete_extremization}. 
\end{remark}

\subsection{A non-negative formulation}
A non-negative formulation based on the variational multiscale formulation can be posed 
as the following constrained minimization problem 
\begin{align}
  \label{Eqn:Optim_VMS_final_objective}
  &\mathop{\mathrm{minimize}}_{\boldsymbol{p}} \quad \frac{1}{2} \boldsymbol{p}^{T} 
  \left(\boldsymbol{K}_{pv} \boldsymbol{K}_{vv}^{-1} \boldsymbol{K}_{pv}^{T} + 
    \boldsymbol{K}_{pp} \right) \boldsymbol{p} - \boldsymbol{p}^{T} 
  \left(\boldsymbol{K}_{pv} \boldsymbol{K}_{vv}^{-1} \boldsymbol{f}_v - 
    \boldsymbol{f}_p\right) \\
  \label{Eqn:Optim_VMS_final_constraint}
  &\mathop{\mbox{subject to}} \quad \boldsymbol{p} \succeq \boldsymbol{0}
\end{align}
The above constrained optimization problem belongs to the class of convex quadratic 
programming, and has a unique global minimizer. 

\begin{remark}
  The variational multiscale formulation (given by equation \eqref{Eqn:Optim_VMS_weak_form}), in general, does 
  not have the (element) local mass balance property. Specifically, one does not have the local mass balance 
  property under linear equal-order interpolation for both $c(\mathbf{x})$ and $\mathbf{v}(\mathbf{x})$, which 
  is employed in this paper. The corresponding non-negative formulation also does not possess the local mass 
  balance property.
\end{remark}

\begin{remark}
  The non-negative method proposed in this section is also applicable for 
  the mixed formulation based on the Galerkin/least-squares. As discussed 
  in Reference \cite{Nakshatrala_Turner_Hjelmstad_Masud_CMAME_2006_v195_p4036} 
  the variational multiscale and Galerkin/least-squares (GLS) mixed formulations 
  differ only in the definition of the stabilization parameter. That is, instead 
  of the term  
  \begin{align*}
    \frac{1}{2}\left(\mathbb{D}^{-1} \mathbf{w} + \nabla q; 
    \mathbb{D}\left(\mathbb{D}^{-1}\mathbf{v} + \nabla c\right)\right)
  \end{align*}
  which is the case for the variational multiscale formulation (see equation 
  \eqref{Eqn:Optim_VMS_weak_form}) we will have 
  \begin{align*}
    \left(\mathbb{D}^{-1} \mathbf{w} + \nabla q; \tau(\mathbf{x})
      \mathbb{D}\left(\mathbb{D}^{-1}\mathbf{v} + \nabla c\right)\right)
  \end{align*}
  and $\tau(\mathbf{x}) \geq 0$ for the GLS mixed formulation. 
  The discrete equations from the GLS formulation also takes the same form 
  as given in equation \eqref{Eqn:Optim_VMS_discrete}. 
\end{remark}

\begin{remark}
As mentioned earlier, non-negative solution is a special case of maximum-minimum principle. Some of the 
formulations presented in the literature produce non-negative solutions but still may violate the (general) 
discrete maximum-minimum principle. For example, see the non-negative formulation presented in Reference 
\cite{Lipnikov_Shashkov_Svyatskiy_Vassilevski_2007_v227_p492}. That is, these formulations avoid undershoots 
but may still produce overshoots. 
  
Though the focus of the present paper is on non-negative solutions, the proposed two non-negative 
optimization-based formulations can be easily extended to satisfy the (general) discrete 
maximum-minimum principle. 
To see this, let us first define the quantities $c_{\mathrm{min}}$ 
and $c_{\mathrm{max}}$ to be the minimum and maximum values of $\mathbf{c}(\mathbf{x})$ based 
on the (continuous) maximum-minimum principle.
Note that the maximum and minimum will occur on the boundary only when $f(\mathbf{x}) = 0$.
In order to enforce these properties in the discrete setting modify the constraints in the 
corresponding optimization problem statements (i.e., equations $\eqref{Eqn:Optim_dual_problem_DMP}_2$ 
and \eqref{Eqn:Optim_VMS_final_constraint})as 
\begin{align}
c_{\mathrm{min}} \boldsymbol{1} \preceq \boldsymbol{p} \preceq c_{\mathrm{max}} \boldsymbol{1}
\end{align}
  where $\boldsymbol{1}$ is a vector of ones of appropriate size. The resulting problems 
  will still belong to quadratic programming. Hence, the proposed mathematical framework 
  and solvers are still applicable. 
However, some interesting questions regarding the numerical performance of the solvers 
(active-set strategy, interior point methods) need to be addressed in future work. For example, 
since in the case of general DMP we have twice the number of constraints than that in the case 
of the non-negative formulation, how large is the violation of the local mass balance in the RT0 
formulation because of the additional constraints? How much additional computational cost will 
be incurred because of the additional constraints.
\end{remark}

\section{NUMERICAL RESULTS}
\label{Sec:Optim_numerical}
In this section we study the performance of the proposed formulations (with respect to non-negative 
solutions and local mass balance) on three canonical test problems. In our numerical experiments 
we have employed five different meshes -- Delaunay, $45$-degree, unstructured and well-centered 
triangular (WCT) meshes; and uniform four-node quadrilateral mesh. In two-dimensions, a well-centered 
triangulation means that all the triangles in the mesh are acute-angled (see Reference 
\cite{Vanderzee_Hirani_Guoy_Ramos_2008}). We will discuss more on WCT meshes in Section 
\ref{Subsec:Optim_WCT}.

The mesh layouts for the aforementioned meshes are shown in Figures \ref{Fig:Optim_typical_meshes}-\ref{Fig:Optim_mesh_problem_4}. 
For the chosen test problems, the variational multiscale and Raviart-Thomas formulations in general do not satisfy the discrete 
maximum-minimum principle. We now show that, for low-order finite elements, the proposed two non-negative mixed formulations 
produce non-negative solutions for all the three test problems and for all the chosen computational meshes. For the variational 
multiscale formulation, we have employed equal order interpolation for the $c$ and $\mathbf{v}$ fields in our numerical simulations. 
Note that (as discussed in Introduction) WCT, 45-degree, Delaunay and square meshes are sufficient to produce non-negative solutions 
for isotropic diffusion. However, these meshes may produce negative solutions in the case of anisotropic diffusion, which will be 
illustrated in this section. 

\subsection{Test problem \#1: Anisotropic and heterogeneous medium}
This test problem is taken from Reference \cite{LePotier_CRM_2005_v341_p787}. The 
computational domain is a bi-unit square with homogeneous Dirichlet boundary conditions. 
The forcing function is taken as 
\begin{align}
  \label{Eqn:Optim_Problem_2_forcing}
  f &= 
  \left\{
    \begin{array}{l}
      1 \quad \mathrm{if}\; (x,y) \in [3/8,5/8]^2 \\
      0 \quad \mathrm{otherwise}
    \end{array} \right.
\end{align}
The diffusivity tensor is given by 
\begin{align}
  \mathbb{D} = \left(
  \begin{array}{cc}
    y^2 + \epsilon x^2 & -(1 - \epsilon)xy \\
    -(1 - \epsilon) xy & \epsilon y^2 + x^2
  \end{array}\right)
\end{align}
In this paper we have taken the parameter $\epsilon = 0.05$. For this test problem, the 
numerical results for the concentration field $c(x,y)$ for various meshes using the 
variational multiscale and corresponding optimization-based formulations are shown 
in Figure \ref{Fig:Optim_HVM_Problem_2}. The contours of the vector field $\boldsymbol{v}$ 
are shown in Figure \ref{Fig:Optim_HVM_velocity_Problem_2}. The numerical results for the 
concentration using the RT0 and corresponding optimization-based formulation are presented 
in Figure \ref{Fig:Optim_RT0_Problem_2}. 

\subsection{Test problem \#2: Diffusion/dispersion tensor in subsurface flows}
This problem is taken from the groundwater modeling literature (for example, see 
Reference \cite{Pinder_Celia}). The diffusivity tensor is given by 
\begin{align}
  \label{Eqn:Optim_Problem_3_diffusivity}
  & \mathbb{D} = a_T \|\boldsymbol{\beta}\| \mathbb{I} + 
  \frac{a_L - a_T}{\|\boldsymbol{\beta}\|} \boldsymbol{\beta} \otimes \boldsymbol{\beta} 
\end{align}
where $\mathbb{I}$ denotes the second-order identity tensor, $\otimes$ the standard 
tensor product \cite{Chadwick}, $\boldsymbol{\beta}$ the velocity vector, and $\alpha_L$ 
and $\alpha_{T}$ are respectively the longitudinal and transverse diffusivity constants. 
Note that $\beta$ is a eigenvector of the diffusivity tensor given in equation 
\eqref{Eqn:Optim_Problem_3_diffusivity}.
In this paper we have taken $\alpha_{L} = 0.1$ and $\alpha_{T} = 0.01$, and the velocity 
vector to be 
\begin{align}
  \label{Eqn:Optim_Problem_3_velocity}
  \boldsymbol{\beta} = \boldsymbol{\mathrm{e}}_x + \boldsymbol{\mathrm{e}}_y 
\end{align}
where $\boldsymbol{\mathrm{e}}_{x}$ and $\boldsymbol{\mathrm{e}}_y$ are the standard 
unit vectors along $x$- and $y$-directions, respectively. The computational domain 
is again a bi-unit square with homogeneous Dirichlet boundary conditions. The forcing 
function is same as in test problem \#1 (see equation \eqref{Eqn:Optim_Problem_2_forcing}). 

For the chosen velocity field \eqref{Eqn:Optim_Problem_3_velocity} (which is aligned 
along south-west to north-east direction) by rotating the current coordinate system 
by $+45$ degrees (i.e., in the anticlockwise direction) the diffusivity tensor written 
in the transformed coordinate system will be isotropic. Therefore, a mesh aligned 
along $+45$-degree mesh should produce non-negative solutions for the chosen velocity 
field (which is illustrated in Tables \ref{Table:Optim_RT0_performance} and 
\ref{Table:Optim_VMS_performance}). However, one will get negative solutions using a 
$-45$-degree mesh. For this test problem, the numerical results for the variational multiscale 
and RT0 formulations and their corresponding optimization-based methods are presented in 
Figures \ref{Fig:Optim_HVM_Problem_3} and \ref{Fig:Optim_RT0_Problem_3}, respectively. 

\subsection{Test problem \#3: Non-smooth anisotropic solution} 
This problem is taken from Reference \cite{Lipnikov_Shashkov_Svyatskiy_Vassilevski_2007_v227_p492}. 
The computational domain is a bi-unit square with a square hole of dimension $[4/9,5/9] \times 
[4/9,5/9]$, which is pictorially described in Figure \ref{Fig:Optim_mesh_problem_4}. The forcing 
function is taken as $f(\boldsymbol{x}) = 0$. On the exterior boundary $c^{\mathrm{p}} (\boldsymbol{x}) 
= 0$ is prescribed, and on the interior boundary $c^{\mathrm{p}}(\boldsymbol{x}) = 2$ is prescribed. 
The diffusivity tensor is given by 
\begin{align}
  \mathbb{D} = 
  \left(\begin{array}{cc} 
      \cos(\theta) & \sin(\theta) \\ 
      -\sin(\theta) & \cos(\theta) \end{array} \right)
  \left(\begin{array}{cc} 
      k_1 & 0 \\ 
      0 & k_2 \end{array} \right)
  \left(\begin{array}{cc} 
      \cos(\theta) & -\sin(\theta) \\ 
      \sin(\theta) & \cos(\theta) \end{array} \right)
\end{align}
In this paper we have taken $k_1 = 1$, $k_2 = 100$ and $\theta = \pi / 6$, which are same as the 
values employed in Reference \cite{Lipnikov_Shashkov_Svyatskiy_Vassilevski_2007_v227_p492}. For 
this test problem, the performance of the variational multiscale and RT0 formulations and their 
corresponding optimization-based formulations are shown in Figure \ref{Fig:Optim_Results_Problem_4}. 
Also it is worth mentioning that (for this test problem) under the RT0 formulation more than $45\%$ of 
the computational domain has negative concentration, which is illustrated in Table 
\ref{Table:Optim_RT0_performance}. 

\subsection{Performance on a well-centered triangular mesh}
\label{Subsec:Optim_WCT}
In two dimensions, a well-centered triangulation means that each element contains its circumcenter, 
which is equivalent to saying that all elements are acute-angled triangles. A well-centered mesh in 
higher dimensions can be similarly defined \cite{Vanderzee_Hirani_Guoy_Ramos_2008}. Note that every 
WCT mesh is also a Delaunay mesh but not vice-versa. For further details on how to generate WCT meshes 
see Reference \cite{Vanderzee_Hirani_Guoy_Ramos_2008}.

It is well known that well-centered triangular (WCT) meshes have some advantages in solving some partial 
differential equations as they preserve some of the underlying mathematical structure. For 
example, as discussed in Introduction, a WCT mesh is sufficient to produce non-negative solutions for 
an \emph{isotropic} diffusion equation. In other words, a WCT mesh respects the discrete maximum-minimum principle 
thereby preserving this key underlying mathematical property. In addition, a WCT mesh enables construction of a 
compatible discretization of a Hodge star (a geometrical object in exterior calculus) \cite{Hirani_PhDThesis_Caltech_2003}. 
In Reference \cite{Hirani_Nakshatrala_Chaudhry_2008} this idea has been used to construct a numerical method for the 
mixed form of the diffusion equation that is locally and globally conservative, and also can exactly represent 
linear variation of concentration in a given computational domain. 

However, in this subsection we numerically show that for the RT0 and variational multiscale formulations, 
even a well-centered triangular (WCT) mesh is not sufficient to produce non-negative solutions in the 
case of full diffusivity tensor. We consider test problem \#1, and use the well-centered triangular mesh 
shown in Figure \ref{Fig:Optim_mesh_WCT}. The obtained numerical results for the concentration using the 
Raviart-Thomas and variational multiscale formulations are shown in Figure \ref{Fig:Optim_HVM_WCTmesh_Problem_2}. 
As one can see, there are regions of negative concentration (which are indicated in white color). The obtained 
numerical results using the corresponding optimization-based formulations are also shown in the figure, and 
(as expected) we have non-negative concentration in the whole domain. 

\subsection{Active set strategy and its numerical performance}
The two main classes of methods for solving quadratic programming problems are active set 
strategy and interior point methods. In this paper we employ the active set strategy, which 
is very effective for small to medium sized convex quadratic programming. For a detailed 
discussion on active set strategy (including a convergence proof) see Luenberger and Ye 
\cite[Section 12.4]{Luenberger_Ye_Nonlinear_Programming}, and for an algorithmic outline 
of the numerical method see Nocedal and Wright \cite[page 462]{Nocedal_Wright}. 

In Tables \ref{Table:Optim_RT0_active_set} and \ref{Table:Optim_VMS_active_set} we have 
studied the performance of active set strategy for the proposed non-negative formulations. 
We considered two cases for the initial active set. The first case is trivial, which is the 
empty set. In the second case, we have take the initial active set as the degrees-of-freedom 
that have negative concentration under the chosen formulation (either VMS or RT0). That is, 
initially one will solve a given problem using either the VMS or RT0 formulation, and then 
identify the nodes (in the case of VMS) or elements (in the case of RT0) that have negative 
concentration. The initial active set is taken as those degrees-of-freedom for concentration 
that have negative values. Based on numerical experiments we found that in many problems the 
second case takes fewer iterations. However, this is not the case always, which is illustrated 
in Tables \ref{Table:Optim_RT0_active_set} and \ref{Table:Optim_VMS_active_set}. Note that in 
these tables, the two choices for initial active set are denoted as `empty set' and `initial 
violated set.'

\subsection{Error in local mass balance under the optimization-based RT0 formulation}
\label{Subsec:Optim_numerical_local_mass_balance}
In subsection \ref{Subsec:Optim_RT0_non_negative} we have shown mathematically that one may have violation 
of local mass balance under the optimization-based RT0 formulation. Based on the KKT optimal conditions we 
have also shown that the violation of local mass balance can occur only in the form of \emph{artificial sinks} 
in some elements. These elements are those for which $\left(\boldsymbol{K}_{pv} \boldsymbol{v} - \boldsymbol{f}_p
\right)_i < 0$, where $i$ denotes the element number. 

In this subsection we study numerically the error in local mass balance (which is characterized by element sink 
strength). For test problems \#1 and \#2 the total source strength is $0.0625$ (which is equal to $\int_{\Omega} 
f \; \mathrm{d} \Omega$). For a given element $\Omega_e$ (with its boundary denoted by $\Gamma_e$) we calculate 
$\int_{\Omega_e} \nabla \cdot \boldsymbol{v} \; \mathrm{d} \Omega \equiv \int_{\Gamma_e} \boldsymbol{v} \cdot 
\boldsymbol{n} \; \mathrm{d} \Gamma$, which should be negative based on the KKT conditions. (As mentioned 
earlier, in the discrete finite element setting the element source/sink strength can be obtained by picking 
the corresponding component in the $\boldsymbol{K}_{pv} \boldsymbol{v} - \boldsymbol{f}_p$ vector.) Contours 
of these element sink strengths are plotted using the built-in cell-centered feature in Tecplot \cite{Tecplot360}. 
In Figures \ref{Fig:Optim_RT0_mass_error_Problem_2} and \ref{Fig:Optim_RT0_mass_error_Problem_3} we have 
shown the contours of element sink strength for various computational meshes for test problems \#1 and 
\#2, respectively. As one can see, for these representative test problems and meshes the violation of 
local mass balance is insignificant. 

For test problem \#3 the volumetric source is zero (i.e., $f(\mathbf{x}) = 0$ in $\Omega$). 
(The problem is driven by non-homogeneous Dirichlet boundary conditions.) Hence, we compare 
the element sink strength with the total flux along the boundary. Under the RT0 formulation (that is, 
without optimization) the total (integrated) flux along the interior and exterior boundaries 
are $-117.3852$ and $+117.3852$, respectively. (This is not surprising as the RT0 formulation 
has both local and global mass balance properties.) Under the optimization-based RT0 formulation, 
total integrated flux along the interior and exterior boundaries are $-117.5615$ and $+127.5694$, 
respectively. Maximum (in magnitude) element sink strength is $0.7477$, which is $0.636\%$ compared 
to the total flux along the interior boundary. The total sink strength by adding the individual 
element volumetric (sink) strengths is $-10.0079$, which matches the difference between the fluxes 
along interior and exterior boundaries. This means that the optimization-based RT0 formulation 
has the global mass balance property (but, as discussed earlier, does not possess local mass balance 
property). In Figure \ref{Fig:Optim_RT0_mass_error_Problem_4} we have shown the contours of element 
sink strength for test problem \#3.

\begin{table} 
  \caption{Performance of the RT0 formulation: minimum concentration produced 
    by the formulation, and percentage of elements that have negative concentrations 
    (denoted as \% of elements violated). \label{Table:Optim_RT0_performance}}
  \begin{tabular}{llcr} \hline
    Test problem & Mesh type & Min. conc. & \% of elements violated \\ \hline
    Problem \#1 & +45-degree   & -0.002510583   & 128/648  $\rightarrow$ 19.75\%  \\
    & Delaunay                 & -0.000011158   & 67/1800  $\rightarrow$ 3.72\%   \\
    & Well-centered            & -0.000824908   & 45/336   $\rightarrow$ 13.39\%  \\

    Problem \#2 & +45-degree   &  0.000000000 &  0/648   $\rightarrow$ 0.00\%   \\
    & -45-degree               & -0.006674991 &  216/648 $\rightarrow$ 33.33\%  \\
    & Delaunay                 & -0.000468342 &  71/1800 $\rightarrow$ 3.94\%   \\
    & Well-centered            & -0.000616864 &  42/336  $\rightarrow$ 12.50\%  \\ 

    Problem \#3 & Mesh in Figure \ref{Fig:Optim_mesh_problem_4} & -0.081453689 & 848/1868 $\rightarrow$ 45.40\% \\ \hline
  \end{tabular}
\end{table}

\begin{table} 
  \caption{Performance of the variational multiscale formulation: minimum 
    concentration produced by the formulation, and percentage of nodes 
    that have negative concentrations (denoted as \% of nodes violated). 
  \label{Table:Optim_VMS_performance}}
  \begin{tabular}{llcr} \hline
    Test problem & Mesh type & Min. conc. & \% of nodes violated \\ \hline
    Problem \#1 & +45-degree   & -0.000986744 & 30/361 $\rightarrow$ 8.31\%  \\
    & Delaunay                 & -0.000000531 & 11/961 $\rightarrow$ 1.14\%  \\
    & Well-centered            & -0.000056615 & 4/191  $\rightarrow$ 2.09\%  \\
    & Four-node quadrilateral  & -0.000000901 & 7/361  $\rightarrow$ 1.93\%  \\ 
    Problem \#2 & +45-degree   &  0.000000000 &  0/361 $\rightarrow$ 0.00\%    \\
    & -45-degree               & -0.000402642 &  40/361 $\rightarrow$ 11.08\%  \\
    & Delaunay                 & -0.000000941 &  12/961 $\rightarrow$ 1.25\%   \\
    & Well-centered            & -0.000007018 &  5/191 $\rightarrow$ 2.62\%    \\
    & Four-node quadrilateral  & -0.000000155 &  6/361 $\rightarrow$ 1.67\%    \\ 
    Problem \#3 & Mesh in Figure \ref{Fig:Optim_mesh_problem_4} & -0.004613415 & 264/998 $\rightarrow$ 26.45\% \\ \hline
  \end{tabular}
\end{table}

\begin{table}
  \caption{Performance of the active-set strategy for the RT0 formulation. We have
    employed the D2-RT0 problem given by equation \eqref{Eqn:Optim_dual_problem_DMP}.
    \label{Table:Optim_RT0_active_set}}
  \begin{tabular}{llc} \hline
    Test problem & Mesh type & Active set iterations \\
                 &      & (empty set, initial violated set) \\ \hline
    Problem \#1 & +45-degree   & (105,43)  \\
    & Delaunay                 & (25,53)   \\
    & Well-centered            & (31,20)   \\

    Problem \#2 & +45-degree   & N/A      \\
    & -45-degree               & (185,57) \\
    & Delaunay                 & (29,51)  \\
    & Well-centered            & (32,15)  \\
    
    Problem \#3 & Mesh in Figure \ref{Fig:Optim_mesh_problem_4} & (658,436) \\ \hline
    \end{tabular}
\end{table}

\begin{table} 
  \caption{Performance of the active-set strategy for the variational multiscale 
    formulation. We have employed the non-negative formulation given by equations 
    \eqref{Eqn:Optim_VMS_final_objective} and \eqref{Eqn:Optim_VMS_final_constraint}. 
    \label{Table:Optim_VMS_active_set}}
  \begin{tabular}{llc} \hline
    Test problem & Mesh type & Active set iterations \\ 
                 &      & (empty set, initial violated set) \\ \hline
    Problem \#1 & +45-degree   & (18,15)  \\
    & Delaunay                 & (8,5)    \\
    & Well-centered            & (4,2)    \\
    & Four-node quadrilateral  & (7,3)    \\ 
    Problem \#2 & +45-degree   & N/A     \\
    & -45-degree               & (25,17) \\
    & Delaunay                 & (6,9)   \\
    & Well-centered            & (5,2)   \\
    & Four-node quadrilateral  & (7,1)   \\ 
    Problem \#3 & Mesh in Figure \ref{Fig:Optim_mesh_problem_4} & (72,240) \\ \hline
  \end{tabular}
\end{table}

\subsection{$h$-Convergence analysis}
In this subsection we study the convergence of the proposed optimization-based methods with 
respect to mesh refinement. We use test problems \#1 and \#2, and employ $45$-degree triangular 
and four-node quadrilateral meshes. (As discussed earlier, we use $+45$-degree and $-45$-degree 
meshes for test problems \#1 and \#2, respectively.) Typical four-node quadrilateral and $45$-degree 
triangular meshes are shown in Figures \ref{Fig:Optim_mesh_Quad4} and \ref{Fig:Optim_typical_meshes}(b), 
respectively. We refine the mesh by increasing the number of nodes along each side. 

In Tables \ref{Table:Optim_HVM_T3_performance} and \ref{Table:Optim_HVM_Quad4_performance} we 
show the variation of minimum concentration and percentage of nodes that produce negative solutions 
with respect to mesh refinement under the VMS formulation for $45$-degree triangular and four-node 
quadrilateral meshes. In Table \ref{Table:Optim_RT0_45_degree_mesh_performance} we have shown 
the variation of minimum concentration and percentage of elements that produced non-negative solutions 
with respect to mesh refinement under the RT0 formulation. (Note that in the RT0 formulation, 
the concentration, $c(\mathbf{x})$, is a constant over each element.) As expected, the percentage 
of non-negative nodes and minimum concentration decrease with respect to mesh refinement, but 
still the persistent non-negative values and spatial extent of the violation prohibits the use of VMS 
and RT0 formulations in simulations for which transport is coupled with reactions. 

In Figures \ref{Fig:Optim_HVM_active_set_strategy} and \ref{Fig:Optim_RT0_active_set_strategy} we have 
compared the number of iterations taken by the active set strategy with respect to mesh refinement for 
the proposed optimization-based methods. As one can see from these figures, the number of iterations 
stabilized with respect to mesh refinement (that is, the strategy takes almost the same number of 
iterations as the mesh is refined). The optimization-based VMS method takes fewer active-set strategy 
iterations compared to the iterations taken by the optimization-based RT0 formulation. 

Figures \ref{Fig:Optim_HVM_CPU_time} and \ref{Fig:Optim_RT0_CPU_time}, respectively, compare the CPU 
time taken by the optimization-based VMS and RT0 methods with respect to mesh refinement. On the $y$-axis 
we plot the ratio between the \emph{additional} CPU time taken by the active-set strategy (or the optimization 
solver to obtain non-negative solution) and the CPU time taken by the corresponding underlying mixed formulation 
(either VMS or RT0 formulation). For the optimization-based VMS method, the additional cost to obtain non-negative 
solution using the four-node quadrilateral mesh is only a fraction of the computational cost of the VMS formulation. 
For the three-node triangular mesh, the additional cost to obtain the non-negative solution is nearly twice the cost of 
the VMS formulation. The optimization-based RT0 method takes relatively more additional CPU time to obtain non-negative 
solution, and the ratio between the additional CPU time and the CPU time taken by the RT0 formulation is nearly 5 for 
test problem \#1 and 20 for test problem \#2. This should not be surprising as in the RT0 formulation the primary 
variable (in our case, the concentration) is poorly approximated by its piecewise constant representation over each 
element. (Note that, in the VMS formulation the primary variable is $C^{0}$ continuous, that is, piecewise linear 
and continuous across elements.)
 
In Figure \ref{Fig:Optim_RT0_min_total_errors} we compare the error in the local mass balance with respect 
to mesh refinement for the RT0 formulation. Since, in the non-negative version of the RT0 formulation we 
always have artificial sinks (that is, the error in the local mass balance will always be negative), we have 
taken the negative of the error before taking the logarithm. We have plotted both the maximum (artificial) 
element sink strength (or error in local mass balance in an element) and also the total sink strength by 
summing the contribution from all elements. From Figure \ref{Fig:Optim_RT0_min_total_errors} one can see 
that, for the chosen problems, \emph{the error in the local mass balance decreases exponentially with respect 
to the element size.}

\newpage
\begin{table} 
\caption{Performance of the variational multiscale formulation using three-node 
triangular element with respect to mesh refinement. We have employed $+45$-degree 
and $-45$-degree meshes for test problems \#1 and \#2, respectively. 
\label{Table:Optim_HVM_T3_performance}}
\begin{tabular}{lccr} \hline
Test problem  &  \# of nodes per side  &  Min. Conc.  &	 \% nodes violated  \\ \hline
Problem \#1   &  10    &  -3.19E-003  &   9/100	   $\rightarrow$  9\%     \\
              &  19    &  -9.87E-004  &	  30/361   $\rightarrow$  8.31\%  \\
              &  28    &  -1.01E-004  &	  54/784   $\rightarrow$  6.89\%  \\
              &  37    &  -8.41E-006  &	  64/1369  $\rightarrow$  4.67\%  \\
              &  46    &  -1.05E-006  &	  71/2116  $\rightarrow$  3.36\%  \\
              &  55    &  -2.05E-007  &	  71/3025  $\rightarrow$  2.35\%  \\
              &  64    &  -7.38E-008  &	  75/4096  $\rightarrow$  1.83\%  \\
              &  73    &  -3.46E-008  &	  77/5329  $\rightarrow$  1.44\%  \\ \hline
Problem \#2  &  10  &	-1.23E-003  &	6/100   $\rightarrow$ 6\%       \\
             &  19  &	-4.03E-004  &	40/361  $\rightarrow$ 11.08\%   \\
             &  28  &	-4.72E-005  &	66/784  $\rightarrow$ 8.42\%    \\
             &  37  &	-5.41E-006  &	66/1369	$\rightarrow$ 4.82\%    \\
             &  46  &	-9.57E-007  &	66/2116	$\rightarrow$ 3.12\%    \\
             &  55  &	-2.40E-007  &	66/3025	$\rightarrow$ 2.18\%    \\
             &  64  &	-8.50E-008  &	66/4096	$\rightarrow$ 1.61\%    \\
             &  73  &	-3.45E-008  &	66/5329	$\rightarrow$ 1.24\%    \\ \hline
\end{tabular}
\end{table}

\begin{table} 
\caption{Performance of the variational multiscale formulation using 
four-node quadrilateral element with respect to mesh refinement. 
\label{Table:Optim_HVM_Quad4_performance}}
\begin{tabular}{lccr} \hline
Test problem & \# of nodes per side & Min. conc. & \% of nodes violated \\ \hline
Problem \#1 & 11  &  -1.15E-004  &  7/121   $\rightarrow$  5.79\%  \\
            & 21  &  -2.02E-007  &  8/441   $\rightarrow$  1.81\%  \\
            & 31  &  -7.17E-009  &  8/961   $\rightarrow$  0.83\%  \\
            & 41  &  -1.01E-009  &  9/1681  $\rightarrow$  0.54\%  \\
            & 51  &  -2.32E-010  &  9/2601  $\rightarrow$  0.35\%  \\ \hline
Problem \#2  & 11  &   -1.15E-004  &  6/121  $\rightarrow$  4.96\%  \\
             & 21  &   -2.02E-007  &  6/441  $\rightarrow$  1.36\%  \\ 
             & 31  &   -7.17E-009  &  6/961  $\rightarrow$  0.62\%  \\
             & 41  &   -1.01E-009  &  6/1681 $\rightarrow$  0.36\%  \\
             & 51  &   -2.32E-010  &  6/2601 $\rightarrow$  0.23\%  \\ \hline
\end{tabular}
\end{table}

\begin{table} 
\caption{Performance of the RT0 formulation with respect to mesh refinement. 
We have used $+45$-degree and $-45$-degree meshes for test problems \#1 and 
\#2, respectively. \label{Table:Optim_RT0_45_degree_mesh_performance}}
\begin{tabular}{lccr} \hline
Test problem & \# of nodes per side & Min. conc. & \% of elements violated \\ \hline
Problem \#1 & 10  &  -0.029349220  &	48/162    $\rightarrow$  29.63\% \\
            & 19  &  -0.002510583  &	128/648	  $\rightarrow$  19.75\% \\
            & 28  &  -0.000065471  &	158/1458  $\rightarrow$  10.84\% \\
            & 37  &  -0.000027412  &	194/2592  $\rightarrow$   7.48\% \\
            & 46  &  -0.000012794  &	224/4050  $\rightarrow$   5.53\% \\
            & 55  &  -0.000006529  &	248/5832  $\rightarrow$	  4.25\% \\ \hline
Problem \#2 & 10  &   -0.039682770  &	64/162	  $\rightarrow$ 39.51\% \\
            & 19  &   -0.006674991  &	216/648	  $\rightarrow$ 33.33\% \\
            & 28  &   -0.001901262  &	384/1458  $\rightarrow$	26.34\% \\
            & 37  &   -0.000756689  &	504/2592  $\rightarrow$	19.44\% \\
            & 46  &   -0.000337093  &	552/4050  $\rightarrow$	13.63\% \\
            & 55  &   -0.000140215  &	576/5832  $\rightarrow$	 9.88\% \\ \hline
  \end{tabular}
\end{table}

\section{CONCLUSIONS}
\label{Sec:Optim_conclusions}
Tensorial diffusion problems arise in a variety of important engineering and scientific 
applications.  Although the continuous problem satisfies a maximum-minimum principle, most 
numerical approximations fail to satisfy this principle in a discrete sense on arbitrary meshes.  
For some applications, violation of the discrete maximum-minimum principle can be problematic 
due to physically meaningless negative values of the dependent variable.

In this paper, we proposed two non-negative low-order mixed finite element formulations for 
the tensorial diffusion equation. (That is, the proposed formulations provide non-negative 
numerical solutions for linear, bilinear and trilinear finite elements.) This is achieved by 
rewriting the formulations as constrained optimization problems. In both the cases, the problem 
belongs to convex quadratic programming, which can be effectively solved using existing numerical 
optimization solvers (e.g., active set strategy and interior point methods).  

One of the formulations is based on the variational multiscale formulation, and the other is based 
on the lowest-order Raviart-Thomas spaces (that is, the RT0 element). We have demonstrated that the 
variational multiscale formulation satisfies a continuous maximum-minimum principle. In the case of 
non-negative formulation based on Raviart-Thomas spaces, two different optimization problems are 
presented -- the primal and dual problems. From the optimization theory it has been inferred that 
these two problems are equivalent. In addition, from the Karush-Kuhn-Tucker optimality conditions 
it is inferred that one \emph{may} have violation of (element) local mass balance in those elements 
that have zero concentration. These violations of local mass balance are in some limited part of the 
domain, and the violations are small for the test cases studies here. 

We have studied the convergence properties of the proposed optimization-based methods. Through numerical 
experiments we have shown that the error in local mass balance under the optimization-based RT0 method 
decreases exponentially with respect to mesh refinement. We have also studied the performance of active-set 
strategy method for solving the resulting convex quadratic programming problems. The performance of the 
proposed non-negative formulations is illustrated on three representative problems, and the formulations 
performed well.

One note worthy feature is that existing solvers based on the variational multiscale and RT0 formulations 
can be easily extended to implement the proposed optimization-based non-negative formulations. Designing 
a non-negative (stabilized) mixed formulation that also possesses local mass balance property is part of 
our future work.

\section*{APPENDIX: SOME CLASSICAL RESULTS ON DISCRETE MAXIMUM-MINIMUM PRINCIPLE} 
One of the early works on DMP dating from the 1960s is by Varga 
\cite{Varga_Matrix_Iterative_Analysis,Varga_SIAMJNA_1966_v3_p355}, and was presented in 
the context of finite difference schemes. To the authors' knowledge, the initial work 
on DMP in the context of the finite element method was done by Ciarlet and Raviart 
\cite{Ciarlet_Raviart_CMAME_1973_v2_p17}. (Another relevant work by Ciarlet is 
\cite{Ciarlet_AequationesMath_1970_v4_p338}, which was in the context of finite difference 
operators.) Reference \cite{Ciarlet_Raviart_CMAME_1973_v2_p17} addressed linear simplicial 
finite elements, and considered the classical Galerkin single-field formulation for the 
Poisson and Helmholtz equations.  

As mentioned earlier, classical finite element formulations do not satisfy the DMP on 
general meshes for full diffusivity tensor. The formulations which satisfy DMP impose 
severe restrictions on both meshes and coefficients of the diffusivity tensor. We now 
outline some of the classical results that are available on DMP in the context of the 
finite element method for a \emph{scalar} diffusion equation.
\begin{itemize}
\item For linear simplicial elements, Ciarlet and Raviart \cite{Ciarlet_Raviart_CMAME_1973_v2_p17} 
  have shown that non-obtuseness is \emph{sufficient} to satisfy DMP.
\item Christie and Hall \cite{Christie_Hall_IJNME_1984_v20_p549} have presented \emph{sufficient} 
  conditions for bilinear finite elements to satisfy the DMP under a homogeneous forcing function. 
  The results can be summarized as follows. For a non-uniform rectangular mesh (see Figure 
  \ref{Fig:Optim_non_uniform_rectangular_mesh}), the DMP will be satisfied provided 
  \begin{align}
    h_{1} h_{2} \leq \frac{1}{2} \mathrm{max}\left(k_{1}^2, \; k_{2}^2\right) \quad 
    \mathrm{and} \quad k_{1} k_{2} \leq \frac{1}{2} \mathrm{max}\left(h_{1}^2, \; h_{2}^2\right) 
  \end{align}
  This implies that a uniform rectangular mesh (that is, $h_1 = h_2$ and $k_1 = k_2$) will 
  satisfy the DMP provided 
  \begin{align}
    \frac{1}{\sqrt{2}} k_1 \leq h_1 \leq \sqrt{2} \; k_1
  \end{align}
  This further implies that a mesh with squares (that is, $h_1 = h_2 = k_1 = k_2$) will 
  always satisfy the DMP for the case of a scalar diffusion equation. 
\item Vanselow \cite{Vanselow_AppMath_2001_v46_p13} has shown that a Delaunay triangulation 
  along with an additional condition on boundary nodes are sufficient for the DMP under the 
  classical single-field Galerkin formulation. We now outline how this additional condition 
  looks for a convex domain. To this end, let $P_1$ and $P_2$ be two neighboring nodes on 
  the boundary. Let us denote their spatial coordinates as 
  $\mathbf{x}_1$ and $\mathbf{x}_2$, and define $\tilde{\mathbf{x}} := 
  \left(\mathbf{x}_1 + \mathbf{x}_2\right)/2$. Then the additional condition 
  can be written as 
  \begin{align}
    \| \mathbf{x}_i - \tilde{\mathbf{x}} \|_2 \leq \|\mathbf{x}_{Q} - 
    \tilde{\mathbf{x}} \|_{2} \; \mbox{for all nodes $Q$ in the triangulation with} \; 
    Q \neq P_i, \; i = 1, 2
  \end{align}
  where $\mathbf{x}_{Q}$ denotes the spatial coordinates of the node $Q$. A similar condition is required 
  for a non-convex domain. In addition, it has also been shown that under some weak additional assumptions on 
  the triangulation, these conditions are necessary \cite[Section 4]{Vanselow_AppMath_2001_v46_p13}. 
\end{itemize}

\begin{figure}
  \psfrag{h1}{$h_1$}
  \psfrag{h2}{$h_2$}
  \psfrag{k1}{$k_1$}
  \psfrag{k2}{$k_2$}
  \includegraphics[scale=0.5]{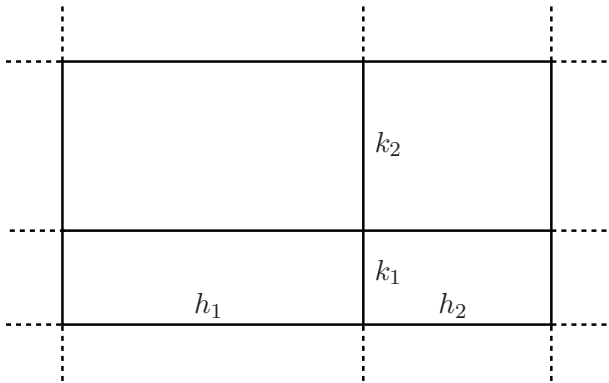}
  \caption{Non-uniform rectangular mesh where $h_{1}$ and $h_{2}$ denote the length 
    of horizontal edges of two neighboring elements. Similarly, $k_{1}$ and $k_{2}$ 
    are for corresponding vertical sides. \label{Fig:Optim_non_uniform_rectangular_mesh}}
\end{figure}

  \section*{ACKNOWLEDGMENTS}
  The research reported herein was supported by the Department of
  Energy through a SciDAC-2 project (Grant No. DOE DE-FC02-07ER64323). 
  This support is gratefully acknowledged. The opinions expressed in 
  this paper are those of the authors and do not necessarily reflect 
  that of the sponsor. We also thank Professor Anil Hirani, University 
  of Illinois at Urbana-Champaign, for providing us with the well-centered 
  triangular mesh that is used in this paper. The first author is grateful 
  to Dr. Vit Pr\.{u}\v{s}a for valuable suggestions. 
  
\bibliographystyle{unsrt}
\bibliography{Books,Master_References}

\newpage
\newpage 

\begin{figure}
  \subfigure[Delaunay mesh]{ 
    \includegraphics[scale=0.45]{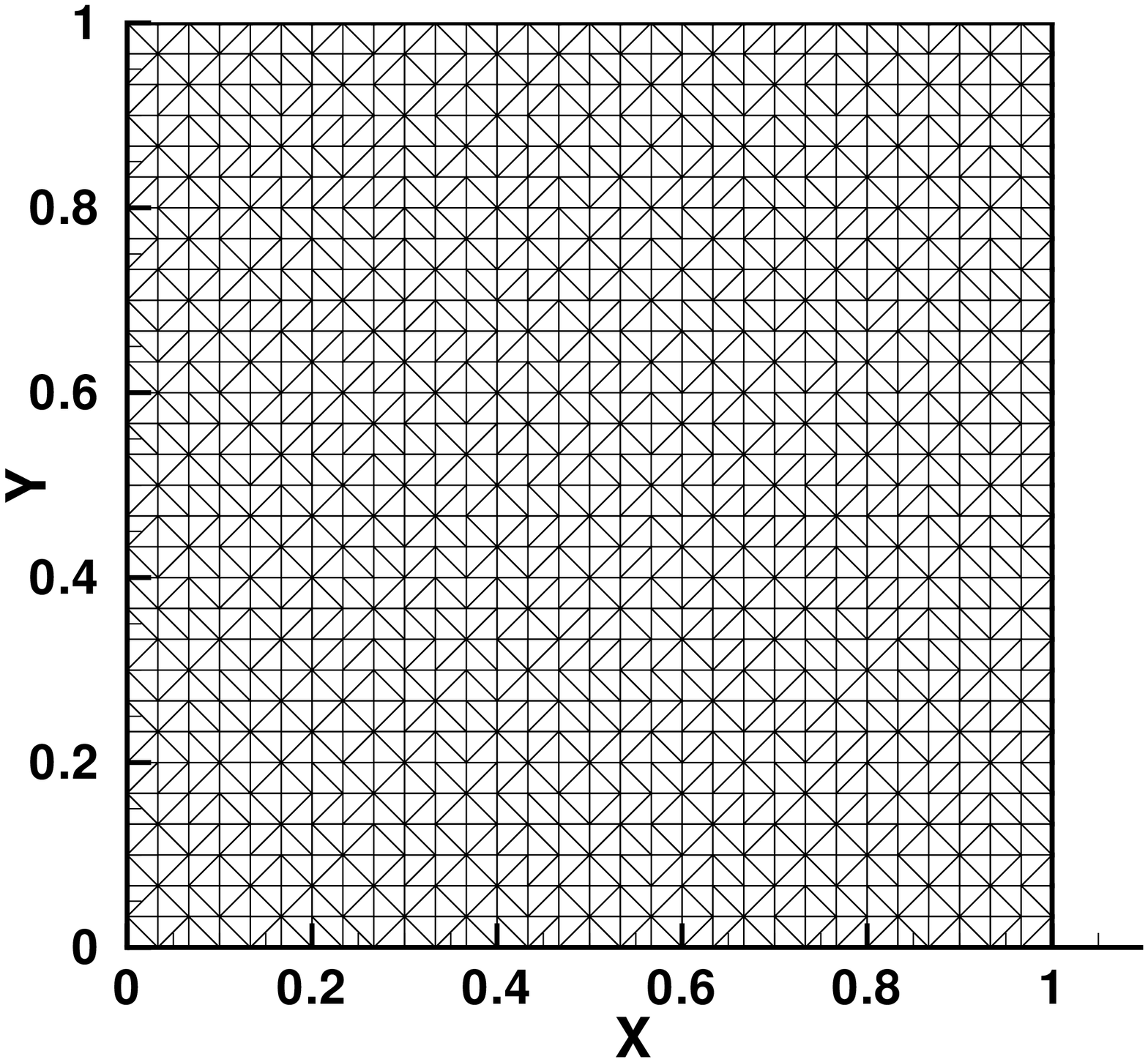}}
  \subfigure[+45-degree mesh]{
    \includegraphics[scale=0.45]{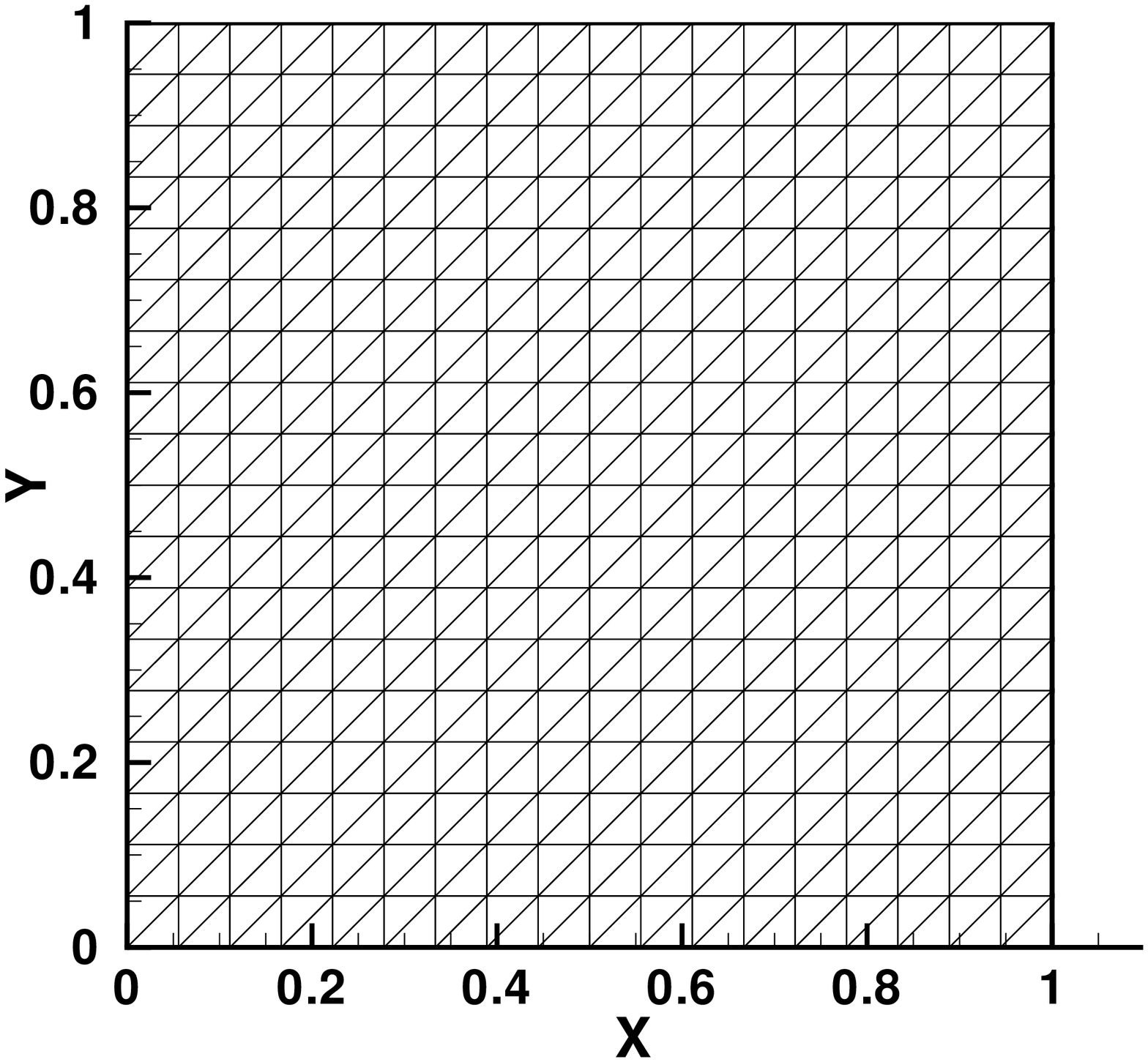}}
  \caption{Typical triangular meshes [Delaunay (top) and $+45$-degree (bottom)] 
    that are used for test problems 1-3 are shown in the figure. In addition, $-45$-degree 
    mesh is also used in numerical simulations, which is similar to the $+45$-degree except 
    that the diagonals run along south-east to north-west direction. \label{Fig:Optim_typical_meshes}}
\end{figure}

\begin{figure}
  \includegraphics[scale=0.45]{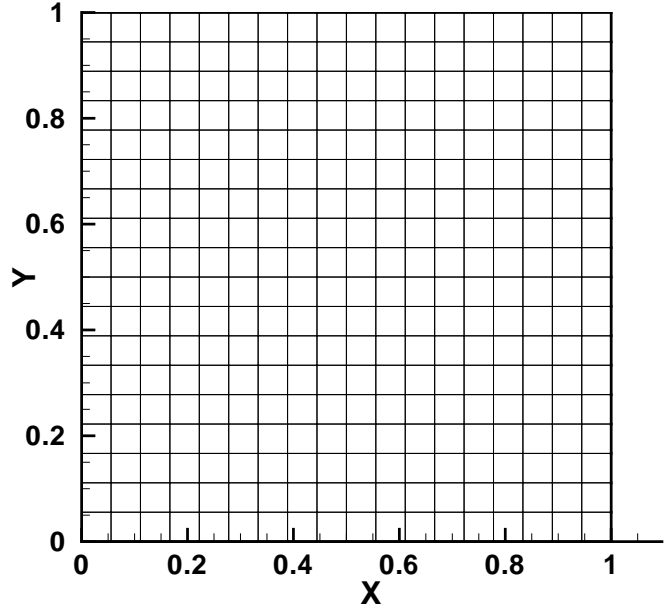}
    \caption{Typical four-node quadrilateral mesh using in numerical simulations. 
    \label{Fig:Optim_mesh_Quad4}}
\end{figure}

\begin{figure}
  \includegraphics[scale=0.45]{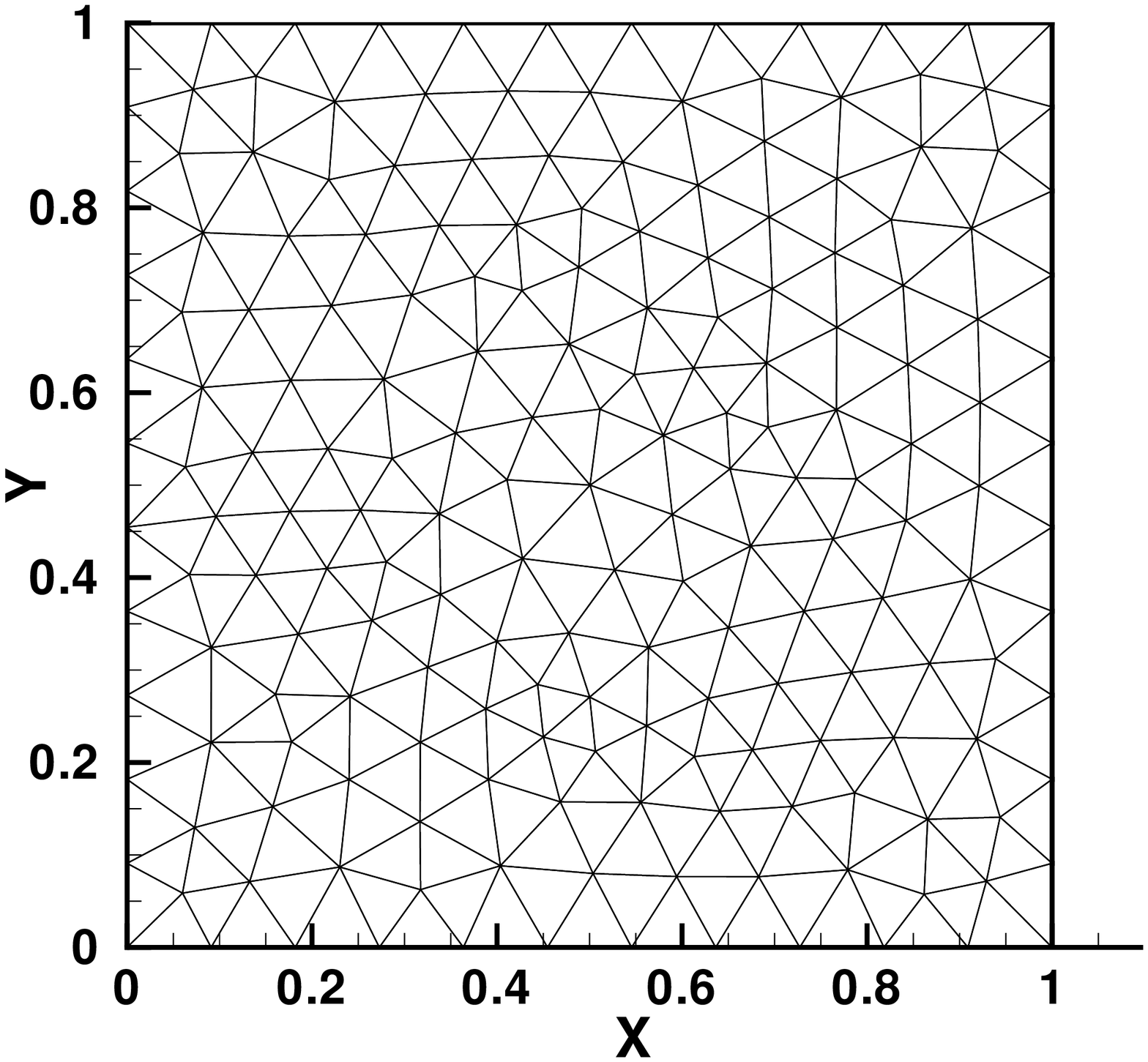}
    \caption{Well-centered triangular (WCT) mesh. \label{Fig:Optim_mesh_WCT}}
\end{figure}

\begin{figure}
  \includegraphics[scale=0.45]{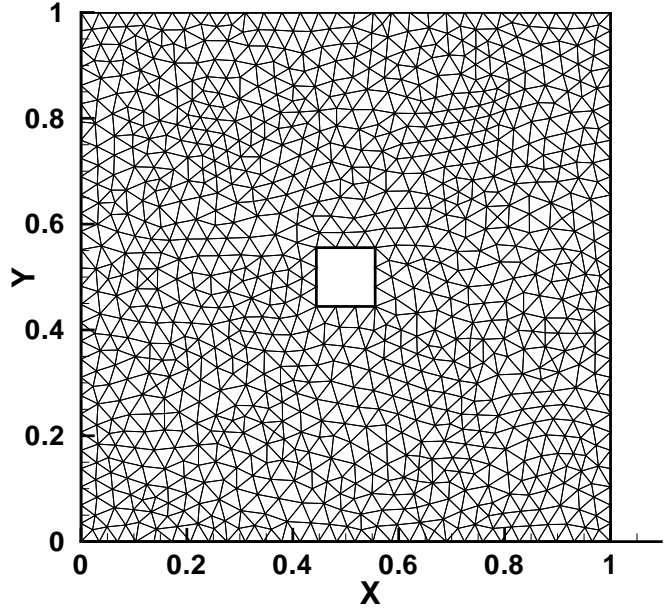}
  \caption{Pictorial description of test problem \#3: Computational domain is the bi-unit square 
    with a square hole $[4/9,5/9] \times [4/9,5/9]$. On the exterior boundary $c^{\mathrm{p}}
    (\boldsymbol{x}) = 0$ is prescribed. On the interior boundary $c^{\mathrm{p}}(\boldsymbol{x}) 
    = 2$ is prescribed. The computational domain is triangulated using Gmsh \cite{gmsh.www}.  
    \label{Fig:Optim_mesh_problem_4}}
\end{figure}


\begin{figure}[htbp]
  \subfigure{
    \includegraphics[scale=0.32]{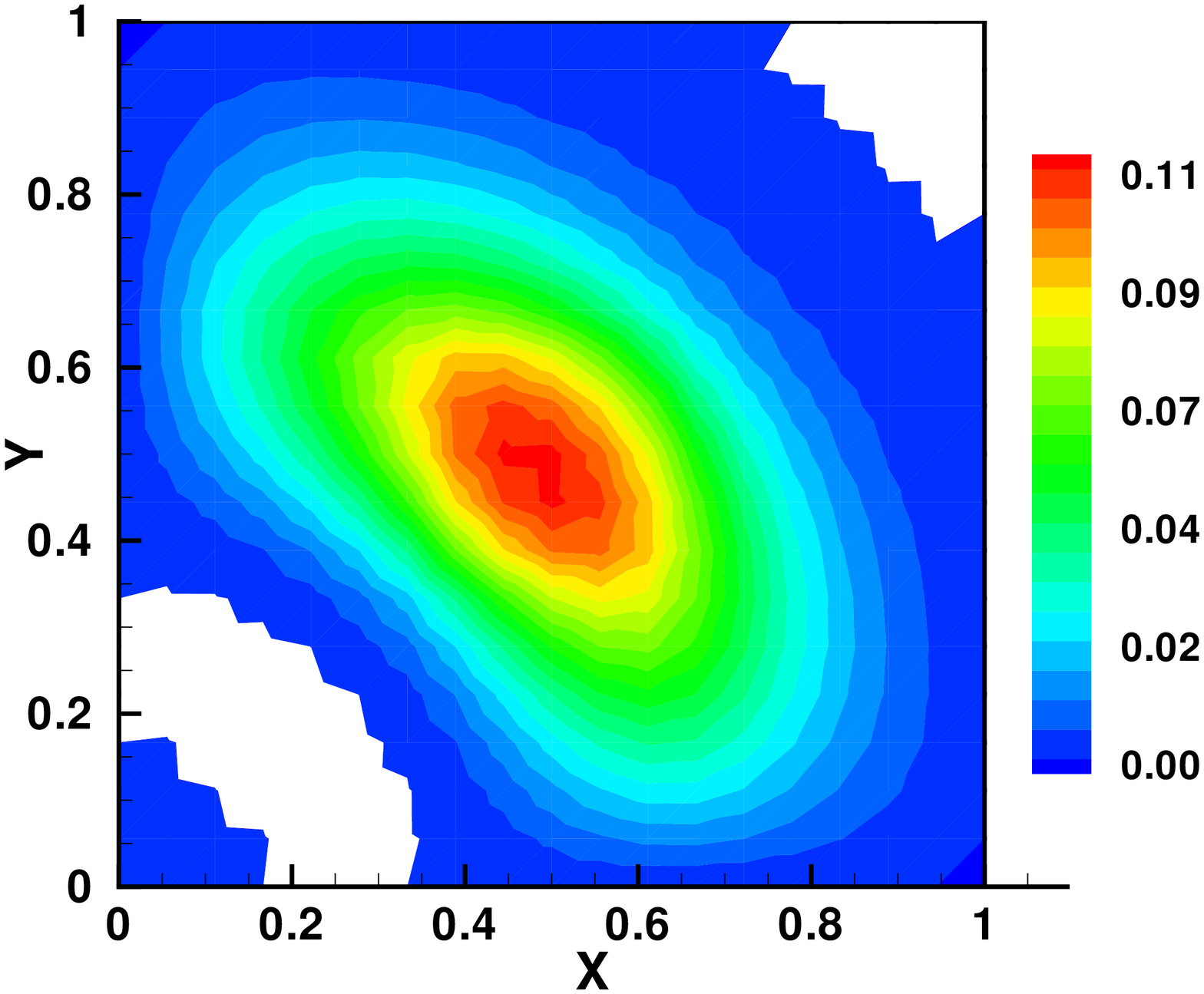}}
  \subfigure{
    \includegraphics[scale=0.32]{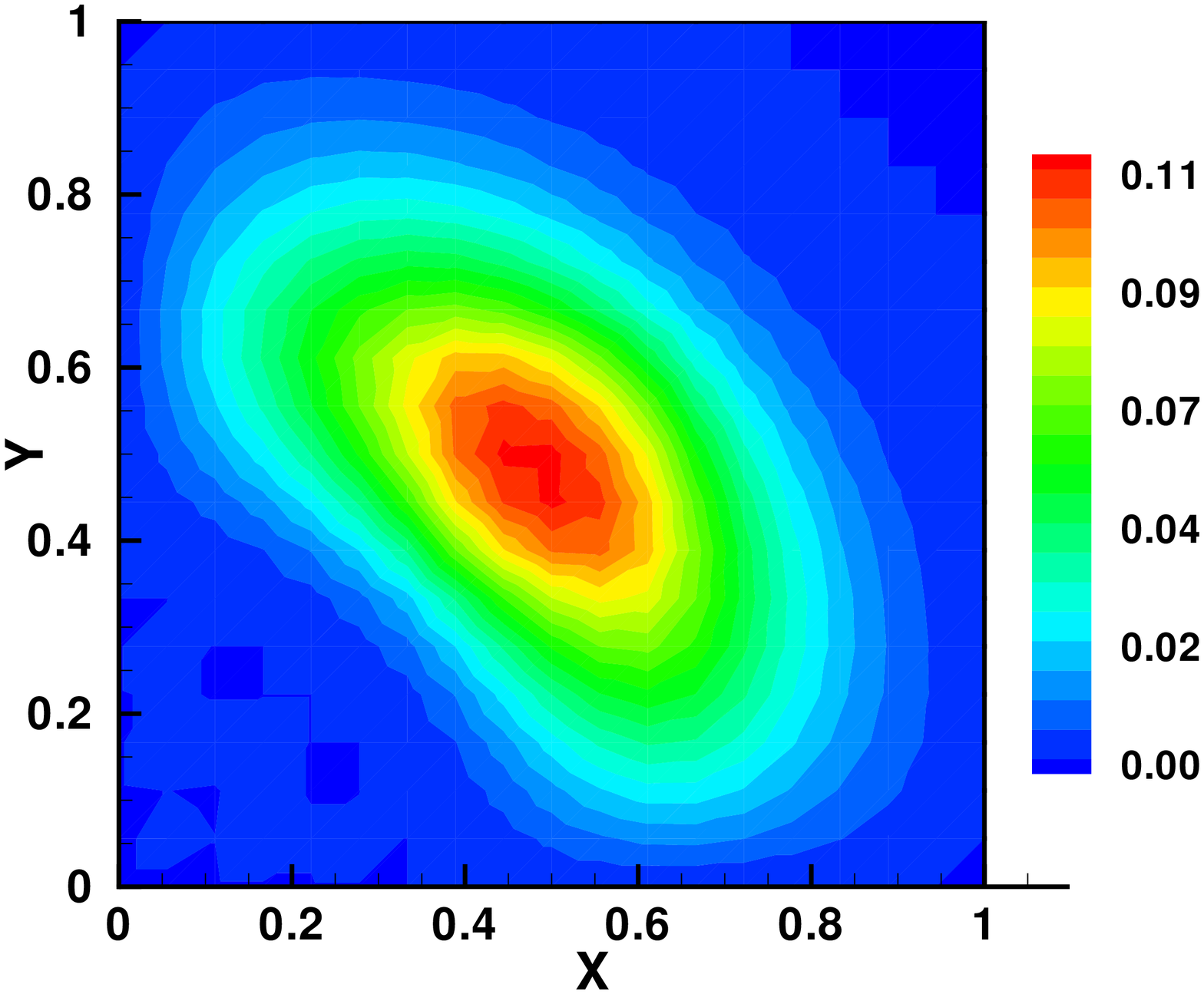}}
 \subfigure{
    \includegraphics[scale=0.32]{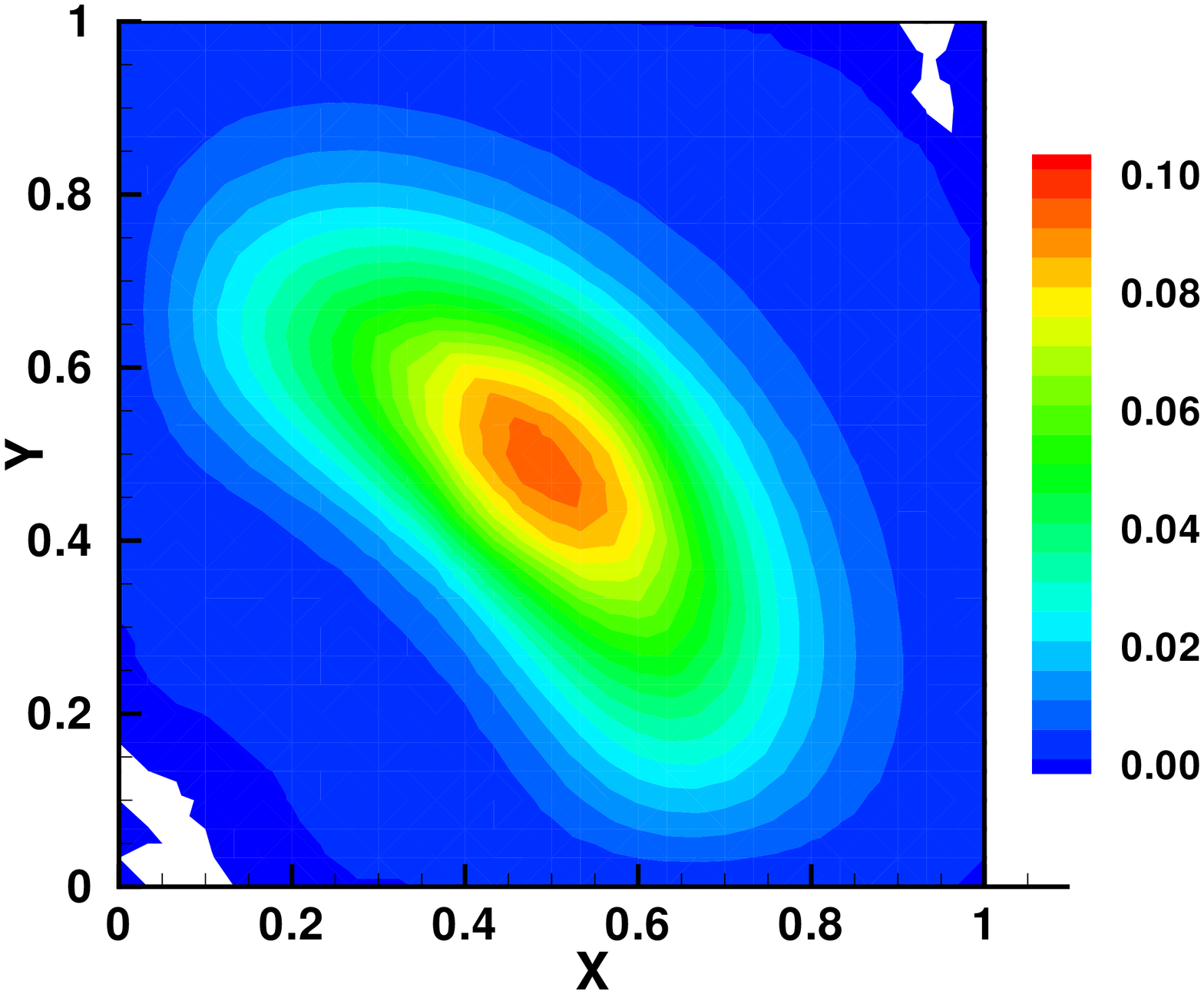}}
  \subfigure{
    \includegraphics[scale=0.32]{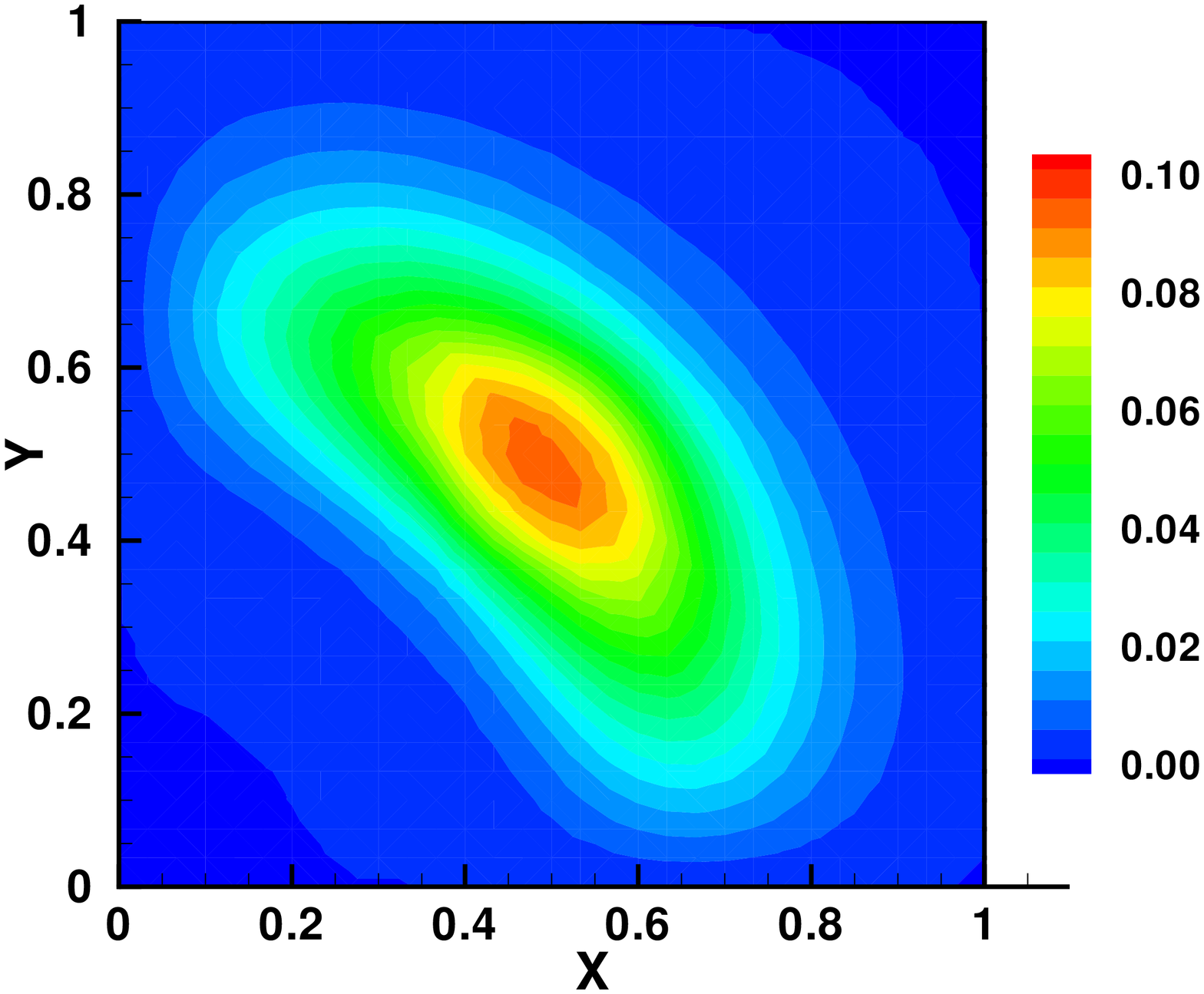}}
  \subfigure{
    \includegraphics[scale=0.32]{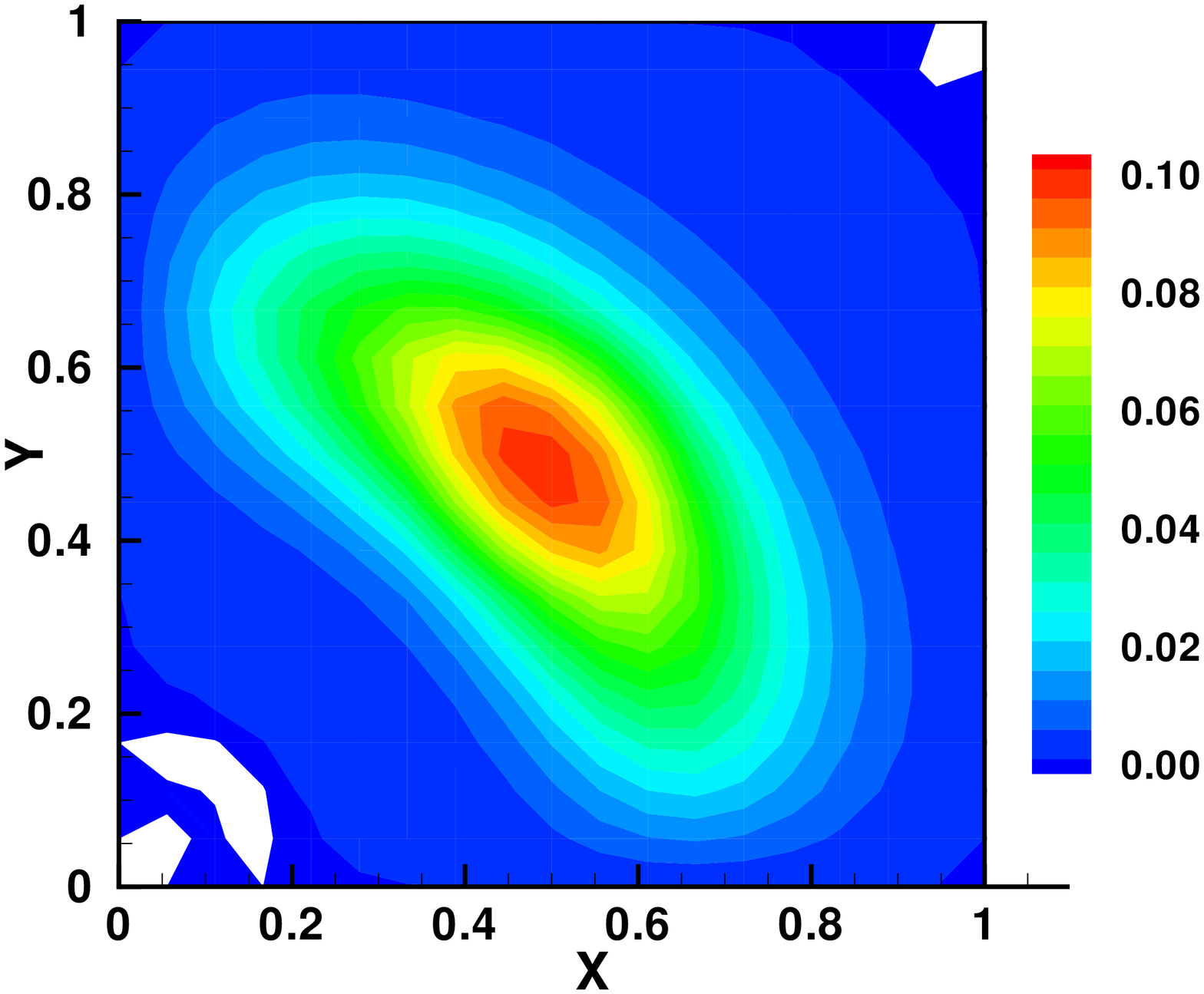}}
  \subfigure{
    \includegraphics[scale=0.32]{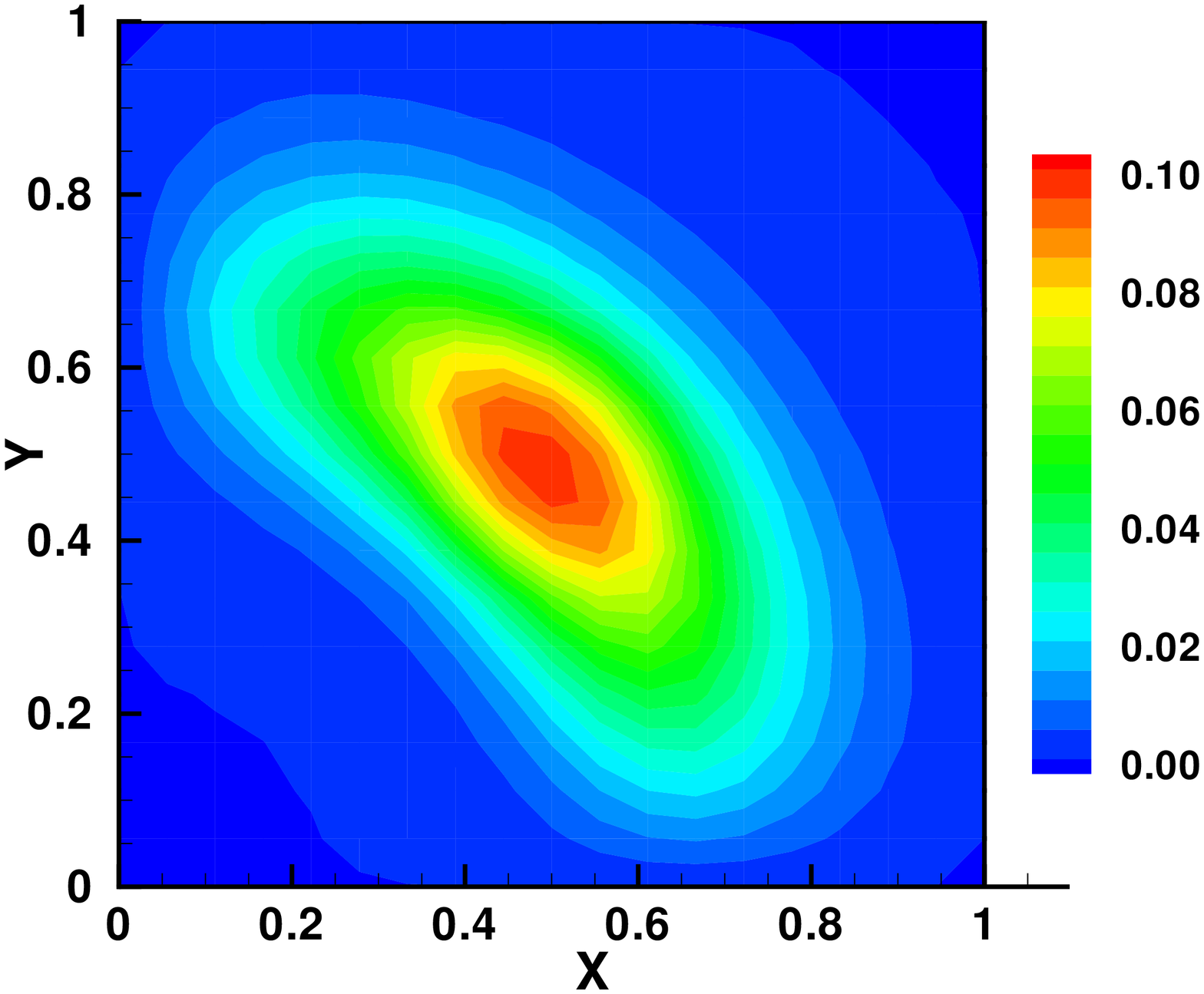}}
  \caption{Test problem \#1: Performance of the variational multiscale (left) and 
    corresponding optimization-based (right) formulations. In the numerical simulations 
    45-degree (top), Delaunay (middle) and four-node quadrilateral (bottom) meshes are 
    employed. The regions that have negative concentration are indicated in white color. 
    \label{Fig:Optim_HVM_Problem_2}}
\end{figure}

\begin{figure}[htbp]
  \subfigure{
    \includegraphics[scale=0.32]{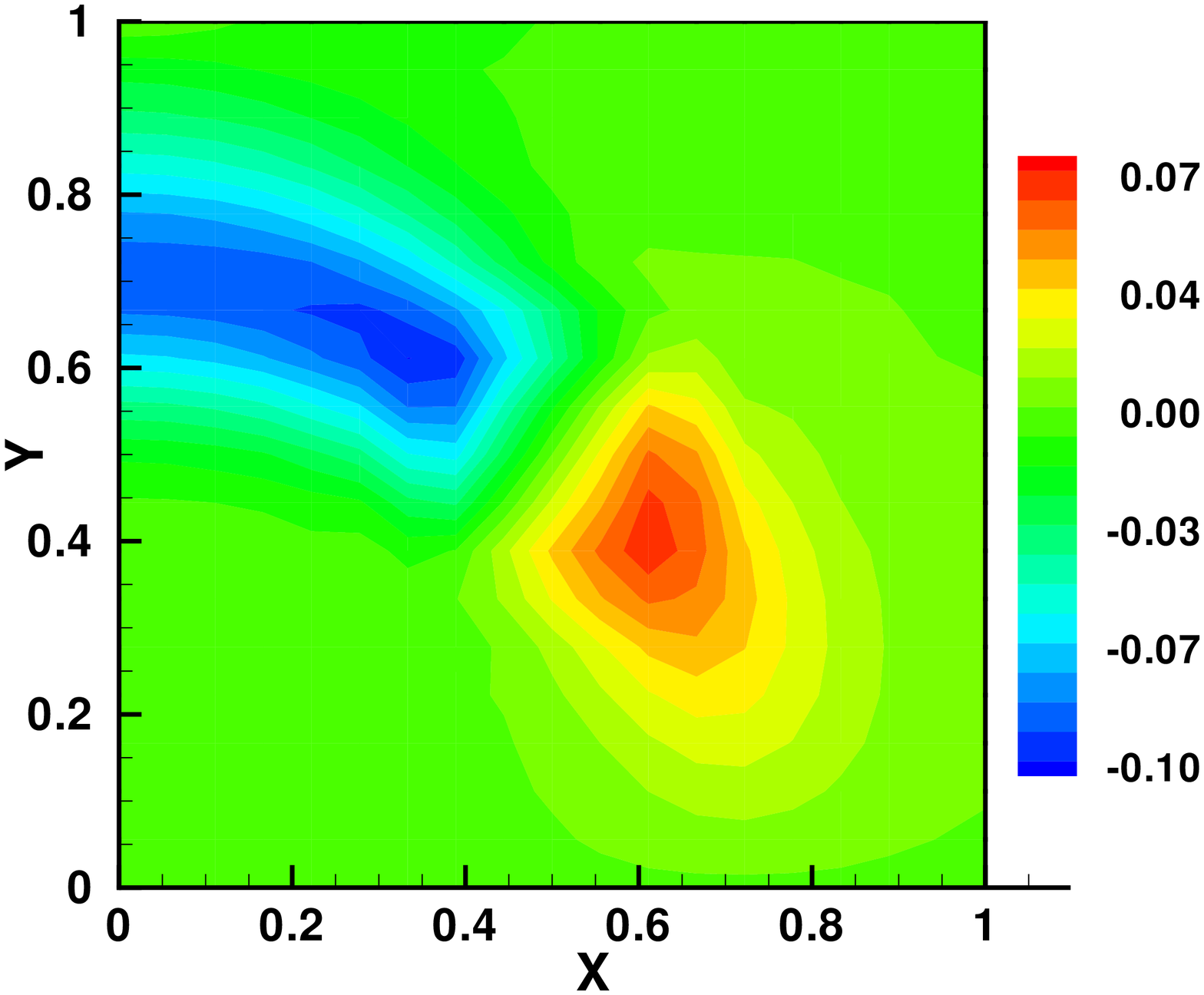}}
  \subfigure{
    \includegraphics[scale=0.32]{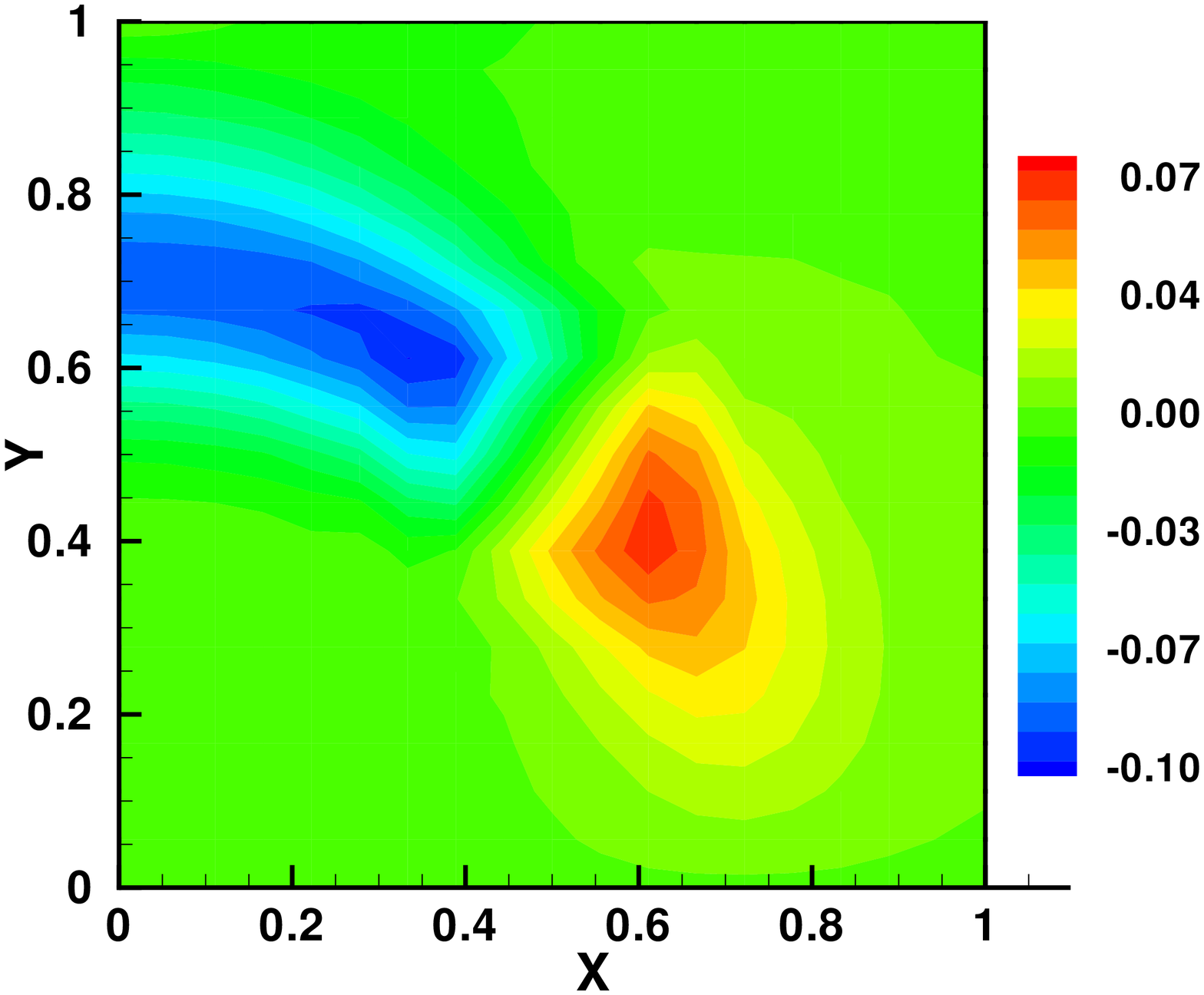}}
 \subfigure{
    \includegraphics[scale=0.32]{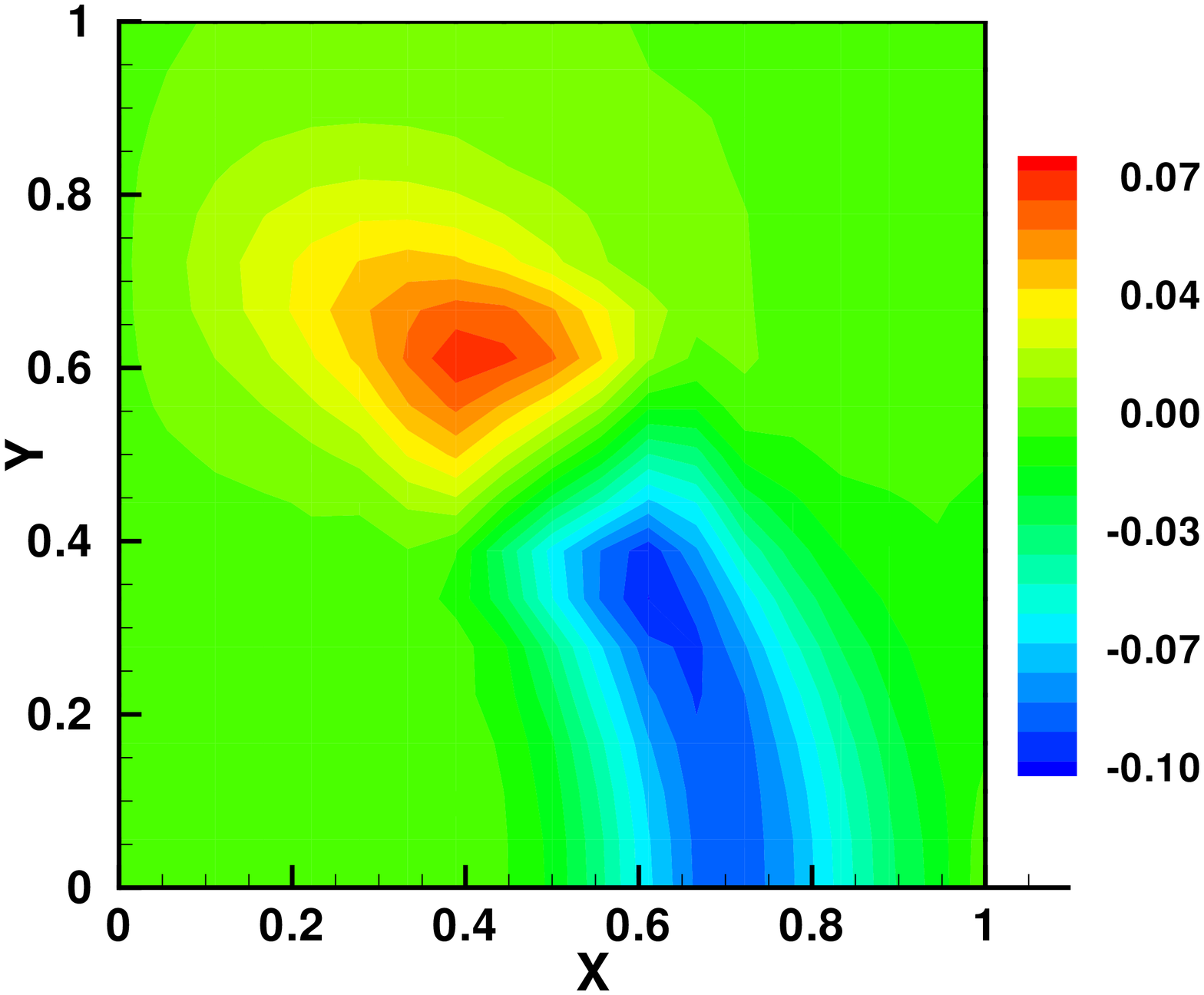}}
  \subfigure{
    \includegraphics[scale=0.32]{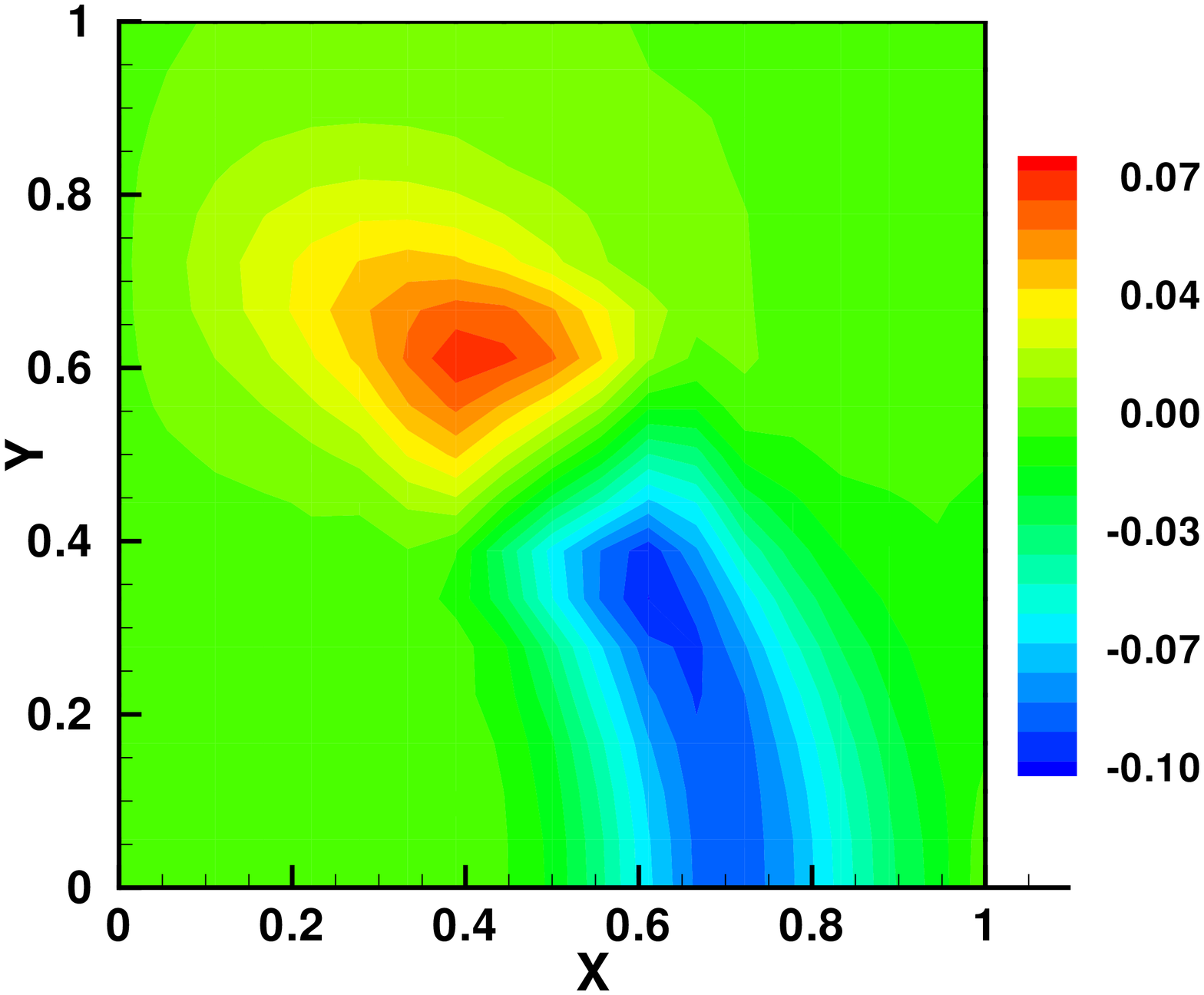}}
  \caption{Test problem \#1: Contours of the components of the vector field $\mathbf{v}$ 
    obtained using the variational multiscale (left) and corresponding optimization-based 
    (right) formulations. The top figures show the $x$-component (that is, $v_x$), and the 
    bottom figures the $y$-component of $\mathbf{v}$. Four-node quadrilateral mesh is 
    used in the numerical simulation. \label{Fig:Optim_HVM_velocity_Problem_2}}
\end{figure}

\begin{figure}[htbp]
  \subfigure{
    \includegraphics[scale=0.32]{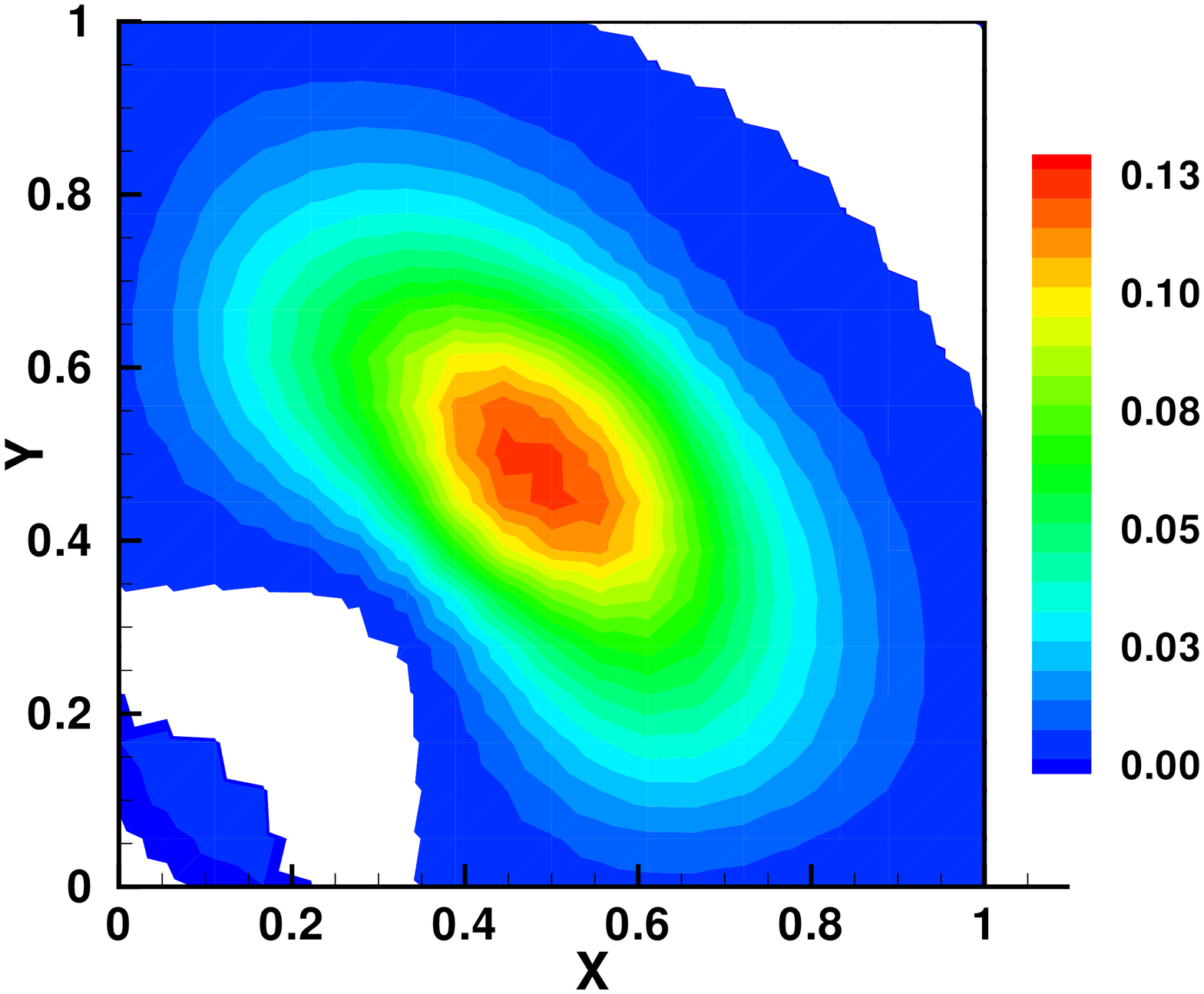}}
  \subfigure{
    \includegraphics[scale=0.32]{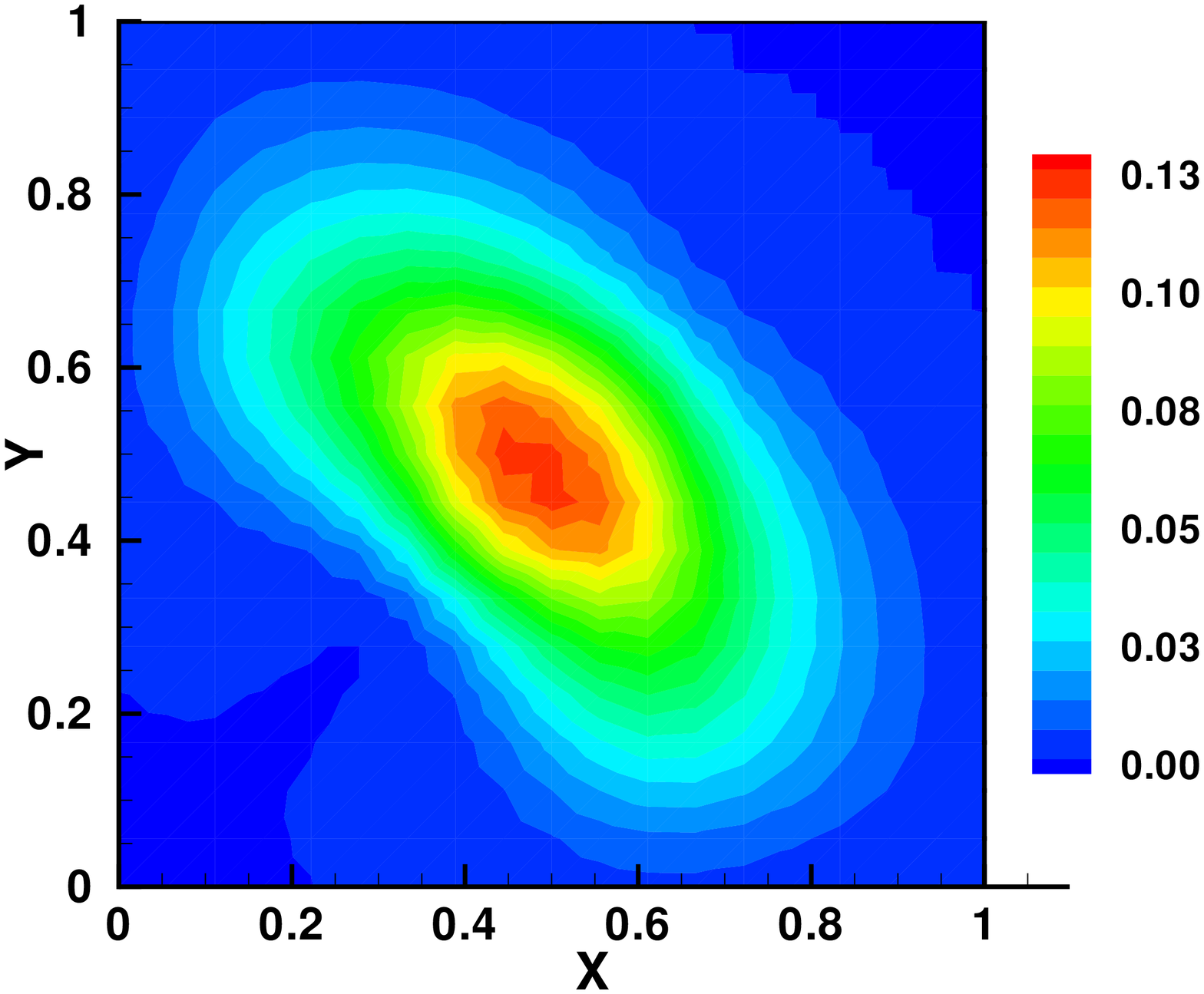}}
 \subfigure{
    \includegraphics[scale=0.32]{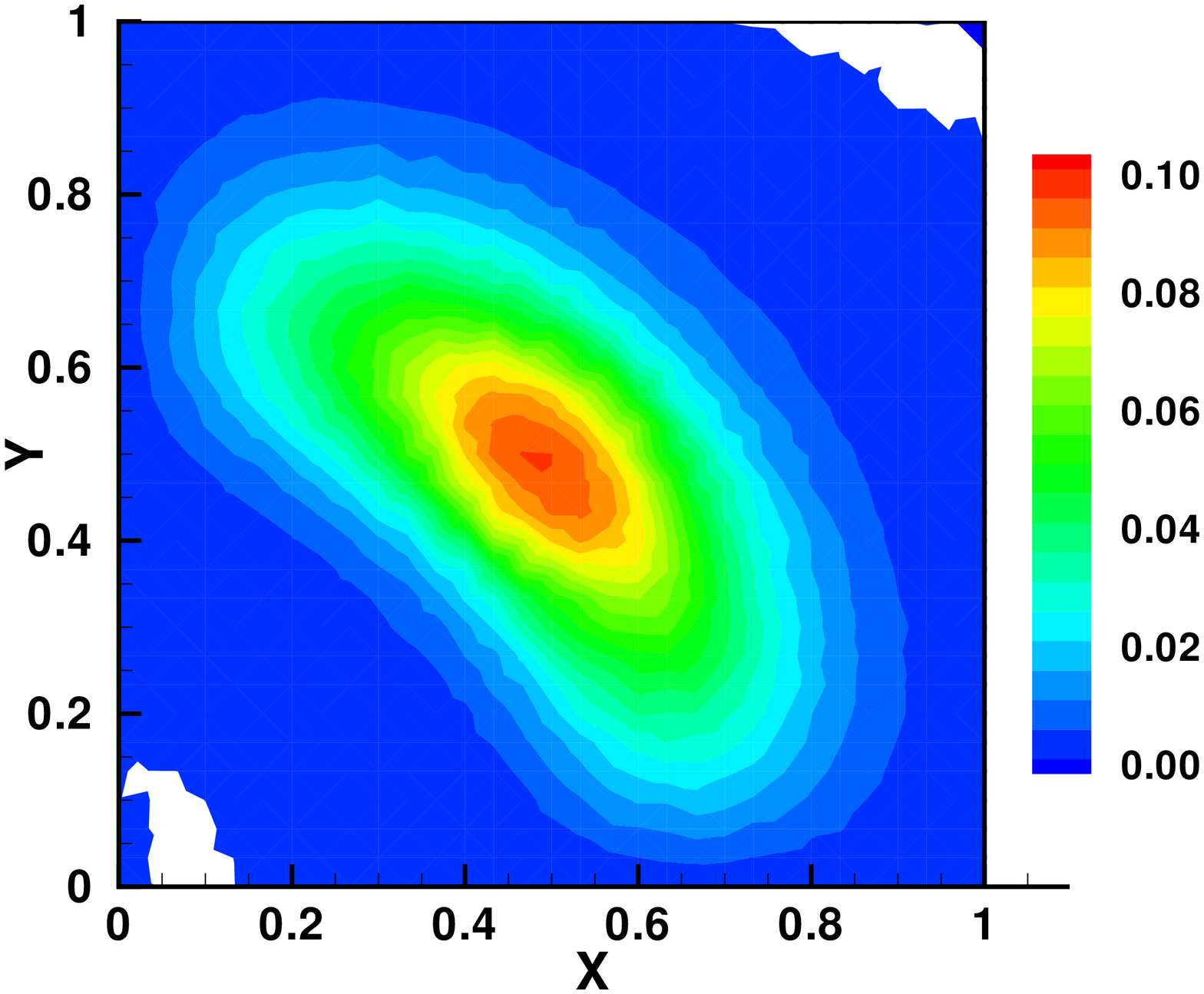}}
  \subfigure{
    \includegraphics[scale=0.32]{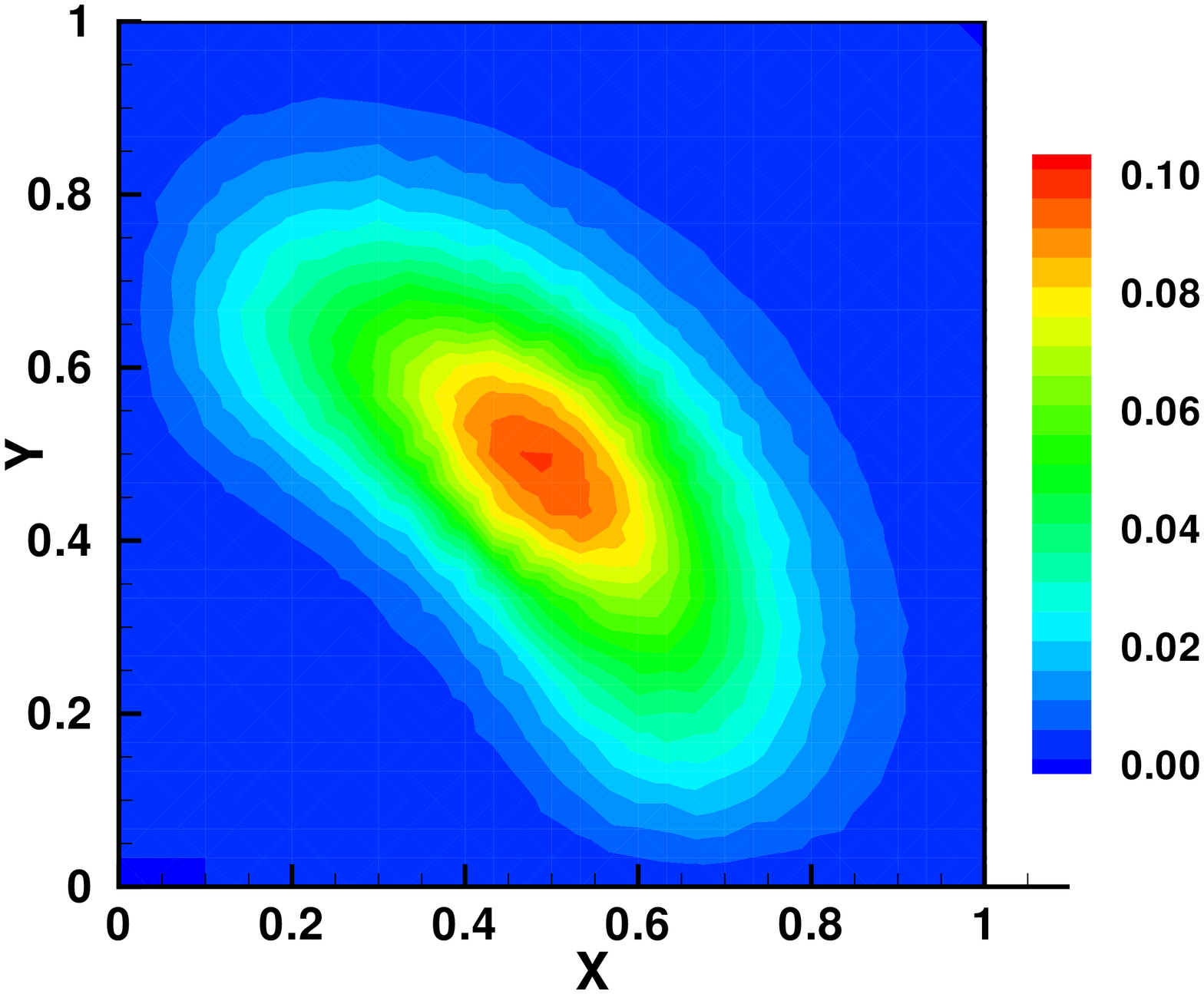}}
  \caption{Test problem \#1: Performance of the RT0 triangular (left) and corresponding optimization-based 
    (right) formulations. In the numerical simulations +45-degree (top) and  Delaunay (bottom)  meshes are 
    employed. The regions that have negative concentration are indicated in white color. \label{Fig:Optim_RT0_Problem_2}}
\end{figure}


\begin{figure}[htbp]
 \subfigure{
    \includegraphics[scale=0.32]{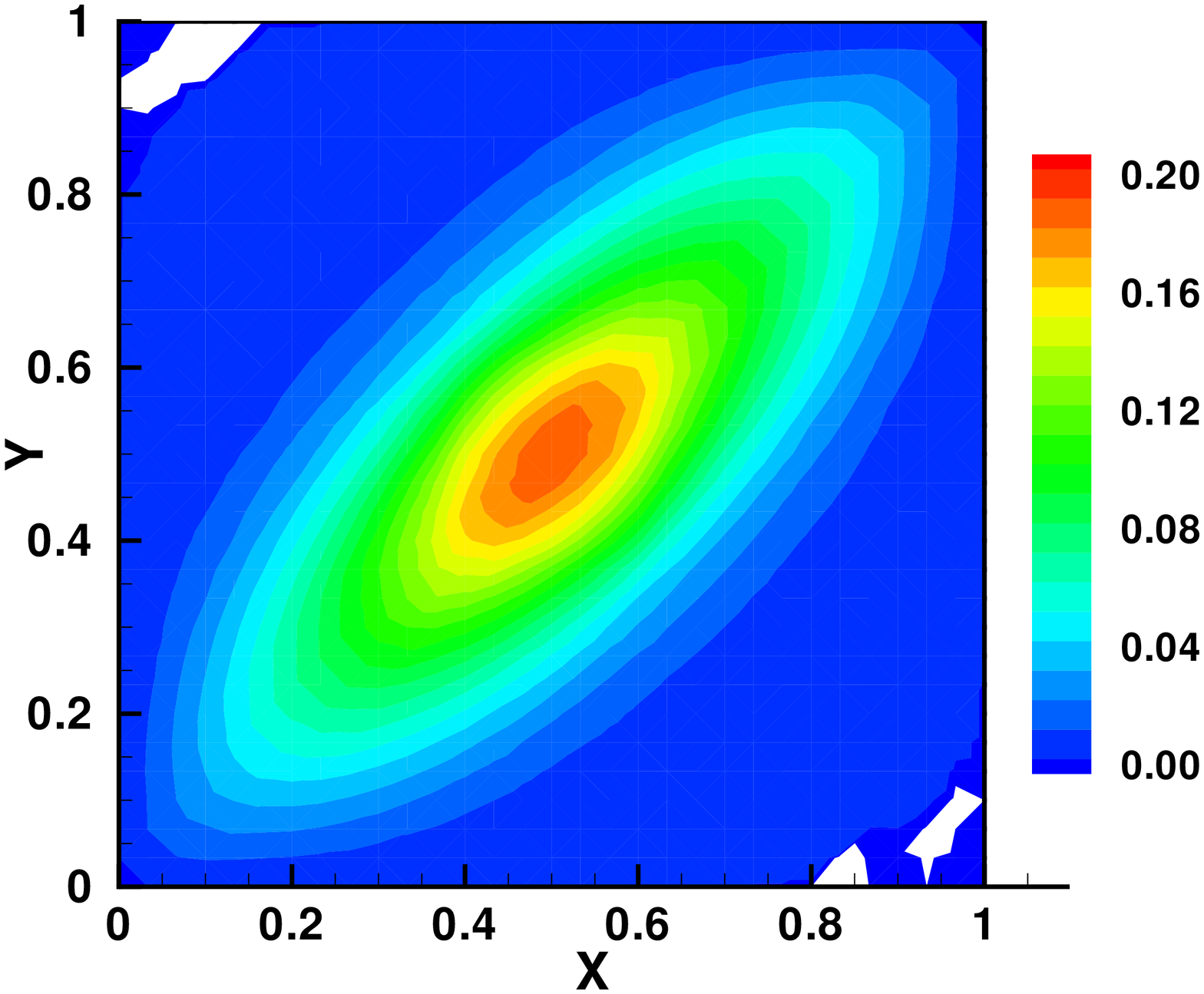}}
  \subfigure{
    \includegraphics[scale=0.32]{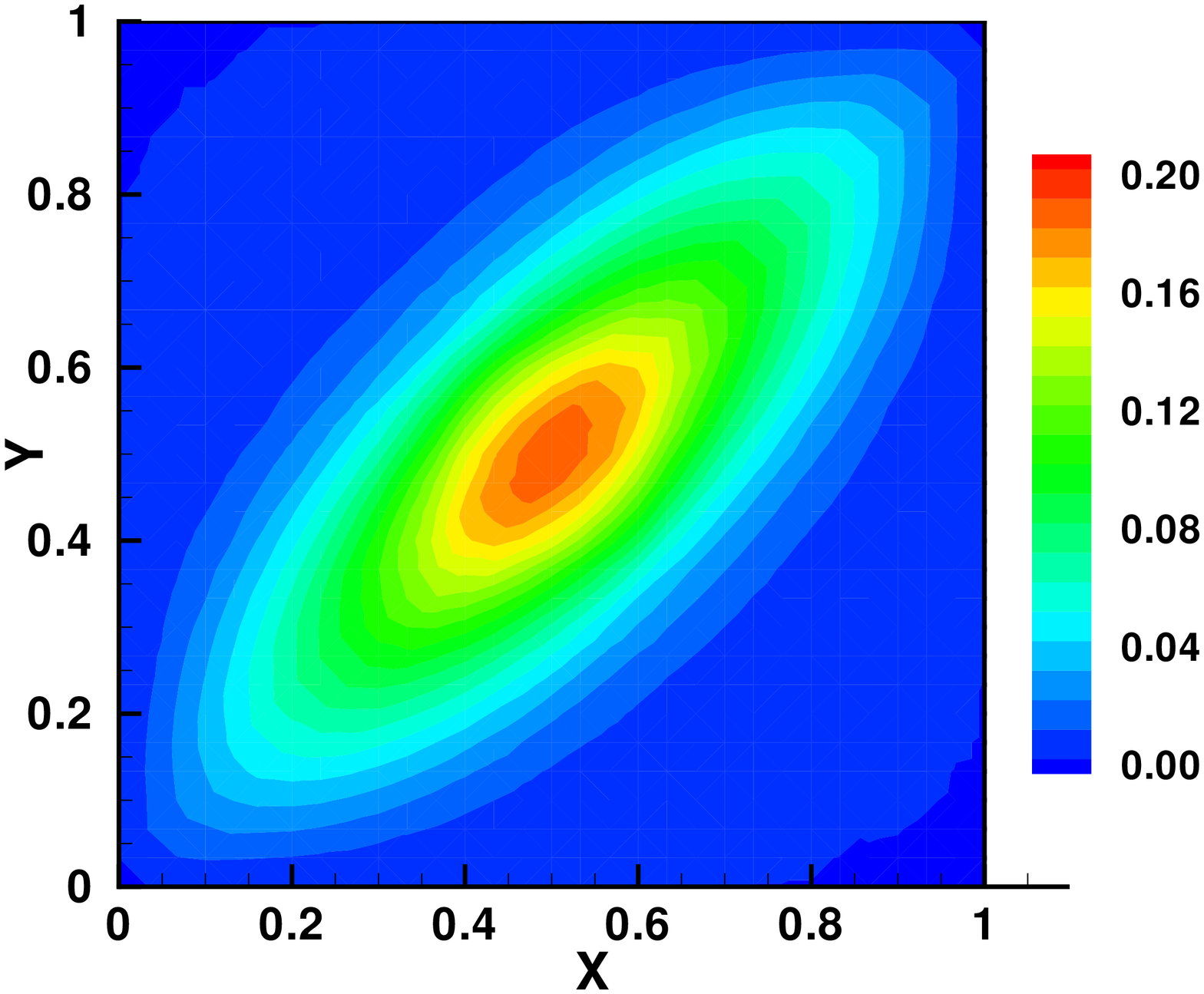}}
  \subfigure{
    \includegraphics[scale=0.32]{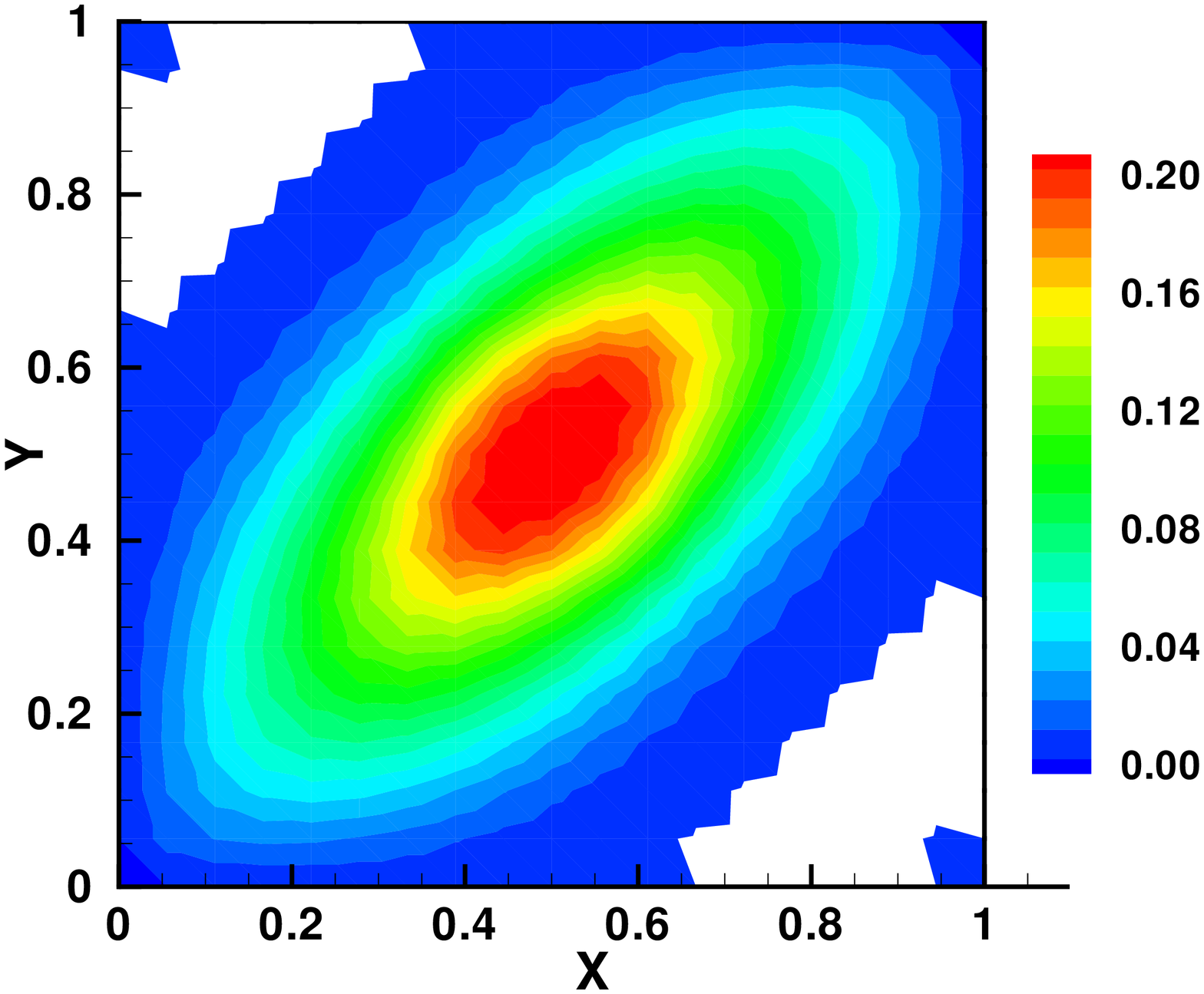}}
  \subfigure{
    \includegraphics[scale=0.32]{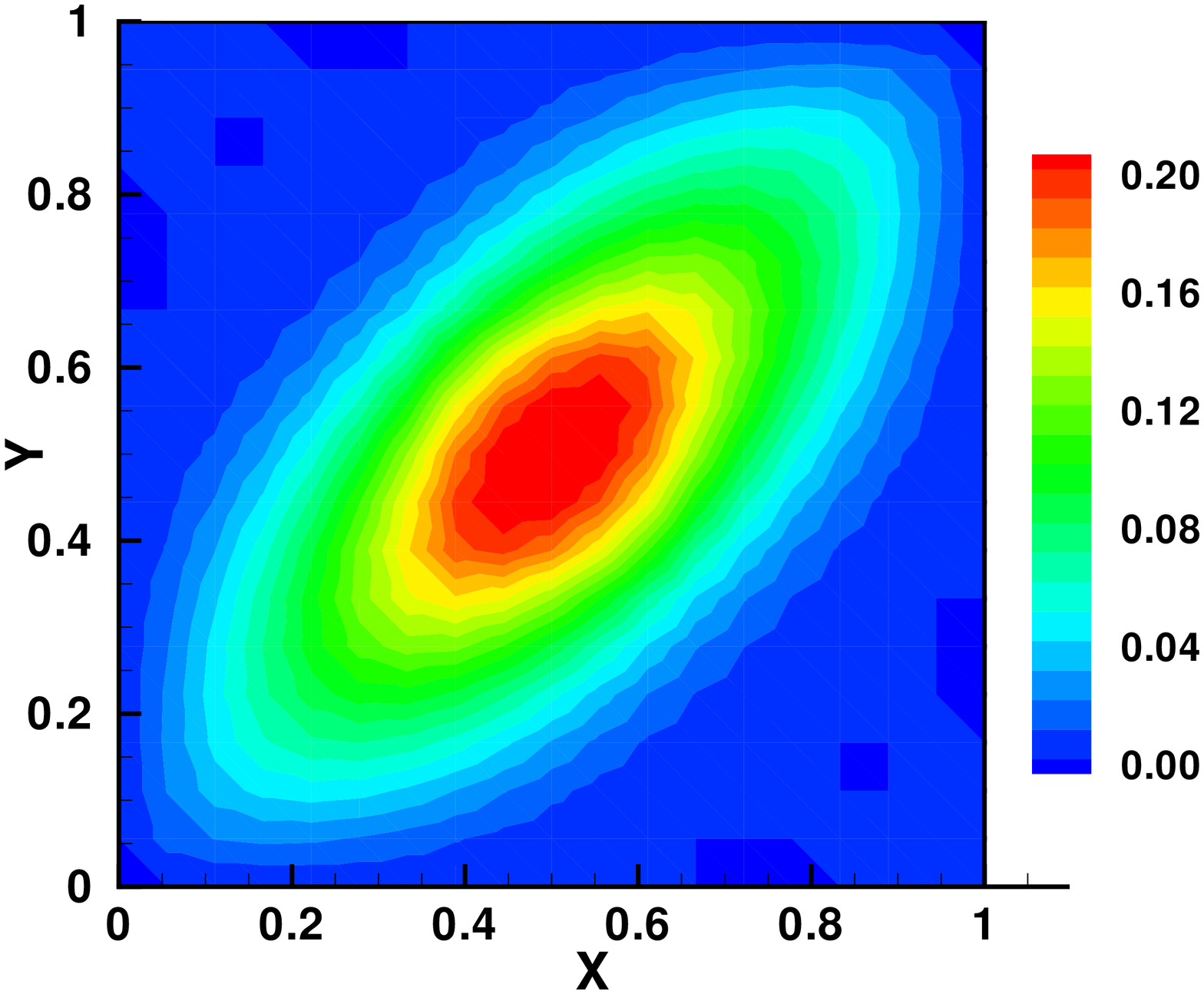}}
  \subfigure{
    \includegraphics[scale=0.32]{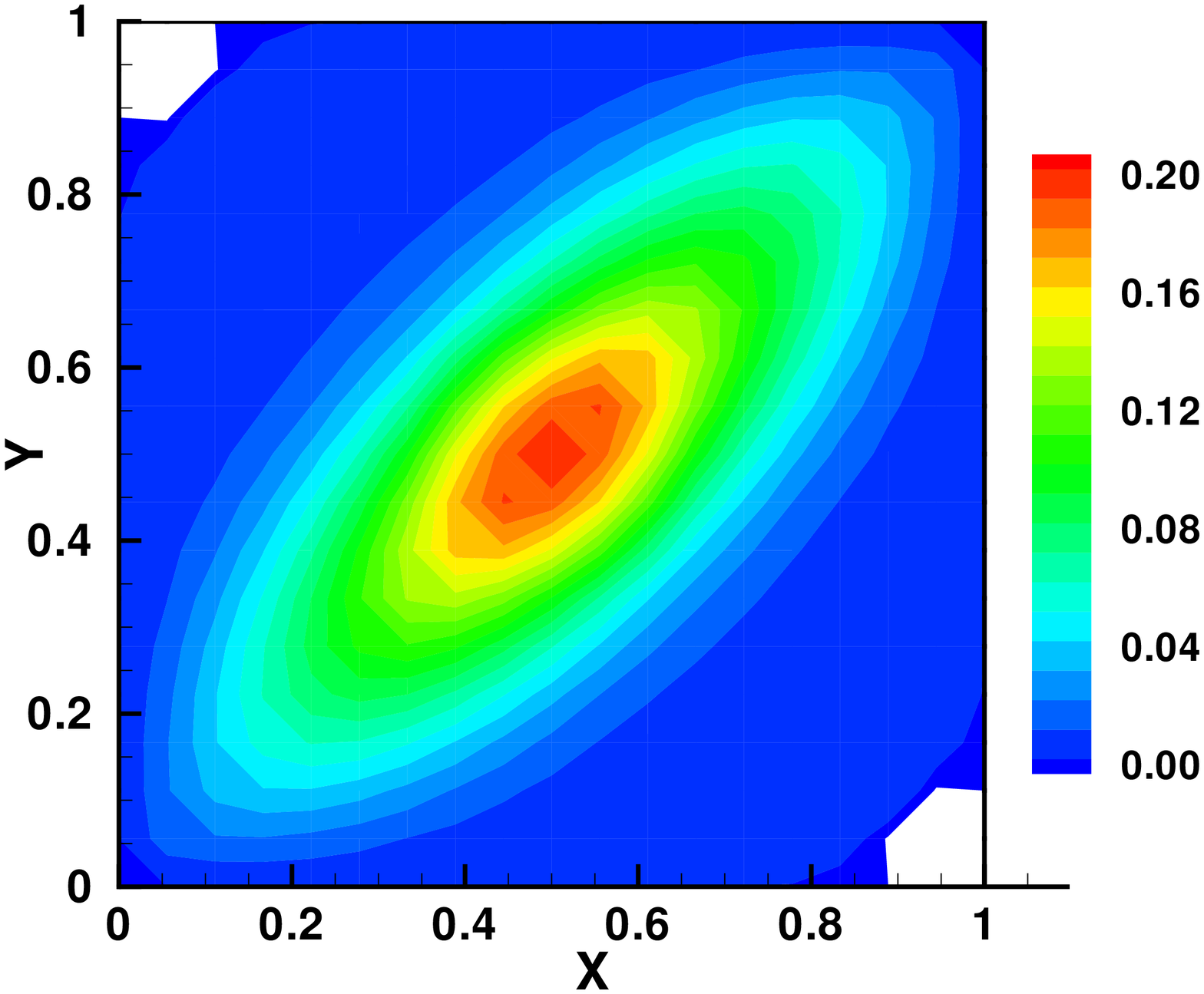}}
  \subfigure{
    \includegraphics[scale=0.32]{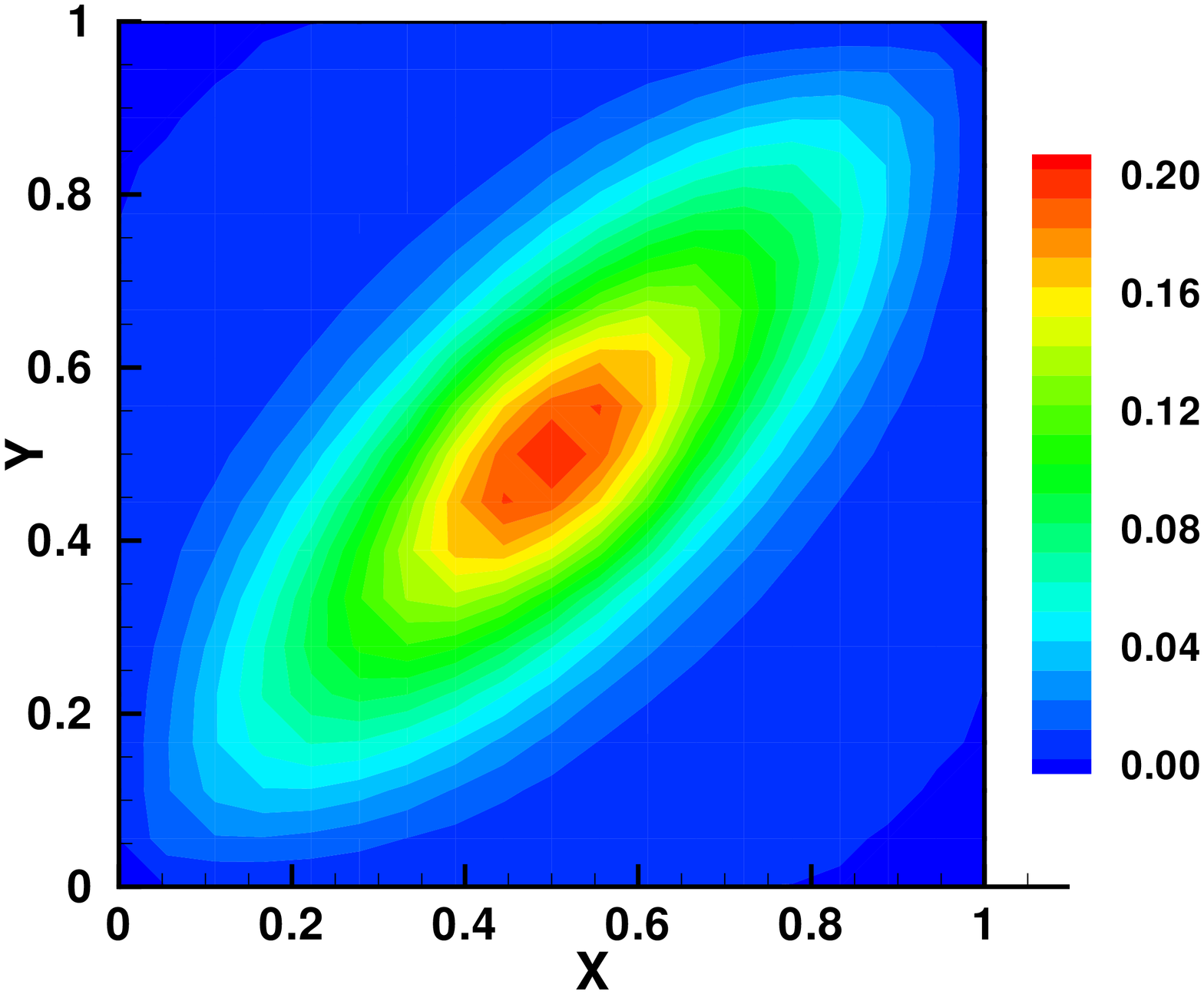}}
  \caption{Test problem \#2: Performance of the variational multiscale (left) and corresponding 
    optimization-based (right) formulations. In the numerical simulations Delaunay (top), $-45$-degree 
    (middle) and four-node quadrilateral (bottom) meshes are employed. The regions that have negative 
    concentration are marked as white. \label{Fig:Optim_HVM_Problem_3}}
\end{figure}

\begin{figure}[htbp]
  \subfigure{
    \includegraphics[scale=0.32]{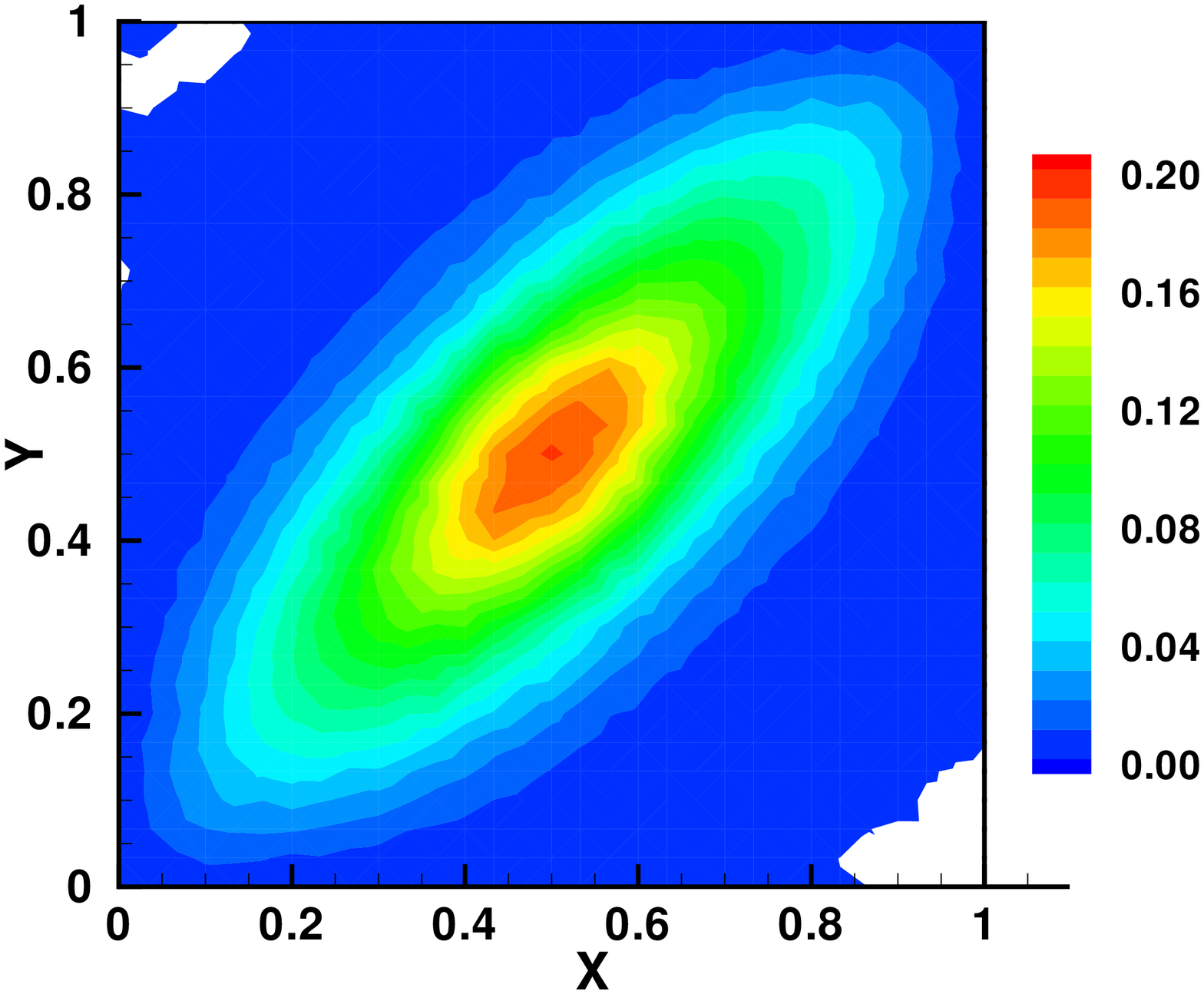}}
  \subfigure{
    \includegraphics[scale=0.32]{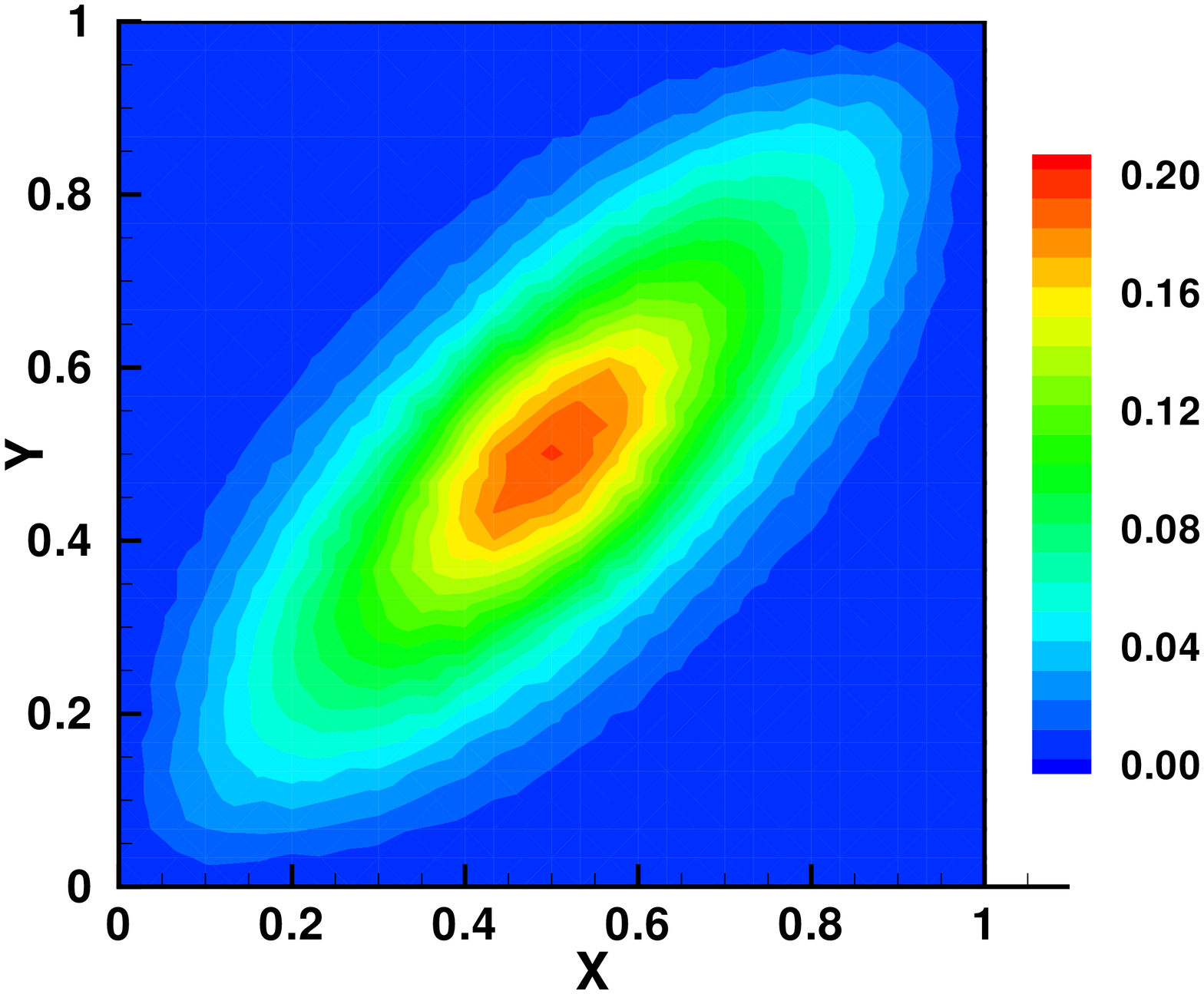}}
  \subfigure{
    \includegraphics[scale=0.32]{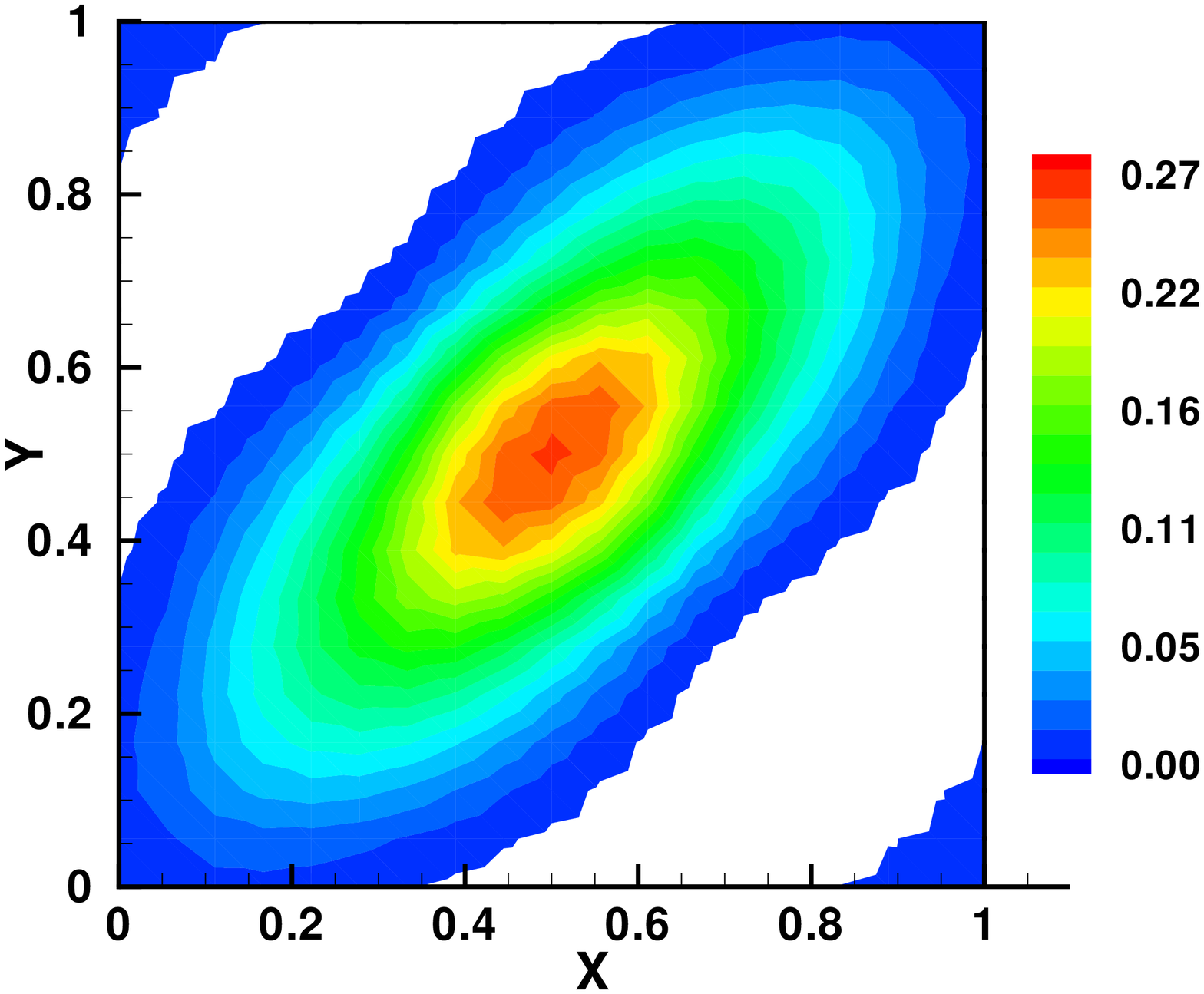}}
  \subfigure{
    \includegraphics[scale=0.32]{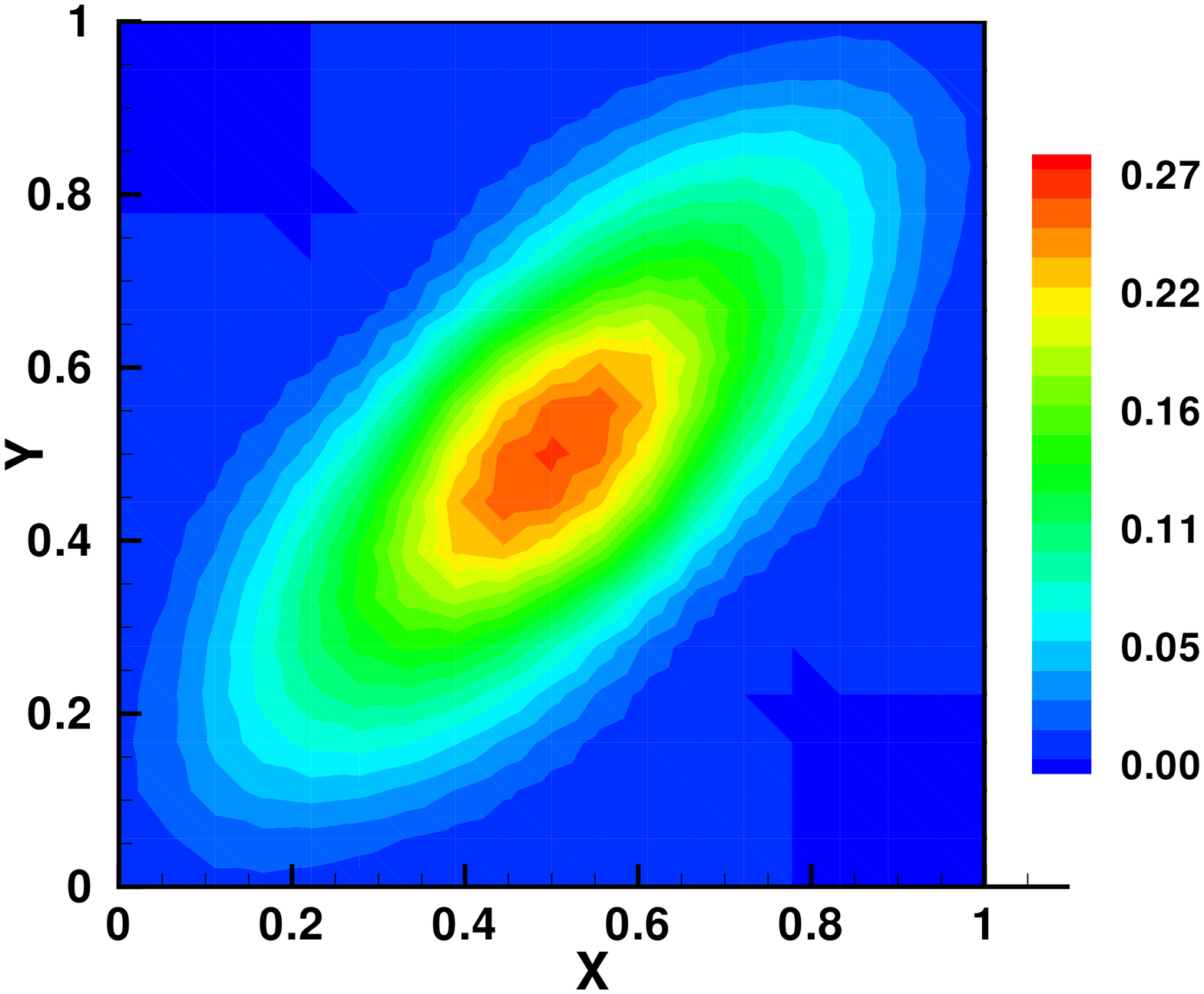}}
  \caption{Test problem \#2: Performance of the RT0 triangular (left) and corresponding optimization-based (right) 
    formulations. In the numerical simulations Delaunay (top) and $-45$-degree (bottom) meshes are employed. The regions 
    that have negative concentration are indicated in white color. \label{Fig:Optim_RT0_Problem_3}}
\end{figure}


\begin{figure}
  \subfigure{
    \includegraphics[scale=0.32]{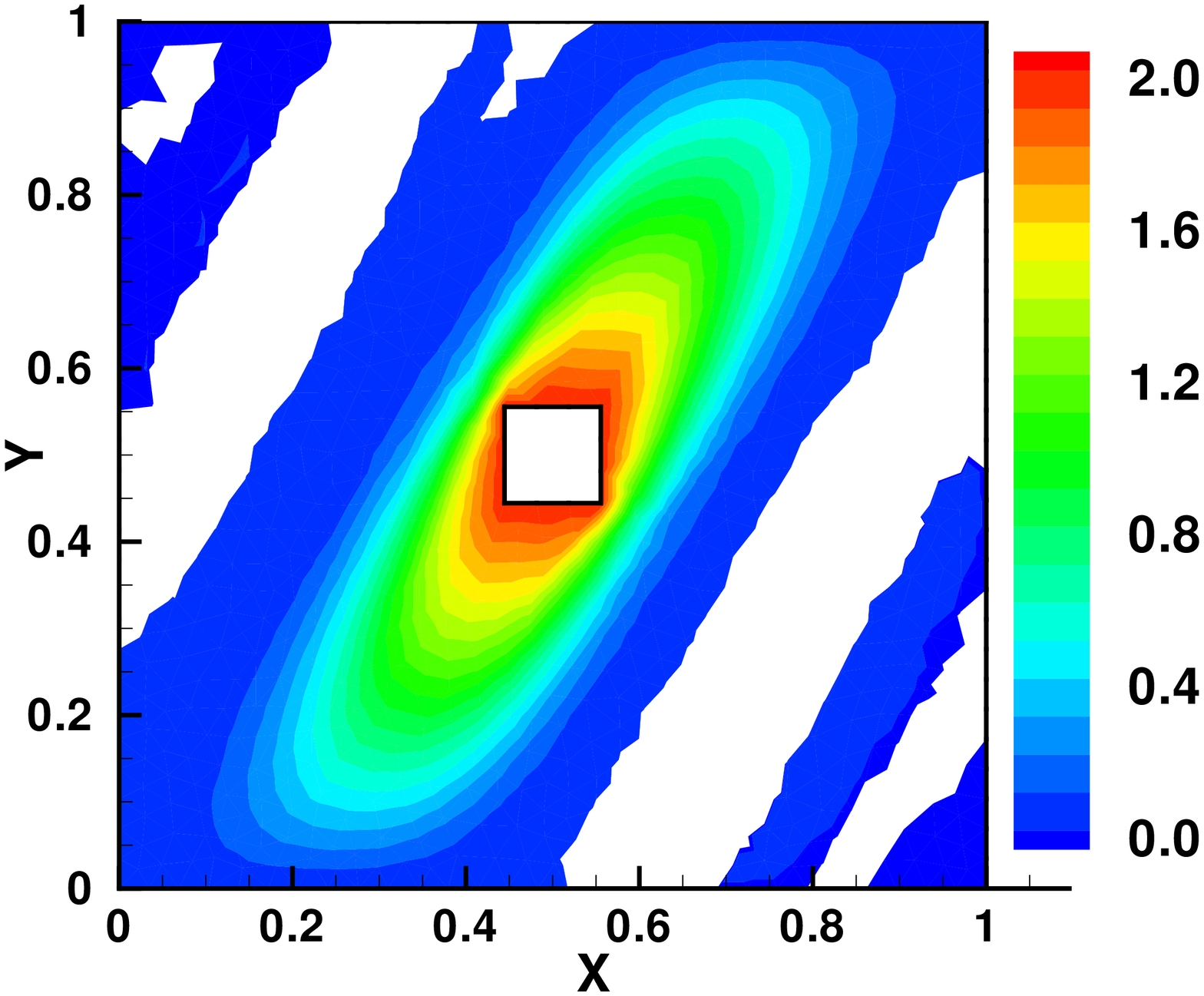}}
  \subfigure{
    \includegraphics[scale=0.32]{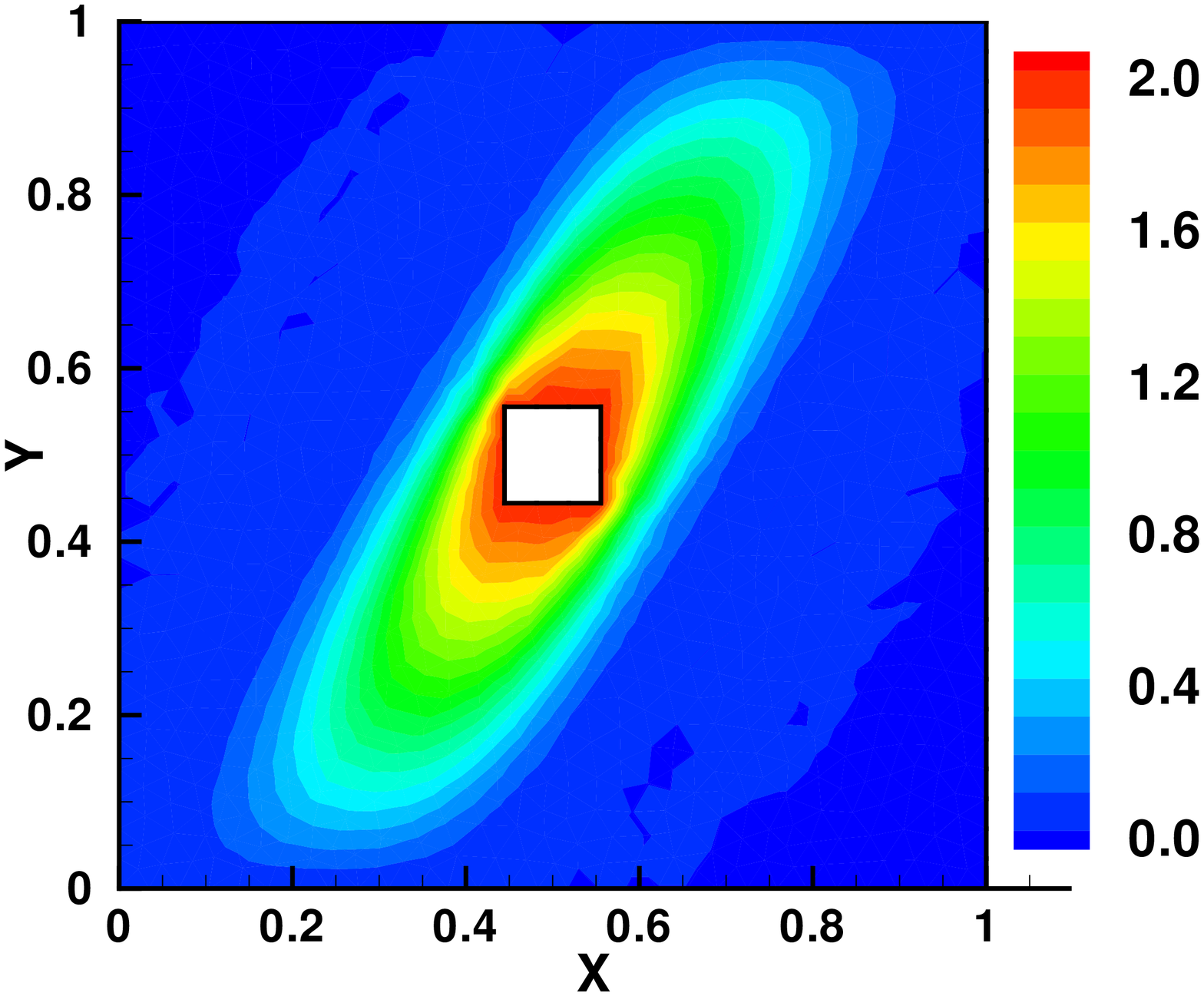}}
  \subfigure{
    \includegraphics[scale=0.32]{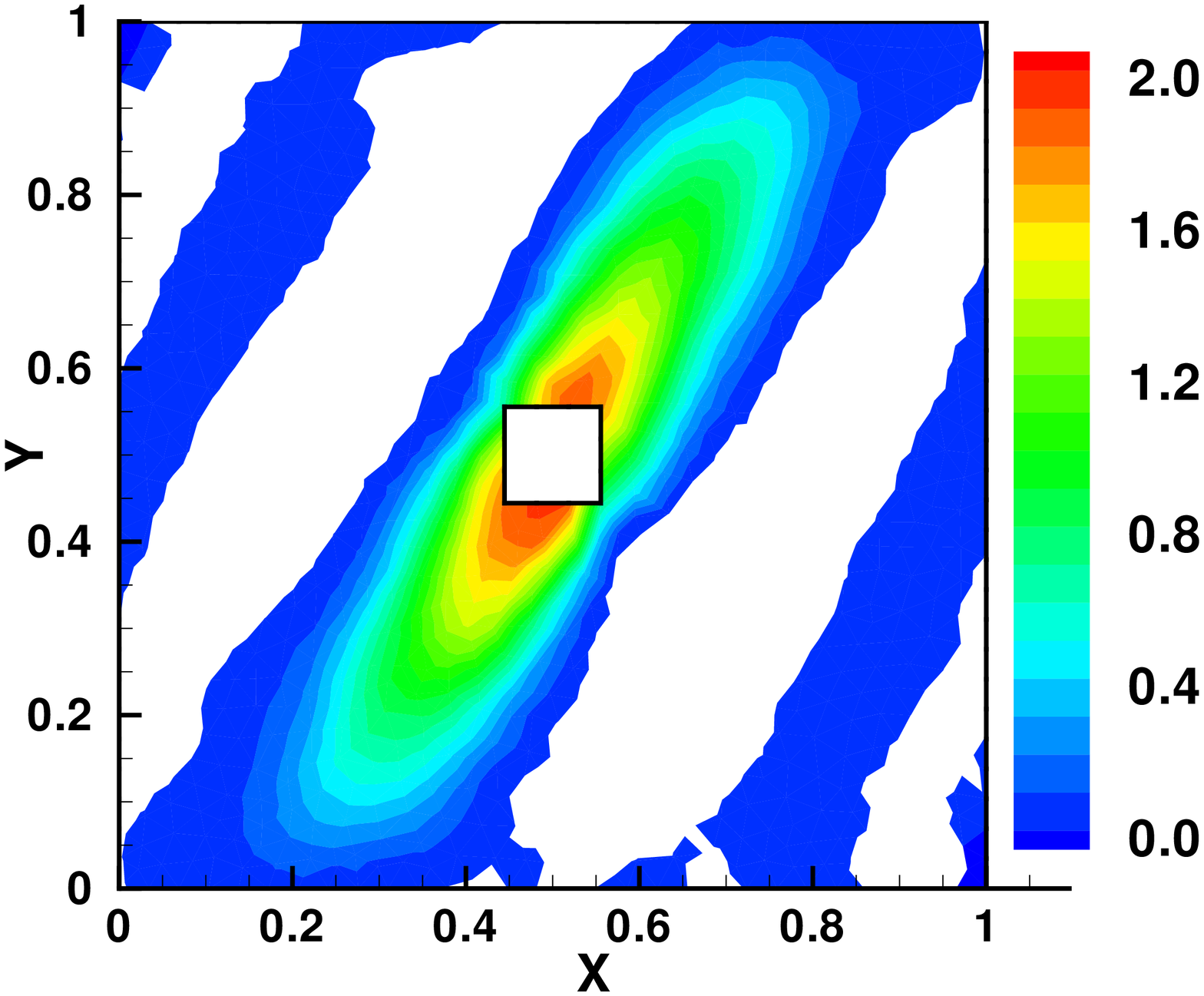}}
  \subfigure{
    \includegraphics[scale=0.32]{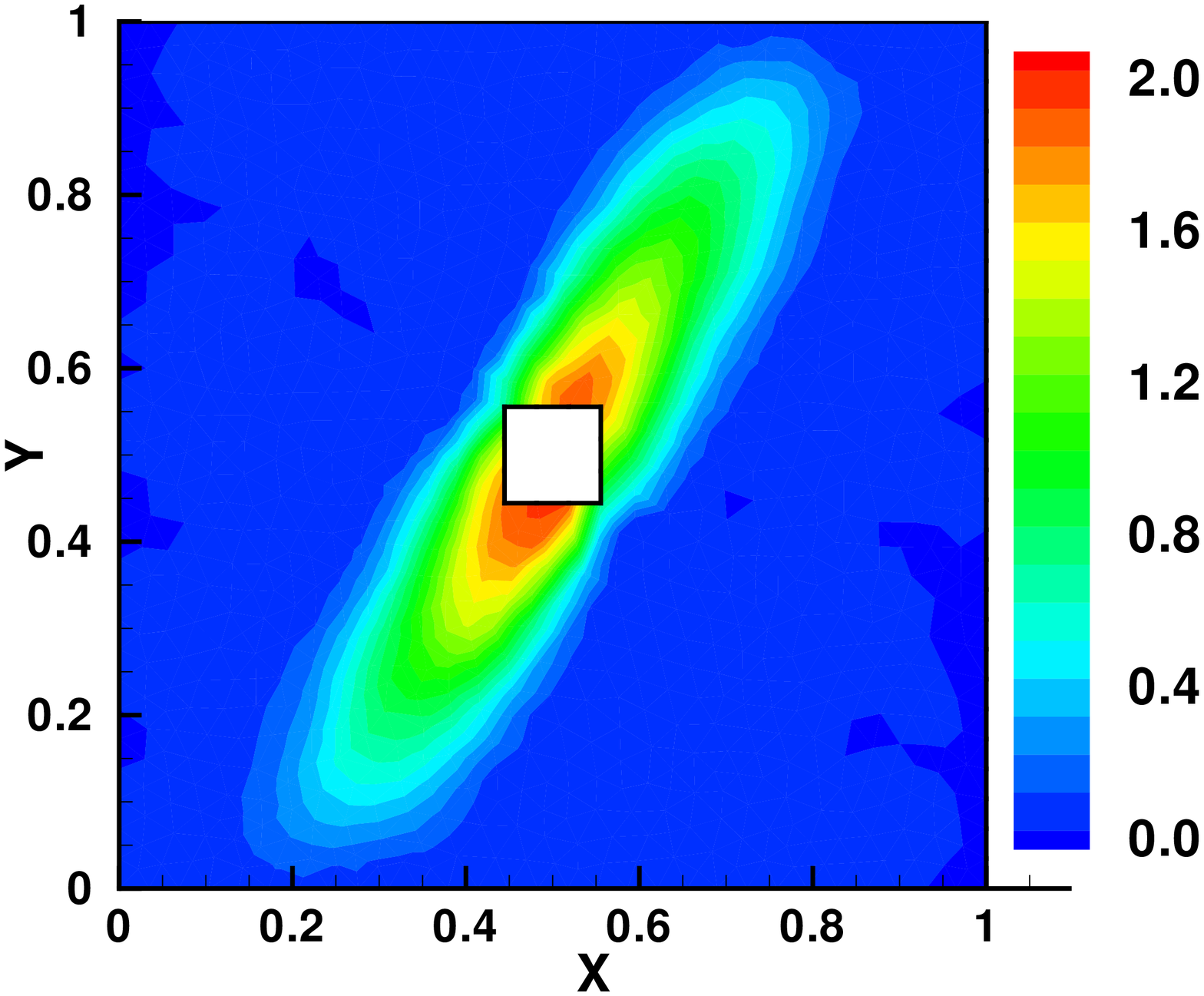}}
  \caption{Test problem \#3: Performance of the variational multiscale (top) and RT0 (bottom) formulations. 
    The figures on the right show the performance of corresponding optimization-based non-negative 
    formulations. As expected, the variational multiscale and RT0 formulations produced negative solutions 
    in significant portions of the domain, which are denoted by the white color. On the other hand, the 
    proposed optimization-based formulations produced desirable non-negative solutions. The computational 
    mesh that is used in these numerical simulations is shown in Figure \ref{Fig:Optim_mesh_problem_4}. 
    \label{Fig:Optim_Results_Problem_4}}
\end{figure}

\begin{figure}
  \subfigure{
    \includegraphics[scale=0.32]{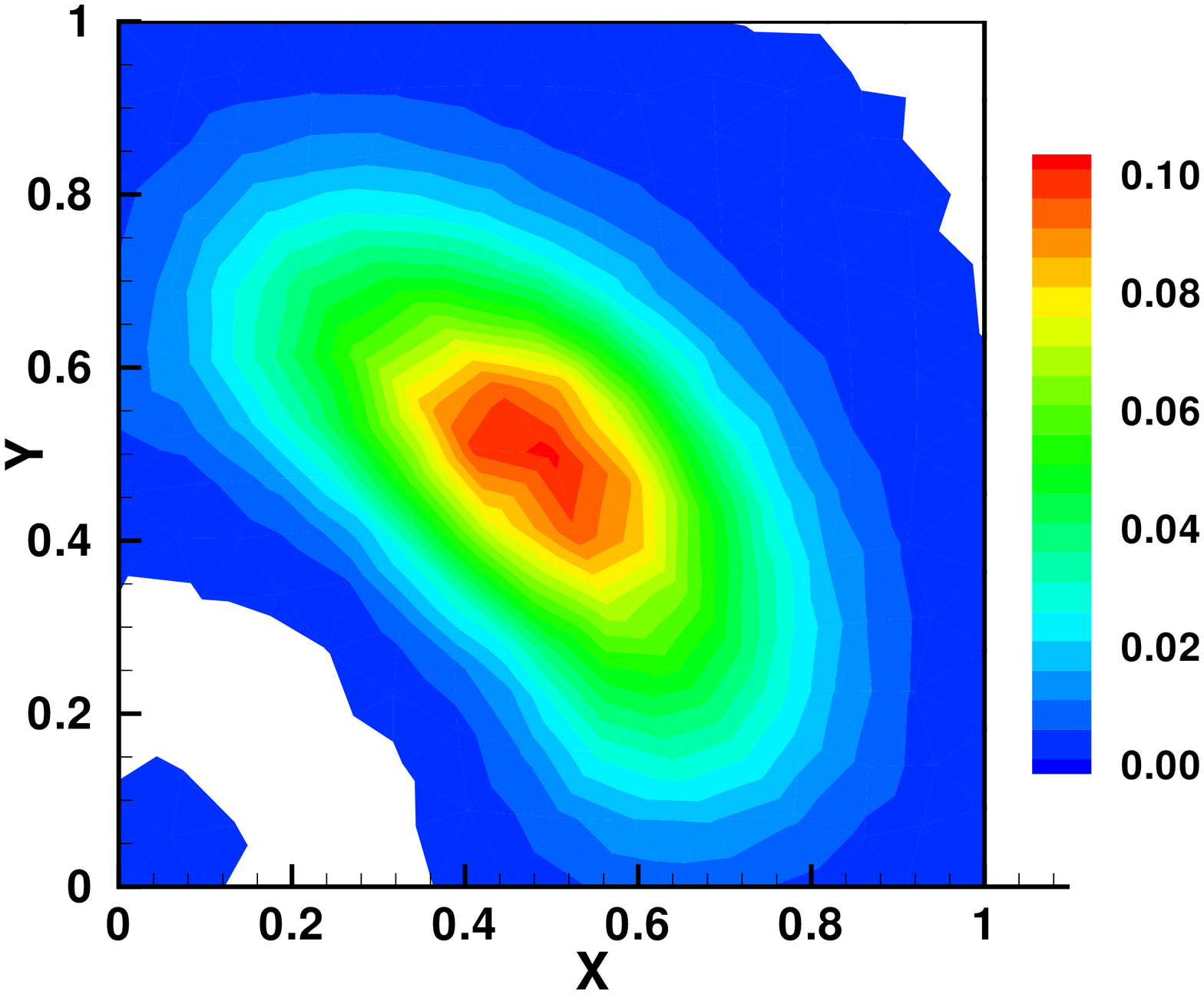}}
  \subfigure{
    \includegraphics[scale=0.32]{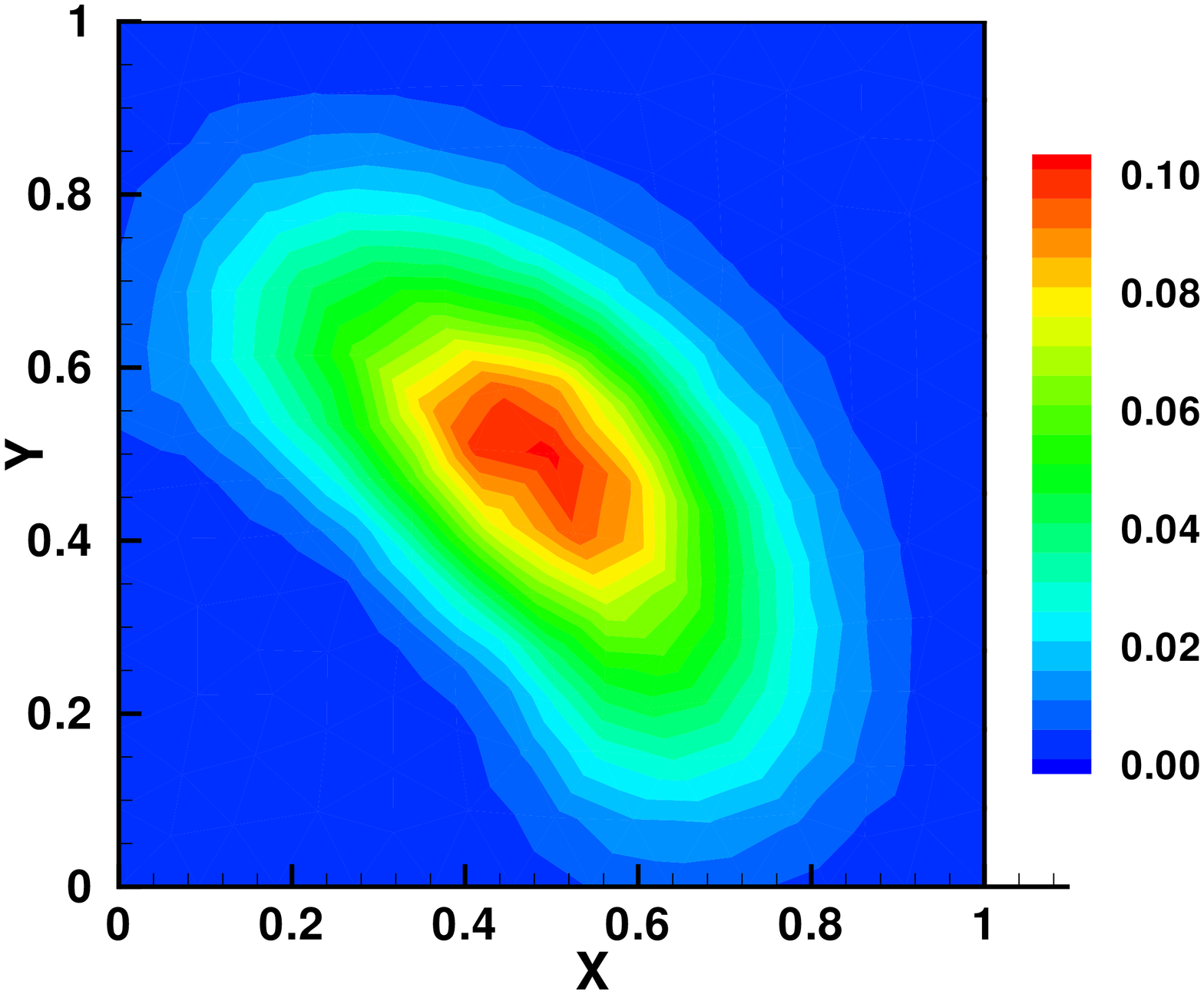}}
  \subfigure{
    \includegraphics[scale=0.32]{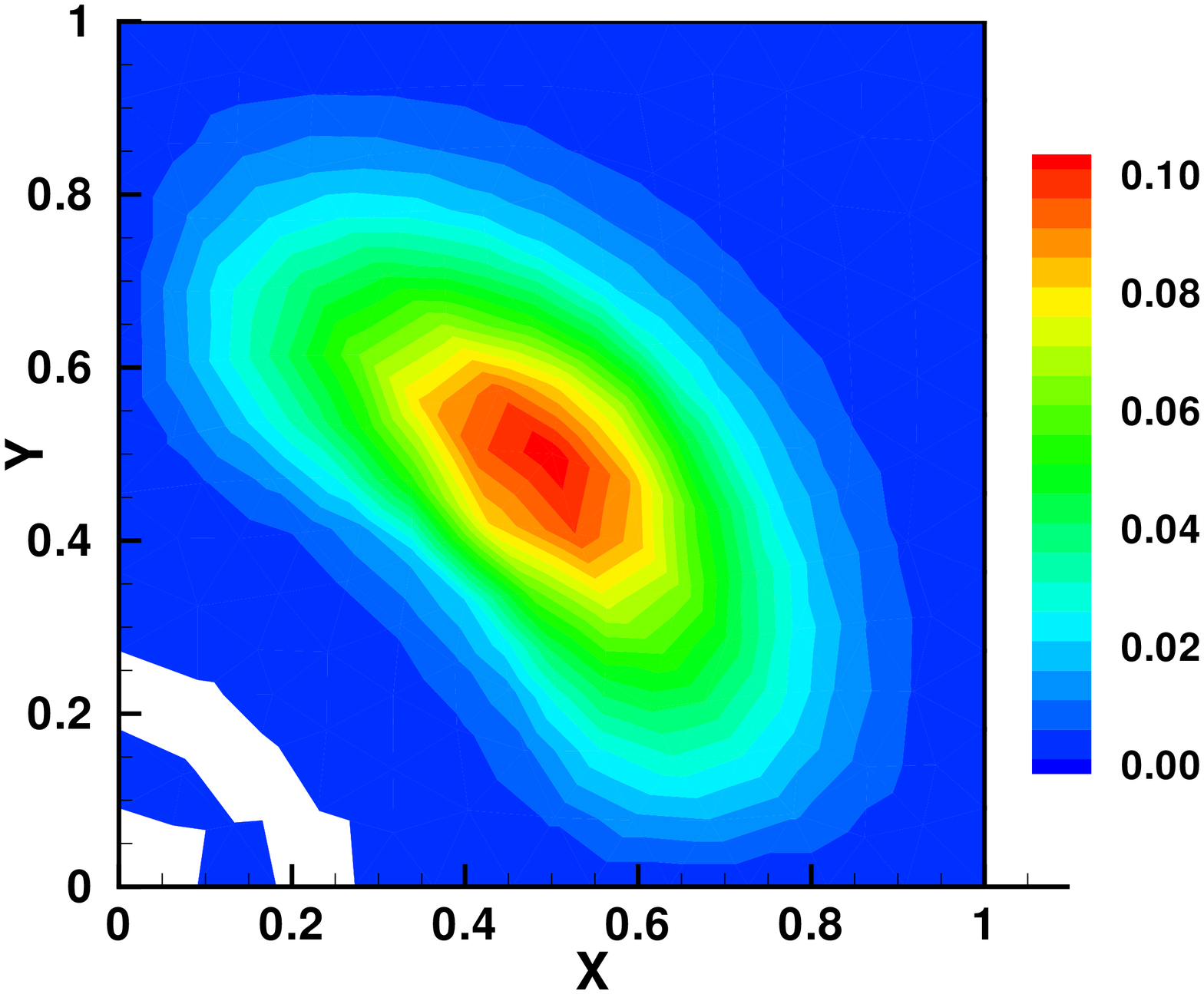}}
  \subfigure{
    \includegraphics[scale=0.32]{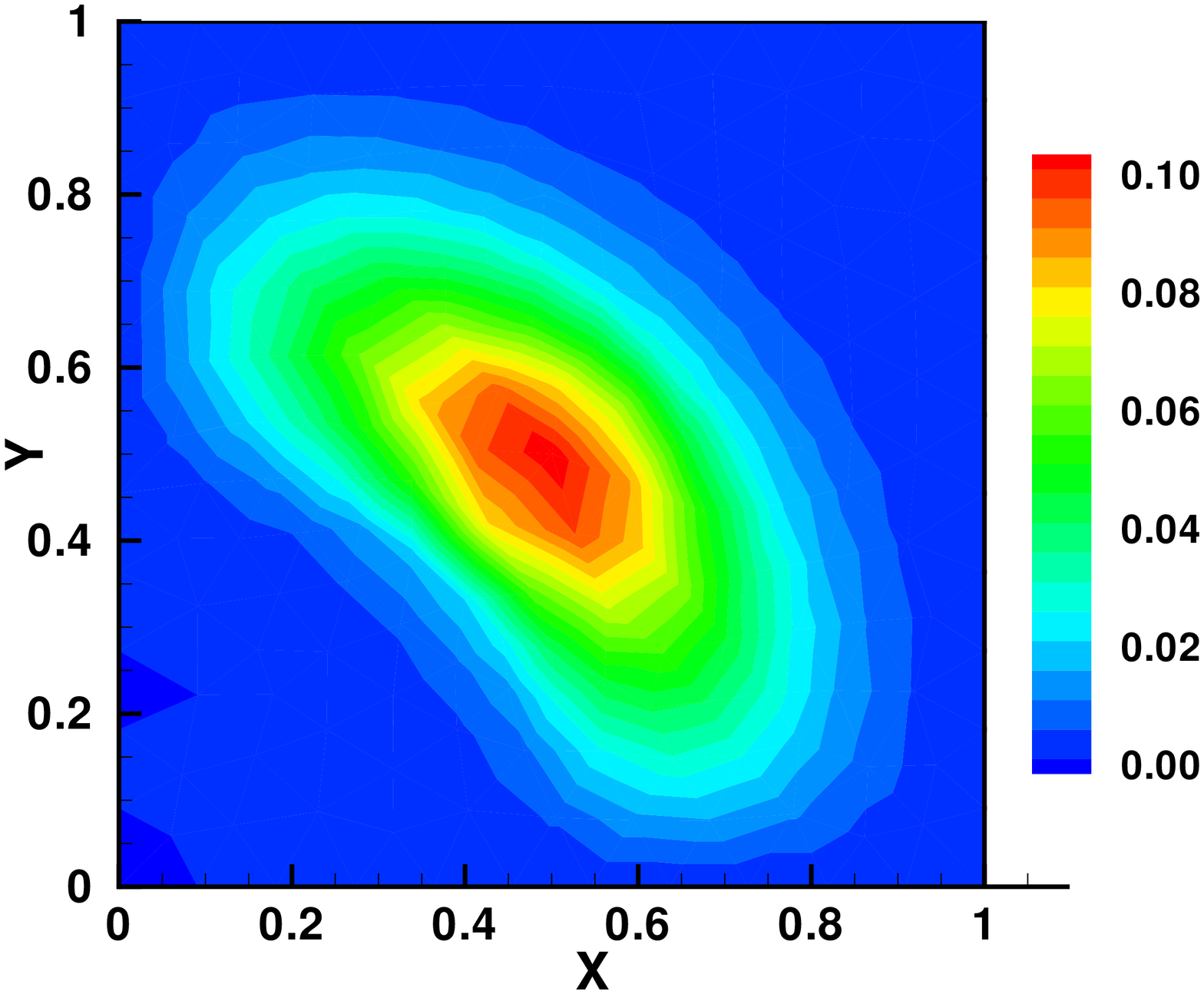}}
  \caption{Performance on a well-centered triangular (WCT) mesh: RT0 (top) and variational multiscale (bottom) 
    formulations on test problem \#1 and using the WCT mesh shown in Figure \ref{Fig:Optim_mesh_WCT}. The 
    left figures illustrates that, in the case of tensorial diffusivity coefficient, even a WCT mesh is not 
    sufficient for the RT0 and variational multiscale formulations to satisfy the discrete maximum-minimum 
    principle. The right figures shows that (as expected) the proposed non-negative optimization-based 
    formulation outlined in Sections \ref{Sec:Optim_RT} and \ref{Sec:Optim_VMS} produces non-negative 
    solutions. \label{Fig:Optim_HVM_WCTmesh_Problem_2}}
\end{figure}

\begin{figure}[htbp]
  \subfigure{
    \includegraphics[scale=0.36]{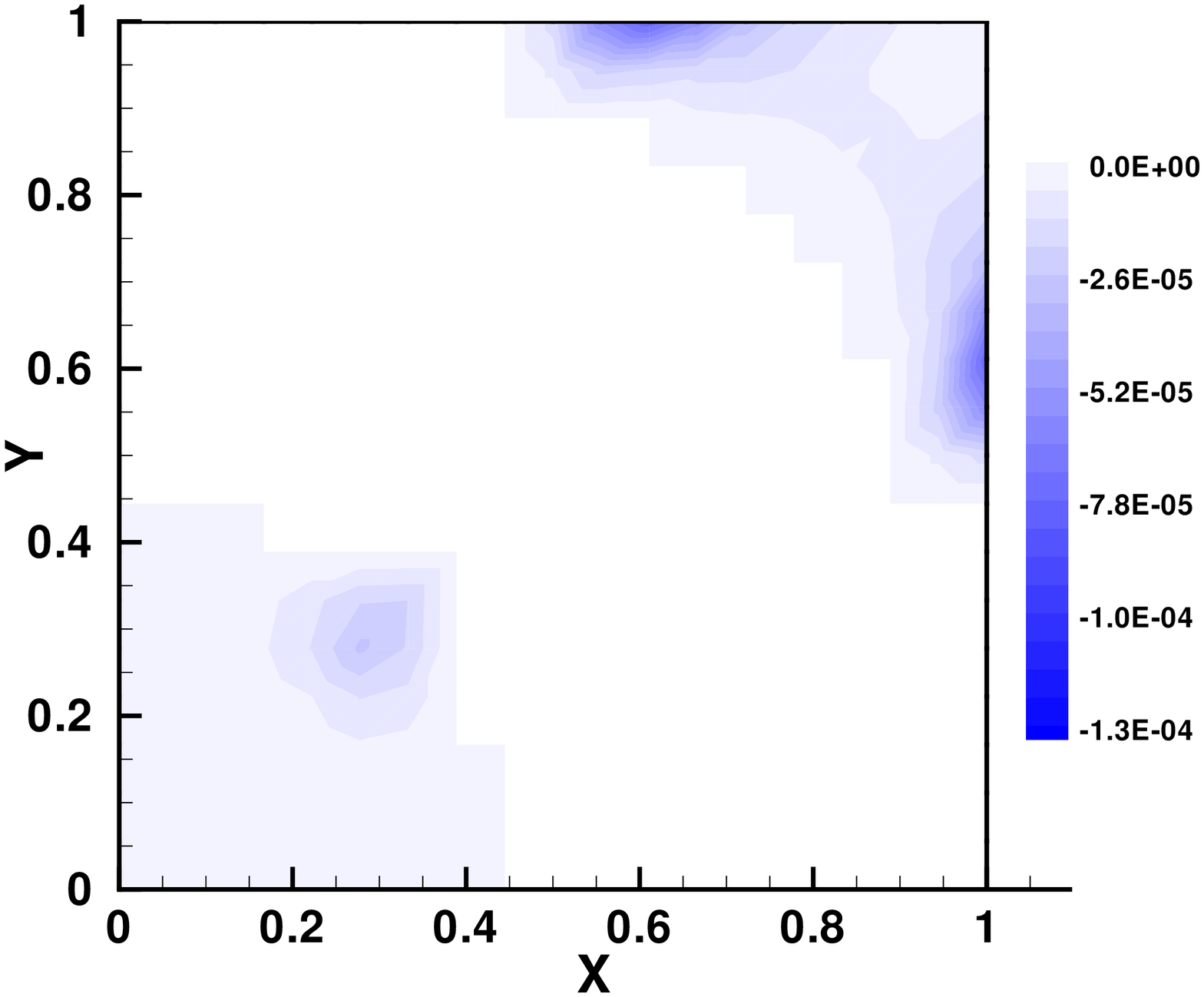}}
 \subfigure{
    \includegraphics[scale=0.36]{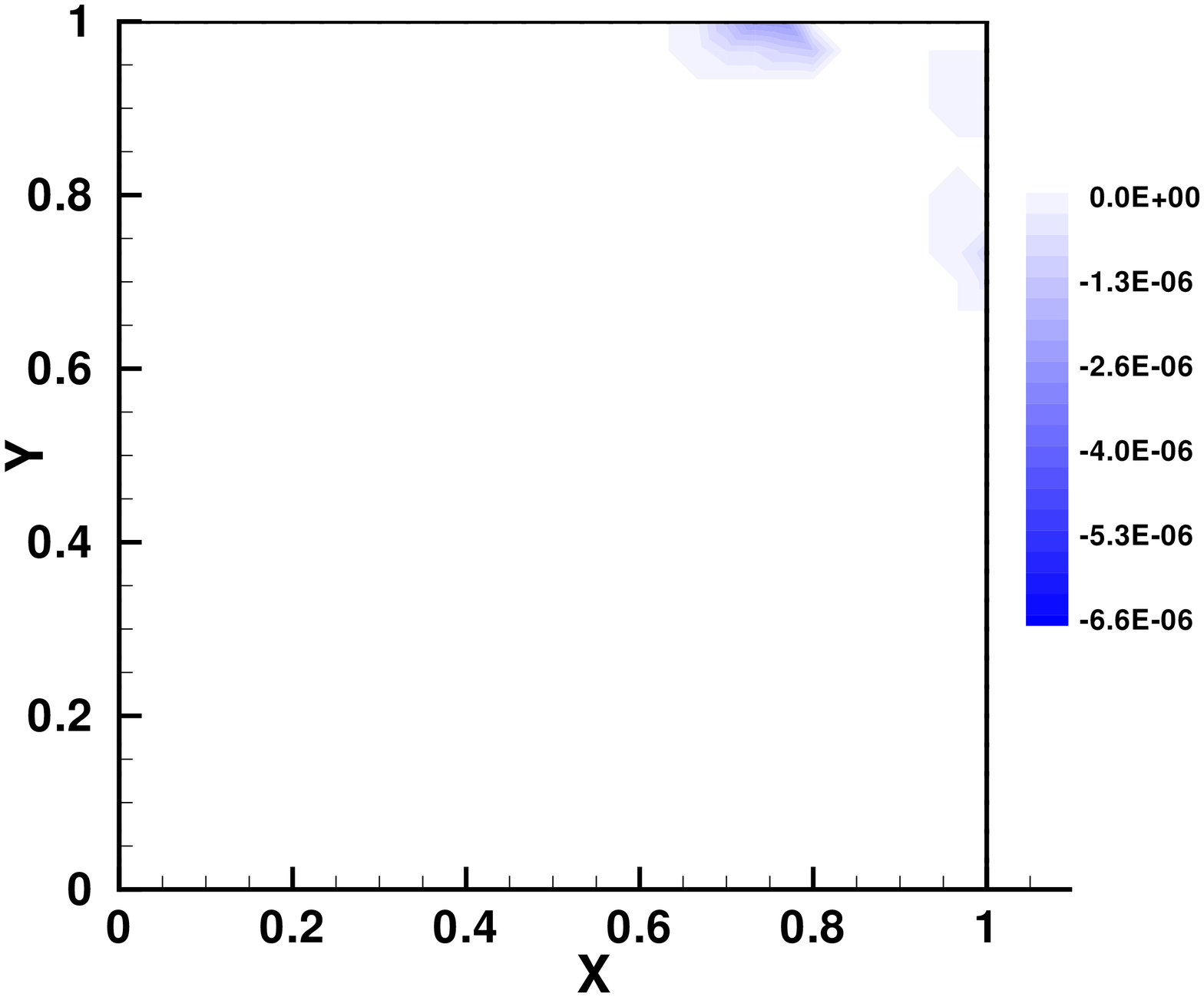}}
  \subfigure{
    \includegraphics[scale=0.36]{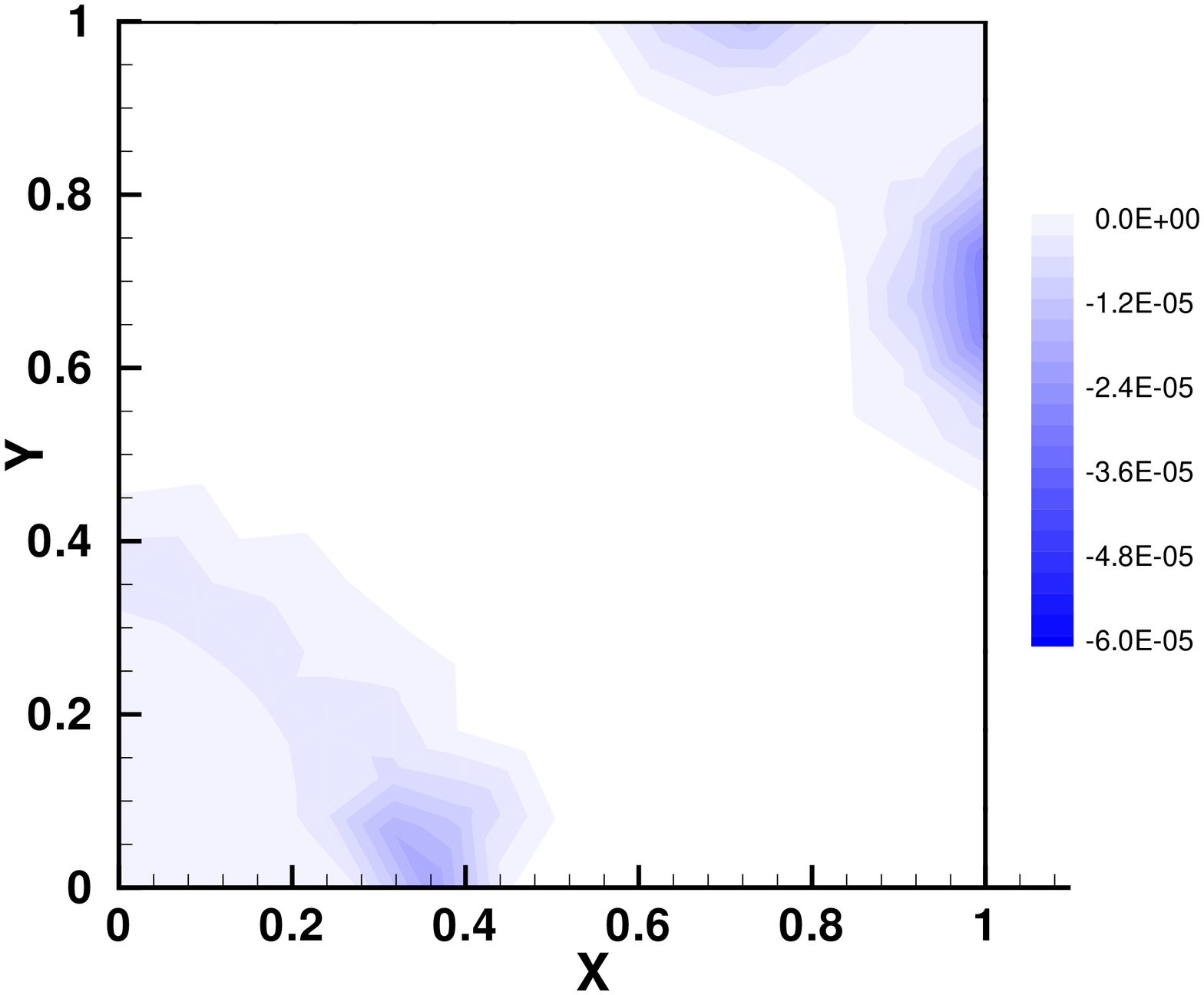}}
  \caption{Test problem \#1: Contours of mass balance error under the optimization-based RT0 formulation. In the 
    numerical simulations +45-degree (top), Delaunay (middle) and WCT (bottom) meshes are employed. 
    \label{Fig:Optim_RT0_mass_error_Problem_2}}
\end{figure}

\newpage
\begin{figure}[htbp]
  \subfigure{
    \includegraphics[scale=0.36]{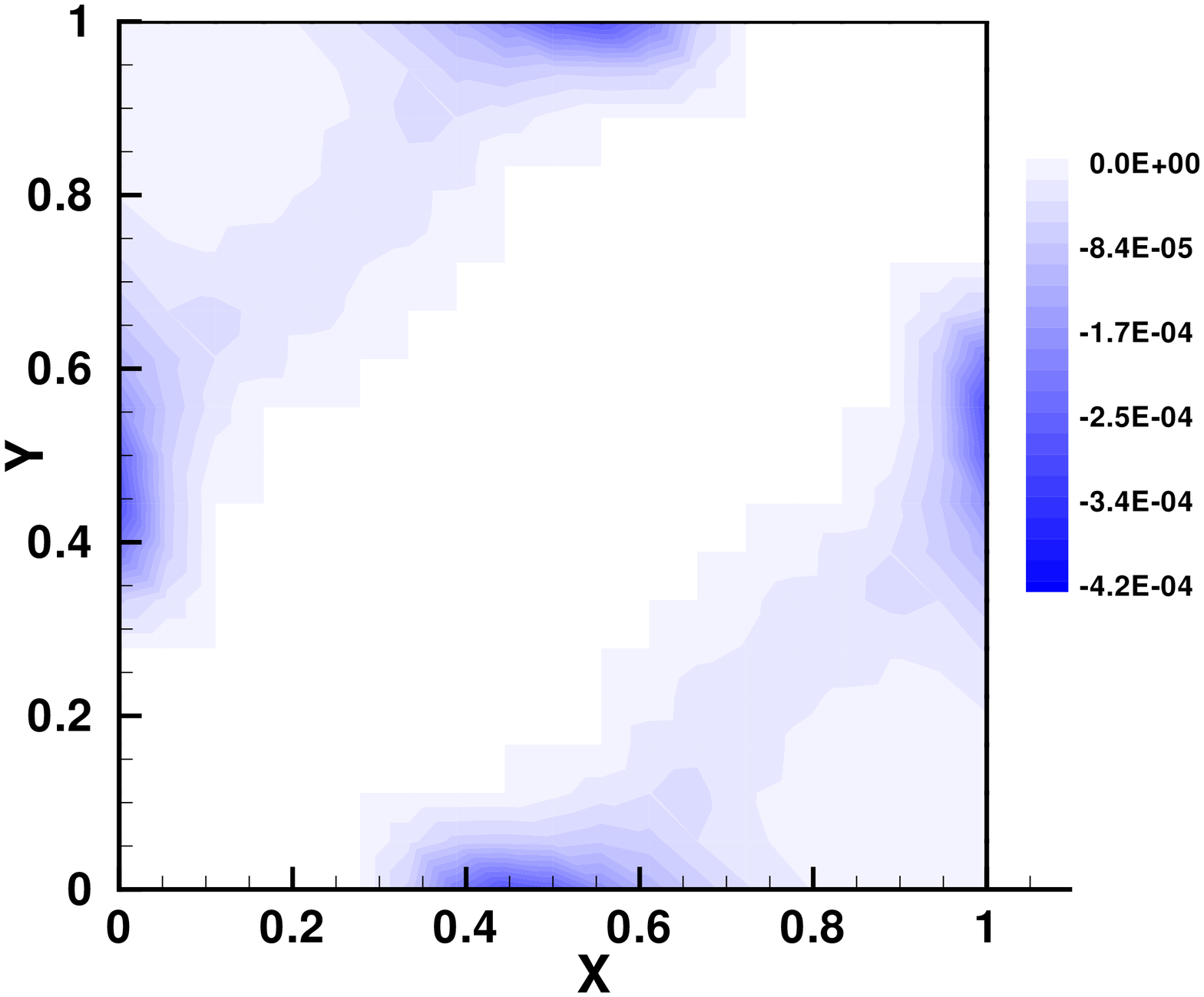}}
 \subfigure{
    \includegraphics[scale=0.36]{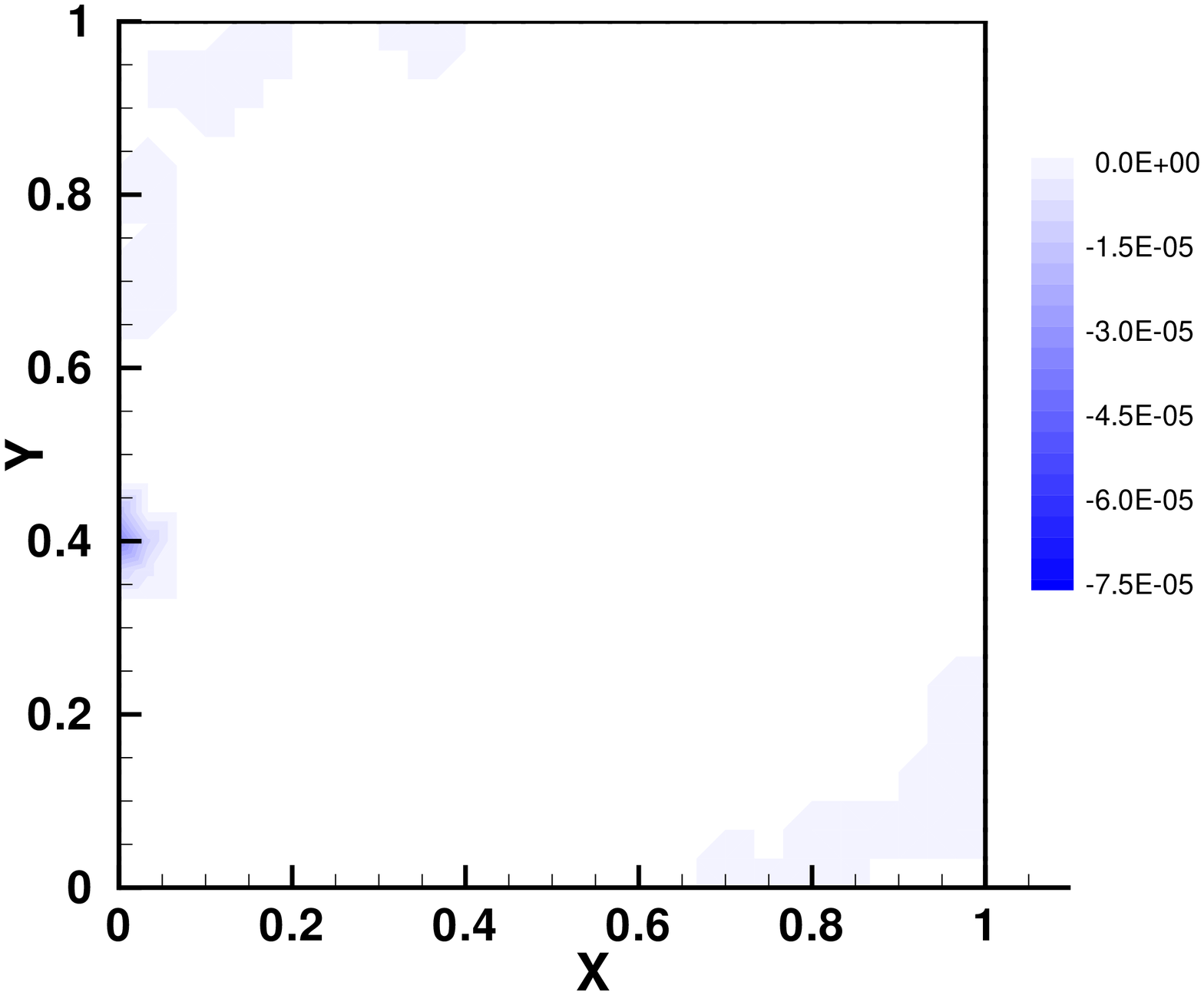}}
  \subfigure{
    \includegraphics[scale=0.36]{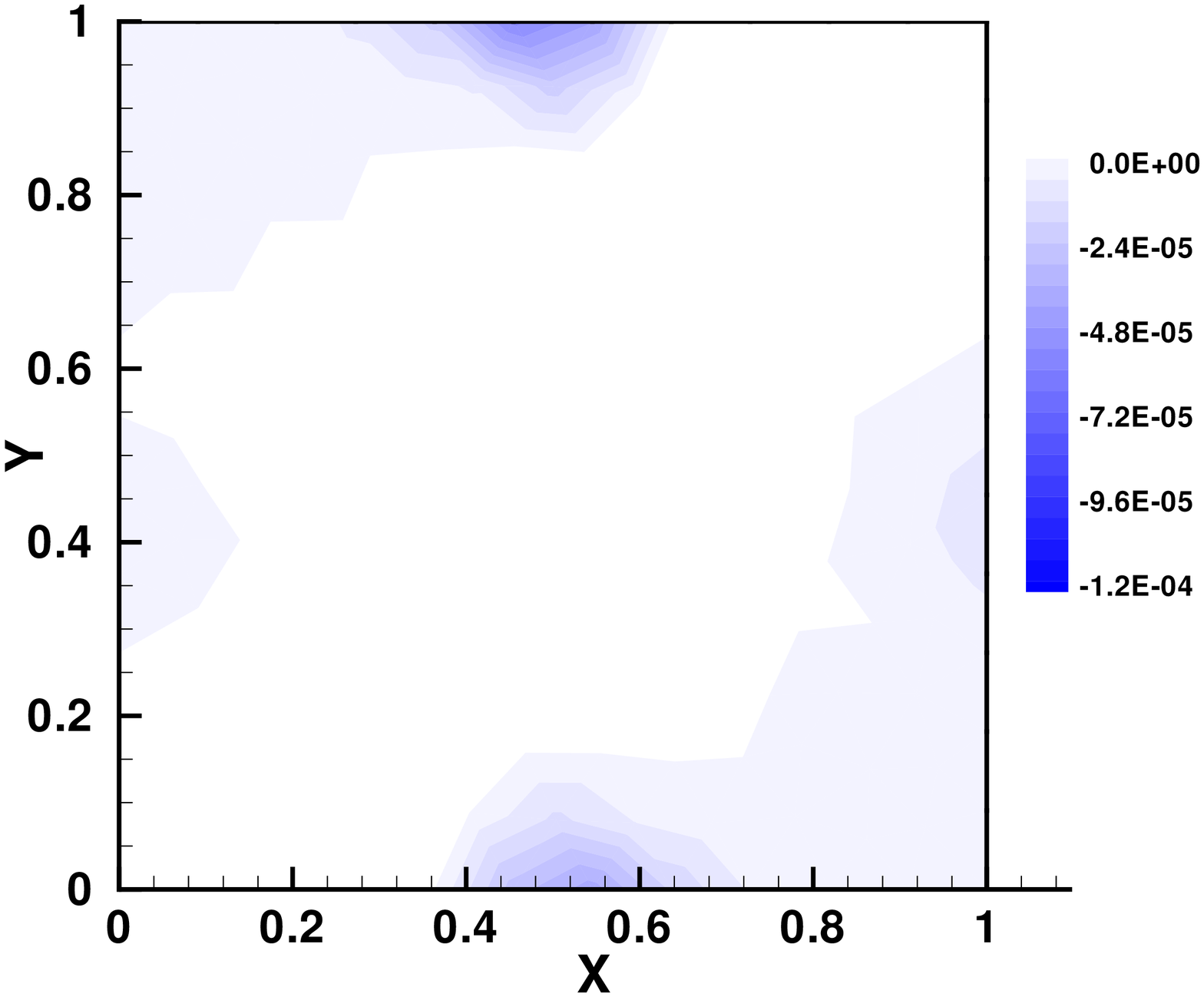}}
  \caption{Test problem \#2: Contours of mass balance error under the optimization-based RT0 formulation. In the 
    numerical simulations -45-degree (top), Delaunay (middle) and WCT (bottom) meshes are employed. 
    \label{Fig:Optim_RT0_mass_error_Problem_3}}
\end{figure}

\newpage
\begin{figure}[htbp]
  \includegraphics[scale=0.36]{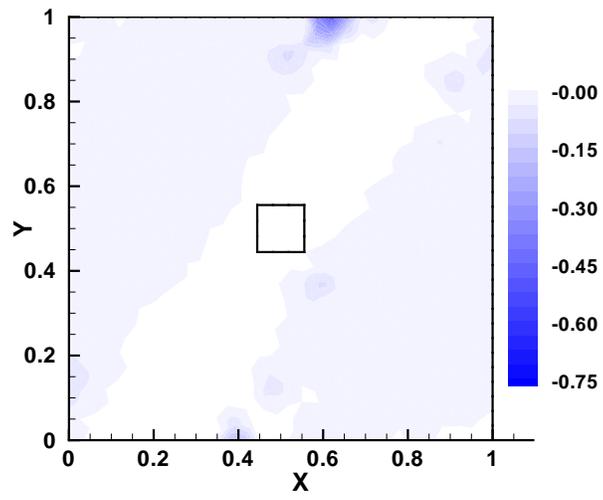}
  \caption{Test problem \#3: Contours of mass balance error under the optimization-based RT0 formulation. 
    \label{Fig:Optim_RT0_mass_error_Problem_4}}
\end{figure}

\newpage
\begin{figure}
\centering
\vspace{-0.5in}
\subfigure[]{\includegraphics[scale=0.4]{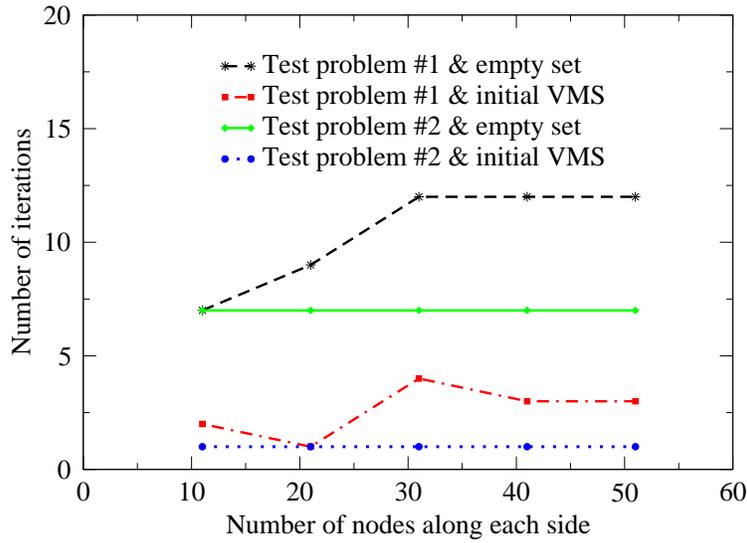}}
\vspace{0.5in}
\subfigure{\includegraphics[scale=0.4]{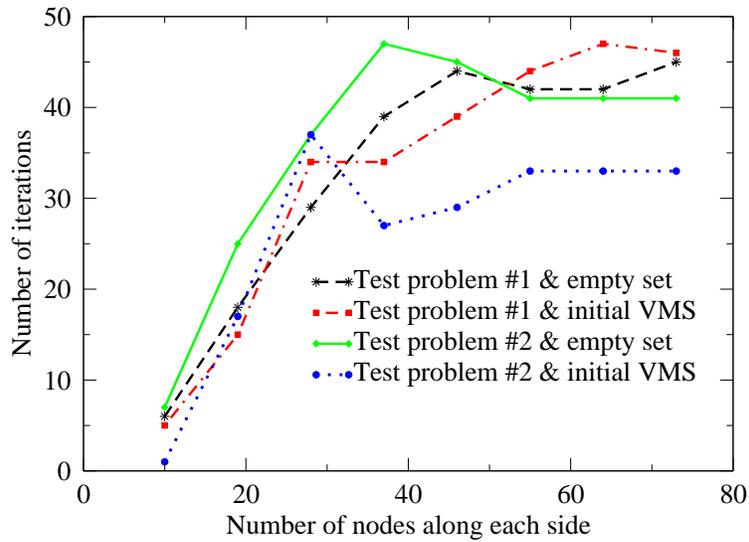}}
\vspace{-0.5in}
\caption{This figure compares number of iterations taken by the active set strategy for the 
\emph{VMS formulation} with respect to mesh refinement. We have employed four-node quadrilateral 
(top figure) and three-node triangular (bottom figure) meshes. Note that, in the case of three-node 
triangular meshes, we have used $+45$-degree and $-45$-degree meshes for test problems \#1 and \#2, 
respectively. For both the test problems two different sets are used as an initial guess in the 
active set strategy. \label{Fig:Optim_HVM_active_set_strategy}}
\end{figure}

\newpage
\begin{figure}
\centering
\vspace{0.75in}
\includegraphics[scale=0.4]{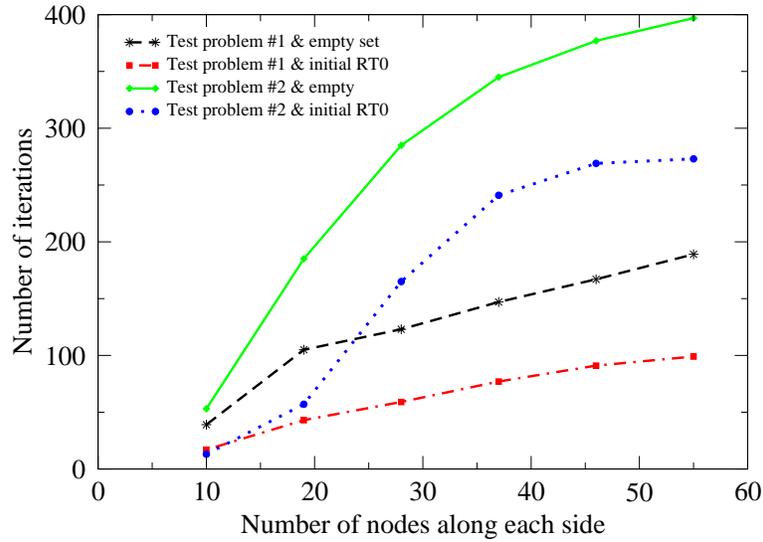}
\caption{This figure compares number of iterations taken by the active set strategy for the 
\emph{RT0 formulation} with respect to mesh refinement. We have employed $+45$-degree and 
$-45$-degree triangular meshes for test problems \#1 and \#2, respectively. For both the 
test problems two different sets are used as an initial guess in the active set strategy. 
\label{Fig:Optim_RT0_active_set_strategy}}
\end{figure}

\clearpage
\begin{figure}
\centering
\vspace{-0.5in}
\subfigure[]{\includegraphics[scale=0.4]{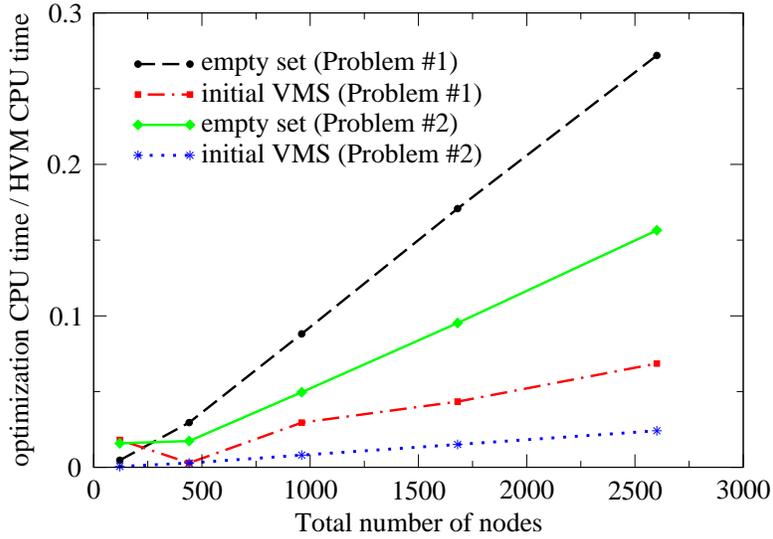}}
\vspace{0.5in}
\subfigure{\includegraphics[scale=0.4]{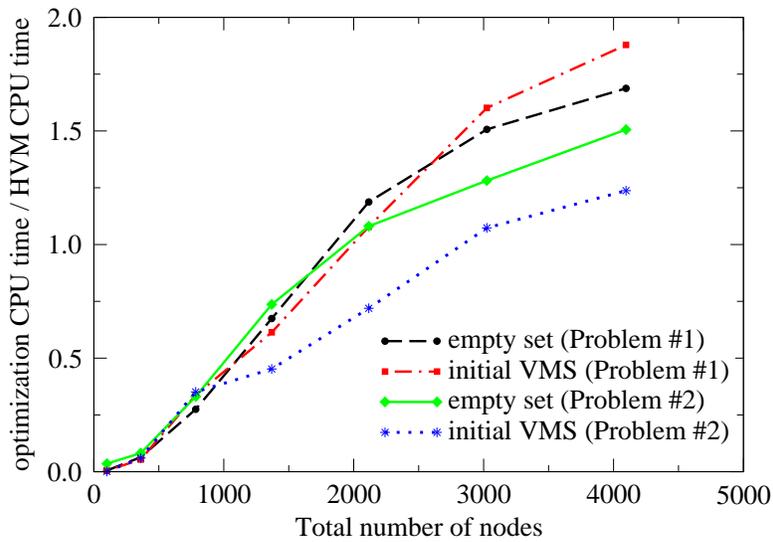}}
\vspace{-0.5in}
\caption{This figure compares the computational effort of the \emph{optimization-based VMS method} with 
respect to mesh refinement. We employed four-node quadrilateral (top figure) and $45$-degree triangular 
(bottom figure) meshes. On the $y$-axis we have the ratio of the additional CPU time taken by the 
optimization-based VMS method to the CPU time taken by the VMS method (which produces negative solutions). 
We considered test problem \#1 and \#2. Note that each node has three degrees-of-freedom. 
\label{Fig:Optim_HVM_CPU_time}}
\end{figure}

\newpage  
\begin{figure}
\centering
\includegraphics[scale=0.4]{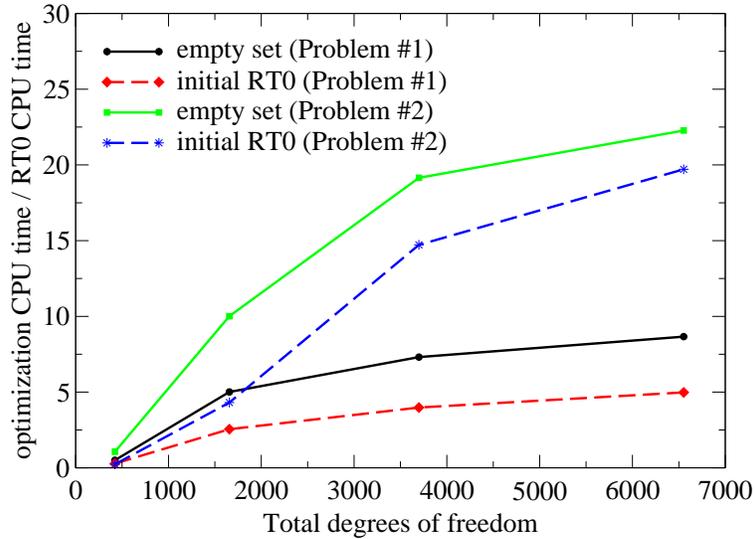}
\caption{This figure compares the computational effort of the \emph{optimization-based RT0 method} with 
respect to mesh refinement. On the $y$-axis we have the ratio of the additional CPU time taken 
by the optimization-based RT0 method to the CPU time taken by the RT0 method (which produced negative 
solutions). We employed $+45$-degree and $-45$-degree triangular meshes for test problems \#1 and \#2, 
respectively. Note that each node has three degrees-of-freedom. \label{Fig:Optim_RT0_CPU_time}}
\end{figure}

\newpage
\begin{figure}
\centering
\includegraphics[scale=0.4]{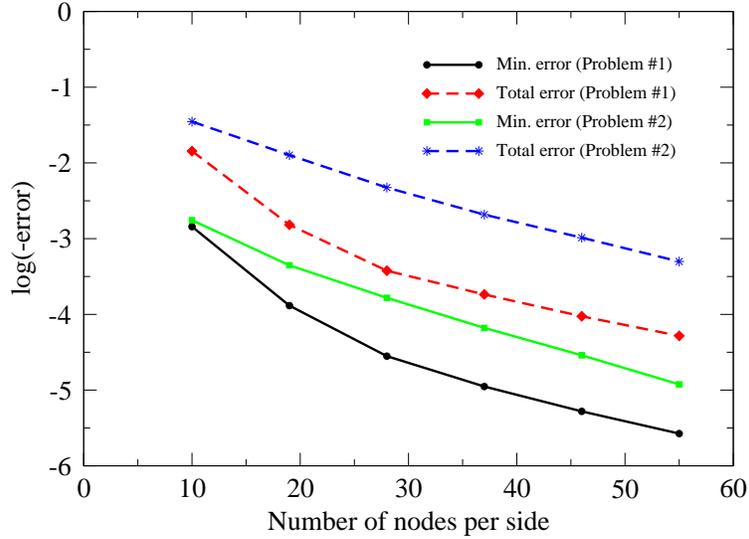}
\caption{This figure compares the error in the local mass balance with respect to mesh 
refinement using the RT0 formulation. We have used $+45$-degree and $-45$-degree triangular 
meshes for test problems \#1 and \#2, respectively. We have plotted both the maximum element 
sink strength, and the total error by summing the artificial sink strengths in all elements. 
From the above figure one can see that both these errors in the local mass balance decrease 
exponentially with respect to the element size. \label{Fig:Optim_RT0_min_total_errors}}
\end{figure}

\end{document}